\documentclass[a4paper, 10pt, english, reqno, dvipsnames, table]{amsart}
\usepackage[utf8]{inputenc}
\usepackage[OT2,T1]{fontenc}

\makeatletter
\renewcommand\paragraph{\@startsection{paragraph}{4}%
  \z@{.5\linespacing\@plus.7\linespacing}{-.5em}%
  {\normalfont\itshape}} 
\makeatother

\usepackage{amsmath}
\usepackage{amsthm}
\usepackage{amssymb}
\usepackage{amsbsy}
\usepackage{mathtools}
\usepackage{mathrsfs}
\usepackage{amsfonts}
\usepackage{bbold}

\usepackage{tikz}
\usepackage{tikz-cd}
\usepackage{quiver}
\usepackage{lipsum}

\usepackage{geometry}
\renewcommand{\arraystretch}{1.5}

\theoremstyle{plain}
\newtheorem{theorem}{Theorem}[section]
\newtheorem{proposition}[theorem]{Proposition}
\newtheorem{lemma}[theorem]{Lemma}
\newtheorem{corollary}[theorem]{Corollary}
\newtheorem{conjecture}[theorem]{Conjecture}
\newtheorem{definition-theorem}[theorem]{Definition-Theorem}
\newtheorem{definition-proposition}[theorem]{Definition-Proposition}

\theoremstyle{definition}

\newtheorem{definition}[theorem]{Definition}

\theoremstyle{remark}
\newtheorem{remark}[theorem]{Remark}

\usepackage{paralist}
\usepackage[figurewithin=none]{caption}
\usepackage{graphicx}
\usepackage{float}
\usepackage{xstring}
\usepackage{multirow}
\usepackage{array}
\newcolumntype{L}[1]{>{\raggedright\arraybackslash}p{#1}}
\usepackage{longtable}
\usepackage{booktabs}
\usepackage{changepage}

\usepackage{xcolor}

\DeclareMathOperator{\Gal}{Gal}

\DeclareSymbolFont{cyrletters}{OT2}{wncyr}{m}{n}
\DeclareMathSymbol{\Sha}{\mathalpha}{cyrletters}{"58}

\newcommand{\Zp}{\mathbb{Z}_p}
\newcommand{\Qp}{\mathbb{Q}_p}

\newcommand{\Frob}{\mathrm{Frob}}
\newcommand{\Hom}{\mathrm{Hom}}
\renewcommand{\Tilde}{\widetilde}

\DeclareMathOperator{\coker}{coker}
\DeclareMathOperator{\image}{im}

\DeclareMathOperator{\tr}{tr}
\DeclareMathOperator{\rank}{rank}
\newcommand{\rk}{\rank}
\DeclareMathOperator{\ord}{ord}

\DeclareMathOperator{\Reg}{Reg}

\AtBeginDocument{}

\newcommand{\Q}{\mathbb{Q}}

\newcommand{\Z}{\mathbb{Z}}

\numberwithin{equation}{section}
\makeatletter
\def\l@subsection{\@tocline{2}{0pt}{2pc}{5pc}{}}
\makeatother

\usepackage{url}
\usepackage[unicode, colorlinks, linktocpage=true, backref=page]{hyperref}
\hypersetup{
    linkcolor = MidnightBlue,
    citecolor = ForestGreen,
    urlcolor = YellowOrange,
    bookmarksnumbered = true
}

\title[A $p$-part BS-D formula over CM extensions]{A $p$-part Birch and Swinnerton-Dyer Formula over Totally Imaginary Quadratic Extension of Totally Real Fields }
\author{Haidong Li}
\address[Haidong Li]{Jiangsu Police Institute. No.48 Shifo Sangong, Pukou District, Nanjing, Jiangsu, 210031, China.}
\email{lihaidong@jspi.cn}

\date{August 12, 2026}

\begin{document}

\begin{abstract}
This article studies a modular semistable elliptic curve $E$ over a totally real number field $F$ such that, upon base change to a totally imaginary quadratic extension $K$, it has analytic rank one. Assuming the Iwasawa main conjecture, along with a substantial number of assumptions, we prove a variant of the $p$-part of   the Birch and Swinnerton-Dyer formula over $K$, where $p$ is an odd prime. More precisely, up to a $p$-adic unit, we have
\[
\frac{L'(E/K,1)}{\Omega^{\mathrm{cong}}_{\mathbf{f}} \operatorname{Reg}(E/K)}
= \#\Sha(E/K)[p^\infty]\prod_{u} c_u(E/K),
\]
where $\Omega^{\mathrm{cong}}_{\mathbf{f}}$ is the congruence period of the Hilbert modular form $\mathbf{f}$ associated to $E$ via the modularity conjecture.
\end{abstract}

\maketitle
\tableofcontents

\section{Introduction}

\subsection{The Birch and Swinnerton-Dyer Conjecture}

Let $F$ be a number field and let $E/F$ be an elliptic curve. By Mordell--Weil theorem \cite{mordell1922rational}, the group $E(F)$ of $F$-rational points is finitely generated. The Birch and Swinnerton-Dyer conjecture predicts a precise relation between the algebraic invariants of $E(F)$ and the analytic invariants of the Hasse--Weil $L$-function $L(E/F,s)$.

\begin{conjecture}[The Birch and Swinnerton-Dyer]
\label{conj:bsd}
Let $E/F$ be an elliptic curve over a number field $F$.
\begin{enumerate}
\item[(a)] The Hasse--Weil $L$-function $L(E/F,s)$ has analytic continuation to the whole complex plane, and
\[
\ord_{s=1}L(E/F,s)=\rk_{\Z}E(F).
\]
\item[(b)] The Shafarevich--Tate group $\Sha(E/F)$ is finite, and
\[
\frac{L^{(r)}(E/F,1)}
{r!\,\Omega_{E/F}\,\Reg(E/F)\,|\Delta_F|^{-1/2}}
=
\frac{\#\Sha(E/F)}{(\#E(F)_{\mathrm{tor}})^2}\prod_{\ell}c_\ell(E/F),
\]
where $r=\ord_{s=1}L(E/F,s)$, $\Omega_{E/F}$ is the integral of a N\'eron differential, $\Delta_F$ is the absolute discriminant of $F$, $\Reg(E/F)$ is the regulator of the N\'eron-Tate height pairing on $E(F)$, and $c_\ell(E/F)$ are the local Tamagawa numbers for finite place $\ell$.
\end{enumerate}
\end{conjecture}

Part (a) is the rank part conjecture. Part (b) is the Birch and Swinnerton-Dyer formula conjecture. This version is due to Tate \cite{tate66}. Tate's comment in \cite[pp.~198]{tate1974arithmetic} is accurate for that time: This remarkable conjecture relates the behavior of a function $L$ at a point where it is not at present known to be defined to the order of a group $\Sha$ which is not known to be finite!

\subsubsection{History}
The numerical experiments of Birch and Swinnerton-Dyer \cite{MR146143,MR179168} led to the conjecture that, for an elliptic curve of rank $r$,
\[
\prod_{p\le x}\frac{N_p}{p}\approx C(\log x)^r
\qquad(x\to\infty),
\]
where $N_p$ is the number of points of $E$ modulo $p$ and $C$ is a constant. Goldfeld \cite{MR679556} proved that this asymptotic formula implies the rank conjecture. Deuring's early work \cite{MR5125} was the only evidence coming from $L$-functions at the time.

For CM elliptic curves, Coates and Wiles \cite{coates1977conjecture} proved that $L(E/F,1)\ne0$ implies $E(F)$ finite for $F=K$ or $F=\Q$, where $K/\Q$ is an imaginary quadratic extension. Arthaud \cite{arthaud1978birch} generalized this to abelian extensions, and Rubin \cite{MR903383} proved finiteness of the Shafarevich--Tate group in this case. Rubin \cite{rubin1991main} proved the $p$-part of the Birch and Swinnerton-Dyer formula for CM elliptic curves of analytic rank zero up to powers of $2$ and $3$, and for analytic rank one under good ordinary reduction. Pollack and Rubin \cite{MR2052361} and Kobayashi \cite{MR3020170} treated the supersingular cases.

For elliptic curves without CM, Gross and Zagier \cite{gross1986heegner} proved that $\ord_{s=1}L(E/K,s)=1$ implies the existence of a point of infinite order in $E(K)$. Kolyvagin \cite{kolyvagin1989finiteness} completed the rank conjecture for $r\le1$ over $\Q$. The rank zero formula was obtained by Kato \cite{MR1338860}, Skinner--Urban \cite{MR3148103}, and Burungale--Skinner--Tian--Wan \cite{burungale2024zetaelementsellipticcurves}. The rank one formula for non-CM curves over $\Q$ is due to W.~Zhang \cite{MR3295917} and Jetchev--Skinner--Wan \cite{jetchev5birch}, with Castella \cite{castella2018p} treating multiplicative reduction. Fouquet--Wan \cite{fouquet2021iwasawa} prove Kato's cyclotomic Iwasawa main conjecture for modular forms of arbitrary reduction type at $p$, which feeds the rank zero formula. For the $2$-part, see Cai--Li--Zhai \cite{MR4093972}.

These results are collected in Table~\ref{tab:history}. The precise technical assumptions of each result are recorded in the cited papers.

\begin{table}[H]
\centering
\caption{History of the Birch and Swinnerton-Dyer conjecture}
\label{tab:history}
\setlength{\tabcolsep}{3pt}
\footnotesize
\begin{tabular}{L{4cm}L{1.4cm}L{1.1cm}L{1.4cm}L{5.3cm}}
\toprule
\textbf{Reference} & \textbf{Field} & \textbf{Rank} & \textbf{CM} & \textbf{Content} \\
\midrule
Coates--Wiles \cite{coates1977conjecture} & $K$ or $\Q$ & $0$ & yes &
$L(E,1)\ne0\Rightarrow E$ finite \\
Gross--Zagier \cite{gross1986heegner} & $K$ & $1$ & not needed &
$\ord_{s=1}L(E/K,s)=1\Rightarrow E(K)$ infinite \\
Kolyvagin \cite{kolyvagin1989finiteness} & $\Q$ & $0,1$ & not needed &
rank conjecture for $r\le1$ \\
Rubin \cite{rubin1991main}, Pollack--Rubin \cite{MR2052361}, Kobayashi \cite{MR3020170} & $\Q$ & $0,1$ & yes &
$p$-part of the formula for $p>3$ \\
Kato \cite{MR1338860}, Skinner--Urban \cite{MR3148103}, BSTW \cite{burungale2024zetaelementsellipticcurves}  & $\Q$ & $0$ & no &
$p$-part of the formula for $p>2$ \\
Zhang \cite{MR3295917}, Jetchev--Skinner--Wan \cite{jetchev5birch} & $\Q$ & $1$ & no &
$p$-part of the formula for $p>2$ \\
\bottomrule
\end{tabular}
\end{table}

\begin{remark}
In the last two rows, ``$p>2$'' hides technical conditions: Skinner--Urban require good ordinary reduction. Zhang requires $p\ge5$, good ordinary reduction, and a ramification condition at the primes $\ell\parallel N$ with $\ell\equiv\pm1\bmod p$, and Jetchev--Skinner--Wan require $p\ge5$ (with $p=3$ allowed when $a_p(E)=0$ in the supersingular case) and irreducibility of $E[p]$.
\end{remark}

\subsubsection{The present article}
The results above are all over $\Q$ or over imaginary quadratic fields. Over a general totally real field $F$, much less is known. This article proves, up to a $p$-adic unit, a $p$-part of a variant of the Birch and Swinnerton-Dyer formula for a semistable modular elliptic curve $E/F$ over a totally imaginary quadratic extension $K/F$ of analytic rank one after base changing to $K$, under an anticyclotomic Iwasawa main conjecture and many technical assumptions.  This is a variant precisely because it normalizes by the congruence period $\Omega_{\mathbf f}^{\mathrm{cong}}$ (which will be defined in Definition~\ref{def:conperiod}) of the Hilbert newform $\mathbf f$ attached to $E$, whereas the classical Birch and Swinnerton-Dyer formula uses the real or complex periods of $E$.

\subsection{The Main Theorem}\label{sec:mainresult}

The main theorem of this article is stated as follows. Since it has a lot of assumptions, the statement is quite long.

\begin{theorem}[Main Theorem]\label{theorem:maintheorem}
Let $E$ be a semistable elliptic curve defined over a totally real field $F$ of degree $d=[F:\Q]$, whose conductor $\mathfrak{N}$ is a squarefree ideal of $\mathcal{O}_F$. Assume that this elliptic curve is modular, i.e.\ $E$ is associated to a Hilbert newform $\mathbf{f}$ of parallel weight $2$. Let $K/F$ be a quadratic totally imaginary extension with relative discriminant $D\subseteq\mathcal{O}_F$ and absolute discriminant $D_K$.

These data satisfy the following \textbf{basic assumptions}:

\begin{itemize}
\item (Prime and prime place) Fix a prime $p>2$, and assume that $p$ is unramified in $K$. Suppose there exists  a prime place $w|p$ of $F$, and assume that $w=v\bar{v}$ splits in $K$ such that $F_w\cong\mathbb{Q}_p$.

\item (Reduction type) $E$ has good ordinary reduction at all prime places above $p$.

\item (Irreducibility of the representation) Assume that $E[p]$ is irreducible as a representation of $G_K$.

\item (Analytic rank one) The analytic rank of $E/K$ is one.

\item (Heegner hypothesis) Let $\mathfrak{N}=\mathfrak{MD}$ be a decomposition into ideals, and suppose that the number of prime divisors of $\mathfrak{D}$ has parity different from $d$. Assume moreover that $u|\mathfrak{M}$ if and only if $u$ splits in $K$, and that $u|\mathfrak{D}$ if and only if $u$ is inert in $K$. This condition implies $(\mathfrak{N},D)=1$.
  \end{itemize}

We further make the following \textbf{technical assumptions}:
  \begin{itemize}
\item (Technical assumption--modularity-1) When $d$ is even, there exists a prime ideal $\mathfrak{q}|\mathfrak{D}$ with $p\nmid\mathrm{Norm}(\mathfrak{q})+1$ such that, up to a $p$-adic unit,
    $$
    \eta_{\mathbf{f}}(\mathfrak{MD},1)=
    \eta_{\mathbf{f}}(\mathfrak{MD}/\mathfrak{q},\mathfrak{q})
    \times c_{\mathfrak{q}}(E/K).
    $$

Here $\eta_{\mathbf{f}}(\cdot,\cdot)$ is the congruence number introduced in Definition~\ref{definition:congruencenumber}, and $c_{\mathfrak{q}}$ is the Tamagawa number.
\item (Technical assumption--modularity-2) If $u$ is a prime place of $F$ which ramifies in $E[p]$ and $\mathrm{Norm}(u)\equiv -1 \bmod p$, then either $E[p]|_{I_u}$ is irreducible or $E[p]|_{G_u}$ is absolutely reducible. Here $G_u$ and $I_u$ are the local Galois group and the inertia group at $u$, respectively.

\item (Technical assumption--modularity-3) $E[p]|_{G_{F(\zeta_p)}}$ is absolutely irreducible.

\item (Technical assumption--Iwasawa theory-1) The Iwasawa main conjecture~\ref{conj:iwasawamain} holds.
  \end{itemize}

Then, up to a $p$-adic unit, one has the equality
  $$
  \frac{L'(E/K,1)}{\Omega^{\mathrm{cong}}_\mathbf{f}
  \mathrm{Reg}(E/K)}
  =\#\Sha(E/K)[p^\infty]\prod_{u} c_u(E/K),
  $$
where $$\Omega_{\mathbf{f}}^{\mathrm{cong}}:=\frac{(8\pi^2)^d(\mathbf{f},\mathbf{f})_{U_0(\mathfrak{N})}}  {\eta_{\mathbf{f}}}$$ is the congruence period.

That is, the $p$-part of the (variant) Birch and Swinnerton-Dyer formula for the elliptic curve $E$ over $K$ holds.
\end{theorem}

We now explain the assumptions in the theorem. The theorem distinguishes \textbf{basic assumptions} and \textbf{technical assumptions}: the basic assumptions are the fundamental setting, and changing them would mean a considerable adjustment of the methods used in this article. The technical assumptions, on the other hand, are assumptions which one hopes can be replaced by newer ones.

First consider the \textbf{basic assumptions}:

\begin{itemize}
\item (Analytic rank one), (Reduction type), and (Irreducibility of the representation) are standard assumptions, related to the basic setting of this article. Note that, in this article, we do not require $E/F$ to have analytic rank one, since the main theorem is a formula over $K$, although our ultimate goal is to deal with the totally real case.

\item The unramified condition in (prime and prime place) is standard, since most tools behave well under the unramified condition. The subtle point is the condition $F_w\cong\mathbb{Q}_p$. This condition makes $K_v\cong\mathbb{Q}_p$, so that $E(K_v)$ has a finite-index subgroup of $\mathbb{Z}_p$-rank $1$. Therefore, it is meaningful to discuss the index (with $\mathbb{Z}_p$-coefficients) of a non-torsion point in $E(K_v)$. This is also a key step in the derivation of the key formula~\ref{prop:keyformula} of this article.
\end{itemize}

\begin{remark}
We remarked that in (Prime and prime place), we can assume that $p$ is unramified in $F$, which suffices for most techniques used in this article. This additional assumption, that $p$ is unramified in $K$,  allows us to avoid any delicate discussion involving $\sqrt{|D_K|}$, since it guarantees that$\sqrt{|D_K|}$ is a $p$-adic unit. Such a discussion would otherwise be rather subtle, as $\sqrt{|D_K|}$ is closely tied to periods and $p$-adic interpolation in various conventions found in the literature.
\end{remark}

Next consider the \textbf{technical assumptions}:

\begin{itemize}
\item (Technical assumption--modularity-1) is needed in the proof of the main theorem~\ref{Theorem:grosszaigerimaintheorem} in Section~\ref{sec:GZ}. There, what must be compared is the relation between the congruence number $\eta_{\mathbf{f}}(\mathfrak{MD},1)$ and the Shimura degree $\delta(\mathfrak{M},\mathfrak{D})$. The strategy adopted in this article is: when $d$ is odd, first prove that the congruence number $\eta_{\mathbf{f}}(\mathfrak{MD},1)$ equals the Shimura degree $\delta(\mathfrak{MD},1)$ up to a $p$-adic unit, then, in the second step, using results of Ribet--Takahashi type, compare the two Shimura degrees $\delta(\mathfrak{MD},1)$ and $\delta(\mathfrak{M}\frac{\mathfrak{D}}{\mathfrak{ab}},\mathfrak{ab})$ to obtain the Tamagawa numbers at $\mathfrak{a}$ and $\mathfrak{b}$. Repeating the use of the Ribet--Takahashi results, one obtains all the Tamagawa numbers on $\mathfrak{D}$. But when $d$ is even, $\delta(\mathfrak{MD},1)$ is meaningless, since then the second place should be an ideal with an odd number of prime factors. Hence the strategy at that point relies on this assumption, modifying the first step to compare $\eta_{\mathbf{f}}(\mathfrak{MD}/\mathfrak{q},\mathfrak{q})$ with $\delta(\mathfrak{MD}/\mathfrak{q},\mathfrak{q})$, and then in the second step repeatedly use Ribet--Takahashi. The condition $p\nmid\mathrm{Norm}(\mathfrak{q})+1$ comes from Proposition~\ref{prop:congruencedividesmodular} in the first step, where Manning's result \cite{manning2021patching} is used.

\item (Technical assumption--modularity-2, 3) come from the assumptions in Manning's article \cite{manning2021patching}, which are used in this article in Proposition~\ref{prop:congruencedividesmodular} of Section~\ref{sec:GZ}.

\item (Technical assumption--Iwasawa theory-1): the Iwasawa main conjecture is an essential tool for proving the Birch and Swinnerton-Dyer formula, but this article only assumes it without proof. This will be the focus of future work, and it is hoped that it can be proved and the assumption removed.
\end{itemize}

\begin{remark}
    The (Technical assumption-modularity-1,2,3) are used in the proof of Theorem~\ref{Theorem:grosszaigerimaintheorem} to derive an explicit Gross–Zagier formula. In particular, we rely on these assumptions to compute the Tamagawa numbers at primes dividing $\mathfrak{D}$. If one could obtain another explicit Gross–Zagier formula, together with the Tamagawa numbers, under a weaker set of assumptions, then the present technical assumptions can be replaced by that weaker set.
\end{remark}

\subsection{Organization}

Section~\ref{sec:preliminaries} fixes notation and states the standing assumptions. Section~\ref{sec:selmer} develops the control theorem and the key formula. Section~\ref{sec:GZ} proves an explicit Gross--Zagier formula with the congruence period. Section~\ref{sec:LZZ} states the Liu--Zhang--Zhang formula and its trivial-character consequence. Section~\ref{sec:BSDoverK} states and proves the main theorem. Appendix~\ref{app:conventions} records the convention tables of Yuan--Zhang--Zhang, Cai--Shu--Tian, and Liu--Zhang--Zhang.

\subsection*{Acknowledgements}
This article is based on the author's Ph.D. thesis, which was defended in May 2025 at the Academy of Mathematics and Systems Sciences, Chinese Academy of Sciences. The author deeply thanks his supervisor, Professor Xin Wan, for introducing him to this field, for his continuous guidance and support throughout the writing, and for his prior work and explanations, which greatly saved the author from unnecessary detours. He also gratefully acknowledges the Academy for providing an excellent study environment and opportunities.

During the research and writing of this article, many people offered invaluable help. The author thanks Li Cai for clarifying questions on the explicit Gross--Zagier formula, Chan-Ho Kim and Wei Zhang for drawing attention to relevant literature, Yichao Tian for providing references and guidance on Shimura degrees, and Yifeng Liu and Yangyu Fan for explanations regarding the \(p\)-adic Waldspurger formula. The author is also grateful to the reviewers for their careful reading and constructive feedback during the defense.

The author has also benefited considerably from discussions with his peers. He thanks Hang Yin, Ruichen Xu, Ruocheng Zhao, Jie Xu, Haijun Jia, and Zerui Xiang, who attentively attended seminars on earlier drafts and provided many valuable comments.

\section{Preliminaries}\label{sec:preliminaries}
This section introduces the preliminary knowledge and the basic assumptions needed throughout the article, and is divided into two subsections.

Subsection~\ref{section:basicnotions} focuses on the basic objects studied in this article, mainly Hilbert modular forms and elliptic curves, and introduces the most basic definitions and properties. Since different sections of this article focus on different knowledge, preliminary knowledge that is needed by a section but not by the whole article will be introduced in the corresponding section, or even in the corresponding subsection or paragraph.

Subsection~\ref{section:basicassumption} is divided into two parts. In the first part, we introduce the basic assumptions and notation used throughout the article. Then, in the second part, we give a sketch of the proof of the main theorem.

\subsection{Basic Notions}\label{section:basicnotions}

\subsubsection{Hilbert modular forms}
This subsubsection mainly introduces Hilbert modular forms, the Hecke operators of Hilbert modular forms, and newforms. Below we mainly consider Hilbert modular forms of weight $k$, level $\mathfrak{N}$, and with trivial character. We mainly follow the setting of Shimura \cite{Shimura1978}.

\paragraph{Definition and basic properties of Hilbert modular
forms} Let $F/\mathbb{Q}$ be a totally real field of degree $d$. Write $\mathcal{H}$ for the classical upper half plane, and let $\mathcal{H}^d$ be the product of $d$ copies of the upper half plane.

Consider the action of $GL_2^+(\mathbb{R})^d$ on $\mathcal{H}^d$: an element $\alpha=(\alpha_i)_i\in GL_2^+(\mathbb{R})^d$, where $\alpha_i=\left(\begin{array}{ll} a_i & b_i \\ c_i & d_i
\end{array}\right)$, acts on $z=(z_i)_i\in\mathcal{H}^d$ by fractional
linear transformations, i.e.
$$
\alpha(z)=\left(\frac{a_iz_i+b_i}{c_iz_i+d_i}\right)_i.
$$

For $k=(k_i)_i\in\mathbb{Z}_{\ge 0}^d$ and $z=(z_i)_i\in\mathbb{C}^d$, we set
$$
z^k=\prod_{1}^d z_i^{k_i}, \quad \{k\}=\sum_1^d k_i, \quad
\{z\}=\sum_1^d z_i
$$
and
$$
e(z)=\exp(2\pi i\sum_1^d z_i).
$$

Consider a complex function $f:\mathcal{H}^d\rightarrow\mathbb{C}$ on the upper half plane. For $k\in\mathbb{Z}_{\ge 0}^d$ and $\alpha=\left(\begin{array}{ll} \ast & \ast \\ c & d
\end{array}\right)\in GL_2^+(\mathbb{R})^d$, define
$$
(f|_k\alpha)(z)=(cz+d)^{-k}f(\alpha(z)), \quad
(f||_k\alpha)(z)=\det(\alpha)^{k/2}f|_k.
$$

Consider the narrow class group $F^\times\backslash\mathbb{A}^\times/F_{\infty,+}^\times\prod {\mathcal{O}_F}_{u}^\times$, of order $h$, which we call the narrow class number. Choose $h$ elements $\{t_\nu\}_{\nu=1}^h$ of $\mathbb{A}$ whose archimedean places are $1$, so that $\{t_\nu\}_{\nu=1}^h$ is a set of representatives of the narrow class group. Write $x_\nu=\begin{pmatrix} 1& \\ & t_\nu
\end{pmatrix}$ and
$x_\nu^{tp}=\begin{pmatrix} t_\nu & \\ & 1
\end{pmatrix}$.

Fix an integral ideal $\mathfrak{N}$ of ${\mathcal{O}_F}$. For any non-archimedean place $u$, define the subgroup $K_u(\mathfrak{N})\le GL_2(F_u)$ of $GL_2(F_u)$:
\[
K_u(\mathfrak{N})=\Bigl\{ x:\ a\mathcal{O}_{F_u}+
\mathfrak{N}_u = \mathcal{O}_{F_u},\ b\in \delta_{F,u}^{-1},\\
c\in \mathfrak{N}_u\mathcal{O}_{F_u},\ d\in \mathcal{O}_{F_u},\
ad-bc\in \mathcal{O}_{F_u}^\times \Bigr\},
\]
where $x=\begin{pmatrix} a& b\\ c&d
\end{pmatrix}$, and $\mathfrak{N}_u$ and $\delta_{F,u}$ are the
$u$-parts of the integral ideal $\mathfrak{N}$ and of the different ideal $\delta_F$, respectively.

Define
$$
K_0(\mathfrak{N}):=\prod_{u<\infty}K_{u}(\mathfrak{N}).
$$

Using the above notation, one obtains the standard decomposition
$$
GL_2(\mathbb{A})=\cup_{\nu=1}^h
Gl_2(F)x_{\nu}^{-tp}\left(GL_2^+(\mathbb{R})^{d}K_0(\mathfrak{N})\right).
$$

For $1\le \nu\le h$, define the congruence subgroup $\Gamma_\nu(\mathfrak{N})$ of $GL_2(F)$,
$$
\Gamma_\nu(\mathfrak{N}):=\left \{
\begin{pmatrix}
a& t_\nu b \\
t_\nu c&d
\end{pmatrix} :
a\in {\mathcal{O}_F}, b\in \delta_F^{-1}, c\in \mathfrak{N}\delta_F,
d \in {\mathcal{O}_F}, ad-bc\in {\mathcal{O}_F}^\times
\right \}
$$

\begin{definition}[$\nu$-Hilbert modular form]
Fix a $k\in\mathbb{Z}_{\ge 0}^d$. A holomorphic complex function $f_\nu$ on $\mathcal{H}^d$ is called a Hilbert modular form of weight $k$ and level $\Gamma_\nu(\mathfrak{N})$ if for every $\alpha\in\Gamma_\nu(\mathfrak{N})$ one has
$$
(f_\nu||_k\alpha)(z)=f_\nu(z).
$$

The space of all Hilbert modular forms of weight $k$ and level $\Gamma_\nu(\mathfrak{N})$ is denoted $M_k(\Gamma_\nu(\mathfrak{N}))$.
\end{definition}

Similarly to the classical theory of modular forms over $\mathbb{Q}$, a Hilbert modular form $f_\nu$ of weight $k$ and level $\Gamma_\nu(\mathfrak{N})$ also has a Fourier expansion, namely
$$
f_\nu(z)=\sum_{\xi}a_\nu(\xi)e(\xi z),
$$
where $\xi$ runs through the totally positive elements of $t_\nu{\mathcal{O}_F}$ and $0$. If, for all $\alpha\in GL_2^+(F)$, the constant term of the Fourier expansion of $f_\nu||_k\alpha$ is $0$, then $f_\nu$ is called a cusp form. The space of cusp forms is denoted $S_k(\Gamma_\nu(\mathfrak{N}))$.

\begin{definition}[Hilbert modular form]
Fix a $k\in\mathbb{Z}_{\ge 0}^d$. A Hilbert modular form of weight $k$ and level $\mathfrak{N}$ is defined as a function on $GL_2(\mathbb{A})$ of the form $\mathbf{f}=(f_1,\cdots,f_h)$ with $f_\nu\in M_k(\Gamma_\nu(\mathfrak{N}))$.

The function is defined as follows: by the decomposition $GL_2(\mathbb{A})=\cup_{\nu=1}^h GL_2(F)x_{\nu}^{-tp}\left(GL_2^+(\mathbb{R})^{d}K_0(\mathfrak{N})\right)$ above, for $g\in GL_2(\mathbb{A})$, write it as $\gamma x_\nu g_\infty k_f$, then
$$
\mathbf{f}(\gamma x_{\nu}^{-\iota} g_\infty k_f):=(f_\nu||_k g_\infty)(\mathbf{i}),
$$
where $\mathbf{i}=(\sqrt{-1},\sqrt{-1},\cdots,\sqrt{-1})$.

The space of Hilbert modular forms of weight $k$ and level $\mathfrak{N}$ is denoted $\mathbf{M}_k(\mathfrak{N})$. If all the $f_\nu$ are cusp forms, then $\mathbf{f}$ is called a cusp form, and the space of cusp forms is denoted $\mathbf{S}_k(\mathfrak{N})$.
\end{definition}

For any other integral ideal $\mathfrak{m}$ of $F$, there exist $\nu\in\{1,\cdots,h\}$ and a totally positive element $\xi\in F$ such that $\mathfrak{m}=\xi t_\nu^{-1}{\mathcal{O}_F}$. Write
$$
C(\mathfrak{m},\mathbf{f})=a_\nu(\xi)\xi^{-k/2}N(\mathfrak{m})^{k_0/2},
$$
where $k_0=\max_j\{k_j\}$. This is a well-defined notion, independent of the choice of representative. We set $C(\mathfrak{m},\mathbf{f})=0$ when $\mathfrak{m}$ is not an integral ideal.

\paragraph{Hecke operators}

Write $K=GL_2^+(\mathbb{R})^d K_0(\mathfrak{N})$. Consider
\[
Y_u:=\Bigl\{ \begin{pmatrix}
a&b \\
c&d
\end{pmatrix}\in GL_2(F_u):\ a\mathcal{O}_{F_u}+\mathfrak{N}_u=\mathcal{O}_{F_u},
b\in \delta_{F,u}^{-1},\ c\in \mathfrak{N}_u\delta_{F,u},\
d\in \mathcal{O}_{F_u} \Bigr\}
\]
and $Y=(GL_2(\mathbb{R})^d\prod_u Y_u)\cup GL_2(\mathbb{A})$.

\begin{definition}
For an integral ideal $\mathfrak{m}$, define the Hecke operator
$$
T_{\mathfrak{m}}:=\sum_y KyK,
$$
where $y$ runs through all $y\in Y$ with $(\det y){\mathcal{O}_F}=\mathfrak{m}$.
\end{definition}

Let $\mathbf{f}$ be a Hilbert modular form of weight $k$ and level $\mathfrak{N}$. Since $KyK=\cup Ky_j$, the double coset $KyK$ acts on $\mathbf{f}$ as follows:
$$
(\mathbf{f}|KyK)(g)=\sum \mathbf{f}(gy_j^{tp}).
$$

For any finite place $u$, let $\varpi_u$ be a uniformizer of the maximal ideal of $\mathcal{O}_{F_u}$. When $\mathfrak{m}$ is a prime ideal, there is the following more precise definition:

\begin{definition}
The Hecke operator $T_u$ acts on $\mathbf{f}$ in the following way:

$$
(T_u\mathbf{f})(g):=\int_{K_u(\mathfrak{N})}
\mathbf{f}\left(g k_u
\begin{pmatrix}
\varpi_u & \\
& 1
\end{pmatrix}
\right)dk_u.
$$
\end{definition}

These two definitions are equivalent.

\paragraph{Hilbert newforms}
Similarly to the theory of modular forms over $\mathbb{Q}$, the so-called Hilbert newforms are the orthogonal complement of the Hilbert modular forms coming from other levels. We now describe this.

Fix a weight $k$. Consider a divisor $\mathfrak{m}$ of the integral ideal $\mathfrak{N}$ and a Hilbert modular form $\mathbf{g}\in\mathbf{S}_k(\mathfrak{m})$ of level $\mathfrak{m}$. For any divisor $\mathfrak{a}$ of $\mathfrak{N}\mathfrak{m}^{-1}$, consider $\mathbf{g}_{\mathfrak{a}}(z):=N(\mathfrak{a})^{-k_0/2} \mathbf{g}\left(z\begin{pmatrix} a^{-1} & \\ &1
\end{pmatrix}\right)$.
Then $\mathbf{g}_{\mathfrak{a}}\in\mathbf{S}_k(\mathfrak{am})$. The forms in $\mathbf{S}_k(\mathfrak{N})$ generated by all forms of the shape $\mathbf{g}_{\mathfrak{a}}$ are called old forms, and the space is denoted $\mathbf{S}_k^{old}(\mathfrak{N})$.

On $\mathbf{S}_k(\mathfrak{N})$ one can define the Petersson inner product: for $\mathbf{f},\mathbf{g}\in\mathbf{S}_k(\mathfrak{N})$,
$$
\left\langle  \mathbf{f},\mathbf{g}\right\rangle _{Pet,H}
:=\sum_{\nu=1}^{h}
\frac{1}{\mu(\Gamma_\nu\backslash\mathcal{H}^d)}
\int_{\Gamma_\nu\backslash\mathcal{H}^d}
\overline{f_\nu(z)}g_\nu(z)y^k d\mu(z).
$$

\begin{definition}[Hilbert newform]
A Hilbert newform $\mathbf{f}$ of weight $k$ and level $\mathfrak{N}$ is a form such that for every $\mathbf{g}\in S_k^{old}(\mathfrak{N})$, one has
$$
\left\langle  \mathbf{f},\mathbf{g}\right\rangle _{Pet,H}=0.
$$
\end{definition}

If a Hilbert newform is a simultaneous eigenfunction of all the Hecke operators and $C({\mathcal{O}_F},\mathbf{f})=1$, then such a newform is  called a normalized  eigen newform. In this article, by “newform” we mean an eigen newform. 

Meanwhile, for a decomposition $\mathfrak{N}=\mathfrak{MD}$, write $\mathbf{S}_k(\mathfrak{M},\mathfrak{D})$ for the space of $\mathfrak{D}$-new forms.

\begin{theorem}
If $\mathbf{f}$ is a normalized eigenform of weight $k$ and level $\mathfrak{N}$, then for the Hecke operator $T_{\mathfrak{m}}$ with $(\mathfrak{N},\mathfrak{m})=1$, one has
$$
T_{\mathfrak{m}}(\mathbf{f})=\mathrm{Norm}(\mathfrak{m})^{1-k_0/2}
C(\mathfrak{m},\mathbf{f})\mathbf{f}.
$$
\end{theorem}

\begin{proof}
    See \cite[Section~2]{Shimura1978}.
\end{proof}

\paragraph{Hecke algebras}

The Hecke algebra in this article means the polynomial ring with $\mathbb{Z}$-coefficients generated by the Hecke operators, namely
$$
\mathbb{T}=\mathbb{Z}[T_u \mid u \text{ is a prime place of } F]
\subseteq \operatorname{End}(\mathbf{S}_k(\mathfrak{N})).
$$

In this article, one often takes the $\mathfrak{D}$-Hecke algebra, i.e.\ the analogue of the above Hecke algebra,
$$
\mathbb{T}(\mathfrak{M},\mathfrak{D})
\subseteq \operatorname{End}(\mathbf{S}_k(\mathfrak{M},\mathfrak{D})),
$$
where $\mathfrak{MD}=\mathfrak{N}$.

We assume that the Hecke algebra in this article are finitely generated over $\mathbb{Z}$. 

\paragraph{The $L$-function of Hilbert modular forms}

For a Hilbert modular form $\mathbf{f}\in\mathbf{S}_k(\mathfrak{N})$, one can define its $L$-function through its Fourier coefficients.

\begin{definition}[The $L$-function of a Hilbert modular form]
The $L$-function of a Hilbert modular form $\mathbf{f}\in\mathbf{S}_k(\mathfrak{N})$ is
$$
L(\mathbf{f},s):=\sum_{\mathfrak{m}\text{ integral}}
\frac{C(\mathfrak{m},\mathbf{f})}{\mathrm{Norm}(\mathfrak{m})^s}.
$$
\end{definition}

\subsubsection{Elliptic curves}
This subsubsection mainly introduces the basic knowledge of elliptic curves used in this article, mainly reviewing the conductor and the $L$-function of elliptic curves. For this purpose we will briefly pass over the arithmetic properties of elliptic curves over finite fields, $p$-adic fields, and number fields.

\begin{definition}
An elliptic curve $E/k$ over a field $k$ is a smooth algebraic curve over a field $k$ of genus $1$ carrying a $k$-rational point. 
\end{definition}

An elliptic curve is an algebraic curve, so one can discuss its rational points. Let $l/k$ be a field extension, then $E(l)$ denotes the $l$-points of the elliptic curve $E$, i.e.\ the $l$-solutions of the algebraic equations.

An elliptic curve is an abelian variety, i.e.\ for every $k$-algebra $R/k$, $E(R)$ is an abelian group.

\paragraph{Elliptic curves over finite fields}
Assume that $k$ is a finite field of characteristic $p$. Then $E(k)$ is a finite group, which can be seen directly from the point of view of solving algebraic equations.

\paragraph{Elliptic curves over $p$-adic fields}
Assume that $k$ is a $p$-adic field, i.e.\ $k/\mathbb{Q}_p$ is a finite extension. Then $E(k)$ contains a subgroup of finite index which is a free $\mathbb{Z}_p$-module of rank $[k:\mathbb{Q}_p]$.

Write $\tilde{E}$ for the reduction of $E$ over the residue field $\kappa$. Then $\tilde{E}$ is a curve over $\kappa$ of genus $1$ with a $\kappa$-rational point. Depending on whether it is smooth or not, the elliptic curve $E$ is divided into good reduction and bad reduction. In the case of bad reduction, by the type of the singular point---a node (multiplicative) or a cusp (additive)---one further distinguishes multiplicative reduction and additive reduction.

In the good reduction case, depending on whether the finite group $\tilde{E}(\kappa)$ has nontrivial $p$-torsion, the reduction is divided into ordinary reduction and supersingular reduction.

Define the exponent $f(E/k) \in \mathbb{Z}_{\ge 0}$ of the local conductor as the sum of the two exponents $\varepsilon(E/k)$ and $\delta(E/k)$:
\[
f(E/k) = \varepsilon(E/k) + \delta(E/k).
\]
The exponent $\varepsilon(E/k)$ is defined by the reduction type:
\[
\varepsilon(E/k) = 
\begin{cases}
0, & E \text{ has good reduction}, \\
1, & E \text{ has multiplicative reduction}, \\
2, & E \text{ has additive reduction}.
\end{cases}
\]
The term $\delta(E/k)$ is the wild part of the conductor exponent, defined as follows:
\begin{itemize}
\item If $E/k$ has good or multiplicative reduction, then $\delta(E/k) = 0$.
\item If $E/k$ has additive reduction:
   \begin{itemize}
   \item if $p \neq 2,3$, then $\delta(E/k) = 0$.
   \item if $p = 2$ or $3$, then $\delta(E/k)$ is a nonnegative integer determined by Tate's algorithm. It can be strictly positive in certain cases.
   \end{itemize}
\end{itemize}

Thus, when $E/k$ has good reduction, multiplicative reduction, or additive reduction with $p \ge 5$, we always have $\delta(E/k) = 0$, and consequently
\[
f(E/k) = \varepsilon(E/k) = 
\begin{cases}
0, & E \text{ has good reduction}, \\
1, & E \text{ has multiplicative reduction}, \\
2, & E \text{ has additive reduction}.
\end{cases}
\]
If $p = 2$ or $3$ and $E/k$ has additive reduction, then $\delta(E/k) \ge 0$ and $f(E/k) = 2 + \delta(E/k) \ge 2$.

If $f(E/k) \le 1$, then $E/k$ is said to have semistable reduction. In other words, semistable reduction corresponds exactly to good or multiplicative reduction.

\paragraph{Elliptic curves over number fields}

Assume that $k$ is a number field, and $E/k$ is an elliptic curve over a number field.

Locally, for any finite place $u$ one can speak of the elliptic curve over $k_u$. The reduction of $E$ at the place $u$ is then the reduction of $E$ over $k_u$, i.e.\ the content explained in the previous paragraph.

Globally, the theorem that lays the foundation for the discussion of this article is the Mordell--Weil theorem:

\begin{theorem}[Mordell--Weil]\label{theorem:mordellweil}
If $E/k$ is an elliptic curve over a number field, then $E(k)$ is a finitely generated abelian group.
\end{theorem}

By Theorem~\ref{theorem:mordellweil}, one can define the Mordell--Weil rank of an elliptic curve, namely
$$
r_{\mathrm{MW}}(E/k):=\mathbf{rank}_{\mathbb{Z}}E(k).
$$

Define the conductor of an elliptic curve by
$$
\mathfrak{N}=\prod_{\mathfrak{q}}\mathfrak{q}^{f(E/k_{\mathfrak{q}})}.
$$

Since $E$ has good reduction at almost all places, the above product is indeed an integral ideal of $k$.

If $\mathfrak{N}$ is a squarefree ideal, i.e.\ $E$ has semistable reduction at every place, then the elliptic curve $E/k$ is called a semistable elliptic curve.

\paragraph{The $L$-function of elliptic curves}

Let $k$ be a number field, $u$ a finite place, $k_u$ the completed local field, and $\kappa_u$ the residue field. Write $q_u:=\#\kappa_u$.

When $u$ is a place of good reduction, write $\tilde{E}_u$ for the elliptic curve obtained by reduction over $\kappa_u$. Then $\tilde{E}_u/\kappa_u$ has the Zeta function
$$
Z(\tilde{E}_u/\kappa_u,T)=\exp\left(\sum_{n=1}^{\infty}
\#\tilde{E}_u(\kappa_{u,n})\frac{T^n}{n}\right),
$$
where $\kappa_{u,n}$ is the unique extension of degree $n$ of $\kappa_u$.

By the book of Silverman \cite[Chapter V.2]{silverman2009arithmetic}, one obtains
$$
Z(\tilde{E}_u/\kappa_u,T)=\frac{L_u(T)}{(1-T)(1-q_u T)},
$$
where $L_u(T)=1-a_u T+q_u T^2\in\mathbb{Z}[T]$, and $a_u=q_u+1-\#\tilde{E}_u(\kappa_{u})$.

When $u$ is a place of bad reduction, define
\begin{center}
$L_u(T) = \left\{
\begin{aligned}
&1-T & E\text{ has split multiplicative reduction} \cr
&1+T & E\text{ has non-split multiplicative reduction} \cr
&1 & E\text{ has additive reduction}
\end{aligned}
\right.$
\end{center}

In every case, $L_u(1/q_u)=\#\tilde{E_{ns}}_u(\kappa_{u})/q_u$, where $\tilde{E_{ns}}_u$ is the non-singular part of the elliptic curve over $\kappa_u$.

\begin{definition}[The Hasse--Weil $L$-function of an elliptic curve]
The $L$-function of $E/k$ is the following Euler product:
$$
L(E/k,s):=\prod_{u\text{ finite place}} \frac{1}{L_u(q_{u}^{-s})}.
$$
\end{definition}

\begin{remark}
For any number field $k$, it is conjectured that $L(E/k,s)$ has analytic continuation to the whole complex plane. In the known cases, the analytic continuation is usually proved by converting this $L$-function into an analytic $L$-function, such as the $L$-function of a Hecke character or of a modular form, and using the analytic continuation of the latter. For general number fields, this conjecture remains open.
\end{remark}

\subsubsection{Galois representations}
This subsubsection introduces Galois representations, mainly recalling the Galois representations of elliptic curves and of Hilbert modular forms.

Let $k$ be a number field and $G=G_k=\Gal(\bar{k}/k)$ the absolute Galois group. If $R$ is a ring, then an $R$-coefficient Galois representation $M$ is an $R[G]$-module.

The coefficient rings $R$ appearing in this article are generally either a $p$-adic field, the ring of integers of such a field, or its residue field.

\paragraph{Galois representations arising from an elliptic curve
$E$}

For an elliptic curve $E/k$, its geometric points $E(\bar{k})$ naturally form a Galois representation with $\mathbb{Z}$-coefficients. Write $E[N]\subseteq E(\bar{k})$, $N\in\mathbb{Z}$, for the $N$-torsion points.

This article is concerned with the following representations: write $T=T_p:=\varprojlim E[p^n]$ for the Tate module. This is a Galois representation with $\mathbb{Z}_p$-coefficients of rank $2$. Consider $V=V_p=T_p\otimes\mathbb{Q}_p$ and $W=W_p=V_p/T_p$, which are Galois representations with $p$-adic-field coefficients and torsion coefficients, respectively.  $W$ can be identified with $E[p^\infty]$. Moreover, $E[p]$ as a Galois representation with $\mathbb{Z}/p\mathbb{Z}$-coefficients is also a focus of the discussion in this article. Write $\rho_E:G_k\rightarrow GL(T)$ for the Galois representation of $E$.

The various representations induced by $E$ have the following basic properties:
\begin{itemize}
\item The Tate module has rank $2$. In other words, as an abstract $\mathbb{Z}_p$-module, $T\cong\mathbb{Z}_p\oplus\mathbb{Z}_p$.
\item If the places above $p$ are all of good reduction and $u\nmid p\mathfrak{N}$, then
    $$
    \det(\rho_{E}(\mathrm{Frob}_{u}))=q_{u}, \quad
    \tr(\rho_{E}(\mathrm{Frob}_{u}))=a_{u}.
    $$
Here $a_u$ is the quantity appearing in the definition of the Hasse--Weil $L$-function.
\end{itemize}

\paragraph{Galois representations arising from a Hilbert modular
form $\mathbf{f}$} For a Hilbert modular form $\mathbf{f}$ of parallel weight $k\ge 1$ and level $\mathfrak{N}$, which is ordinary at $\lambda|p$, one can, by the work of Wiles \cite{MR969243}, attach a $\lambda$-adic representation
$$
\rho_{\mathbf{f},\lambda}:G_F\rightarrow GL_2(\mathcal{O}_{F_{\lambda}})
$$
satisfying the following conditions:
\begin{itemize}
\item it is unramified outside $\mathfrak{N}q_{\lambda}$.

\item for $u\nmid\mathfrak{N}q_{\lambda}$ one has
    $$
    \det(\rho_{\mathbf{f},\lambda}(\mathrm{Frob}_{u}))=q_{u}^{k-1},
    \quad \tr(\rho_{\mathbf{f},\lambda}(\mathrm{Frob}_{u}))=C(u,\mathbf{f}).
    $$
\end{itemize}

Here $\lambda$-ordinary means that for every $\mathfrak{p}|q_\lambda$, the equation
$$
x^2-C(\mathfrak{p},\mathbf{f})x+q_{\mathfrak{p}}^{k-1}
$$
has at least one root which is a unit modulo $\lambda$.

Later in this article we will take $\lambda=w$. In that case $F_\lambda=\mathbb{Q}_p$.

\subsection{Basic Assumptions and A Sketch of Proof}\label{section:basicassumption}
This subsection is divided into two subsubsections. First, Subsubsection~\ref{sec:assumptions} introduces the assumptions used throughout the article, and fixes the notation of the whole article while stating them. Then, Subsubsection~\ref{section:sketch} outlines the proof of the main theorem.

\subsubsection{Assumptions and notation used throughout the
article}\label{sec:assumptions}

The assumptions gathered in this subsubsection are used throughout the article and fix the notation of the whole text.

\paragraph{Basic assumptions on fields}
\begin{itemize}
\item $\mathbb{Z}, \mathbb{Q}, \mathbb{R}, \mathbb{C}, \mathbb{Q}_p$ denote the ring of integers, the field of rational numbers, the field of real numbers, the field of complex numbers, and the field of $p$-adic numbers, respectively.

\item $F/\mathbb{Q}$ is a fixed totally real field of degree $d$, with different ideal $\delta_F$ and absolute discriminant $D_F$, i.e.\ for every embedding $\tau:F\rightarrow\mathbb{C}$, one has $\tau(F)\subseteq\mathbb{R}$. Write $\mathcal{O}_F$ for the ring of integers of $F$.

\item $\mathbb{A}:=\mathbb{A}_F$ is the ad\`ele ring of $F$, and $\mathbb{A}_f$ is the finite part of the ad\`ele ring.

\item $K/F$ is a totally imaginary quadratic extension, i.e.\ for every embedding $\tau:K\rightarrow\mathbb{C}$, one has $\tau(K)\nsubseteq\mathbb{R}$. Write $\Gal(K/F)=\{1,c\}$, where $c$ is complex conjugation.

\item $u$ denotes a prime place of $F$ or $K$.

\item $F_u$, $K_u$ denote the completions at the place $u$.

\item $p>2$ is a fixed odd prime. Assume that $p\mathcal{O}_F=\prod_1^g \mathfrak{p}_i^{e_i}$ is the prime decomposition of $p$ in $F$ and that $p$ is unramified, i.e.\ $e_i=1$.

\item $\mathfrak{p}_i$ is a chosen prime ideal above $p$, denoted for short by $w$, and we require $F_w\cong\mathbb{Q}_p$.

\item $\iota_w:\overline{\mathbb{Q}}\rightarrow \overline{F}_{w}$ and $\iota_\infty:\overline{\mathbb{Q}} \rightarrow\mathbb{C}$ are two fixed embeddings, and $\iota:\mathbb{C}_p\rightarrow\mathbb{C}$ is a fixed isomorphism. Moreover, they are compatible with each other.

\item $L'/\mathbb{Q}_{p}$ is a finite extension, and $\mathcal{O}_{L'}$ is the ring of integers of $L'$. Below, $L'$ and $\mathcal{O}_{L'}$ are generally used as the coefficient field or coefficient ring of $p$-adic representations. When used as the coefficient ring of the Tate module of an elliptic curve, $L'=\mathbb{Q}_p$.
\end{itemize}

\paragraph{Elliptic curves}
\begin{itemize}
\item $E$ is a fixed elliptic curve over $F$, i.e.\ a smooth projective algebraic curve of genus $1$ defined over $F$ carrying a rational point.

\item $\mathfrak{N}$ is the conductor of $E$.

\item $E$ is semistable, i.e.\ the conductor of $E$ is $\mathfrak{N}=\prod\mathfrak{q}^{k_{\mathfrak{q}}}$, where $k_{\mathfrak{q}}\le 1$ and almost all $k_{\mathfrak{q}}=0$. This is an ideal of $\mathcal{O}_{F}$.

\item $E$ has good ordinary reduction at $\mathfrak{p}_i$ for all $1\le i\le g$. By the good reduction assumption, $k_{\mathfrak{p}_i}=0$ for all $1\le i\le g$.

\item $E(F)$ is the set of $F$-points of the elliptic curve. Write $r_{\mathrm{MW}}(E/F)=\mathbf{rank}_{\mathbb{Z}}(E(F))$.

\item $L(E/F,s)$ is the Hasse--Weil $L$-function of the elliptic curve $E$. Assume that it has analytic continuation to the whole complex plane, and write $r_{\mathrm{an}}(E/F)=\ord_{s=1}L(E/F,s)$.

\item $\Sha(E/F)=\ker\{H^1(F,E)\rightarrow\prod H^1(F_{u},E)\}$ is the Shafarevich--Tate group of the elliptic curve $E$.

\item the analytic rank of $E$ over $K$ is $r_{\mathrm{an}}(E/K)=1$. Then, by the generalized Gross--Zagier--Kolyvagin theorem, one has $r_{\mathrm{MW}}(E/K)=1$ and $\Sha(E/K)$ is finite.

\item $T_p(E)$ is the Tate module of $E$, $V_p(E)=T_p(E)\otimes\mathbb{Q}_p$, and $W_p(E)=V_p(E)/T_p(E)$.

\item $E[p]$ is an irreducible representation of $G_K$.
\end{itemize}

\begin{remark}
Note that, in this article, we do not require $E/F$ to have analytic rank one, since the main theorem is a formula over $K$, although our final aim is to deal with the totally real case.

The irreducibility of $E[p]$ is also a basic assumption, as in Jetchev--Skinner--Wan \cite{jetchev5birch}. Under this condition, the $p$-part of $E(K)_{tor}$ in the Birch and Swinnerton-Dyer formula is trivial and need not be considered. On the other hand, the irreducibility of $E[p]$ is also used frequently in Section~\ref{sec:GZ} of this article.
\end{remark}

\paragraph{Hilbert modular forms}

\begin{itemize}
\item $\mathbf{f}$ is a Hilbert newform (normalized eigenform) of parallel weight $k=(2,2,\cdots,2)$ and level $\mathfrak{N}$, whose Fourier coefficients lie in $\mathbb{Z}$.

\item $(\rho_{\mathbf{f},w},V_{\mathbf{f}})$ is the $p$-adic Galois representation attached to ${\mathbf{f}}$.

\item $T_{\mathbf{f}}$ is a $G_F$-stable lattice.

\item $W_{\mathbf{f}}=V_{\mathbf{f}}/T_{\mathbf{f}}$.
\end{itemize}

\begin{remark}
The assumptions on Hilbert modular forms are standard. Since the starting point of this article is elliptic curves, the corresponding Hilbert modular form is easy to describe.
\end{remark}

\paragraph{The modularity conjectures}

The  modularity conjectures mean that, for the Hilbert modular form $\mathbf{f}$, one can find an isogeny class of abelian varieties over $F$ associated with it, and the elliptic curve $E$ studied in this article belongs to this isogeny class.

The following points describe the modularity conjectures from several different aspects:

\begin{itemize}
\item The elliptic curve $E$ is modular, meaning that there exists a Hilbert modular form of parallel weight $2$ and level $\mathfrak{N}$ such that
    $$
    L(E/F,s)=L(\mathbf{f},s).
    $$

\item The elliptic curve $E$ is modular, meaning that there exists a Hilbert modular form $\mathbf{f}$ of parallel weight $2$ and level $\mathfrak{N}$ such that for every $u|p$, as representations of $G_F$, one has $E[p]\cong\rho_{\mathbf{f},w}|_{G_{F_u}}\bmod\varpi_u$.

\item The elliptic curve $E$ is modular, meaning that there exists a nontrivial morphism over $F$
    $$
    X_{\mathfrak{M},\mathfrak{D}}\rightarrow E,
    $$
where $\mathfrak{MD}=\mathfrak{N}$, and $X_{{\mathfrak{M}},\mathfrak{D}}$ is the Shimura curve defined in Section~\ref{sec:GZ} below.
\end{itemize}

In this article, we assume all these modularity conjectures.

\begin{remark}
The modularity conjecture is standard. When $F=\mathbb{Q}$, these assumptions were proved around the year 2000 by mathematicians led by Andrew Wiles. For details, see the literature \cite{1995Modular,0Ring,2001On}.
\end{remark}

\subsubsection{Sketch of the proof of the main theorem}\label{section:sketch}
This subsubsection briefly describes how the main theorem~\ref{theorem:maintheorem} of this article is proved using the various tools of this article.

Under the notations of Theorem~\ref{theorem:maintheorem}, the Birch and Swinnerton-Dyer formula consists mainly of the following ingredients:

\begin{itemize}
\item the derivative special value $L'(E/K,1)$ of the Hasse--Weil $L$-function.

\item the congruence period $\Omega_{\mathbf{f}}^{\mathrm{cong}}$, the regulator $\mathrm{Reg}(E/K)$.

\item the Shafarevich--Tate group $\Sha(E/K)[p^\infty]$.

\item the Tamagawa numbers $c_u(E/K)$.
\end{itemize}

Note that, since Theorem~\ref{theorem:maintheorem} assumes that $E[p]$ is irreducible as a representation of $G_K$, the torsion group $E(K)_{tor}$ in the formula is ignored here. Also note that since we assume that $p$ is unramified in $K$, then the square root $\sqrt{|D_K|}$ of the absolute value of the absolute discriminant of $K$ is a $p$-adic unit.

In the study of the Birch and Swinnerton-Dyer formula, the above objects have been carefully examined at different historical stages, especially the derivative special value of the Hasse--Weil $L$-function and the Shafarevich--Tate group. The idea of the proof of this article is similar to that of Jetchev--Skinner--Wan \cite{jetchev5birch}. In this article, our strategy can be viewed as a strategy centered on ``patching the Tamagawa numbers''. We describe it briefly below.

First, using the Gross--Zagier formula (Theorem~\ref{Theorem:grosszaigerimaintheorem}), one introduces the derivative special value $L'(E/K,1)$ of the Hasse--Weil $L$-function, as well as the congruence period $\Omega_{\mathbf{f}}^{\mathrm{cong}}$ and the square root $\sqrt{|D_K|}$ of the absolute value of the absolute discriminant of $K$. By the choice of $K$, the Gross--Zagier formula gives part of the Tamagawa numbers. Precisely, those at the prime places of multiplicative reduction which are inert in $K$. But at this point the pairing of Heegner points is introduced.

Then, since the pairing of Heegner points can be written as the product of the regulator $\mathrm{Reg}(E/K)$ and the square of the index of the Heegner point in $E(K)$, after substitution the desired Birch and Swinnerton-Dyer formula still lacks the Shafarevich--Tate group and the remaining Tamagawa numbers.

Next, by the key formula~\ref{prop:keyformula} of Section~\ref{sec:selmer}, the square of the index of the Heegner point in $E(K)$ can be replaced by the product of the Shafarevich--Tate group, the anticyclotomic Selmer group, and the error terms.

Then, the control theorem~\ref{theorem:controltheorems} together with the Iwasawa main conjecture~\ref{conj:iwasawamain} replaces the anticyclotomic Selmer group by the product of the special value of the $p$-adic $L$-function and the remaining Tamagawa numbers.

Finally, by the Liu--Zhang--Zhang formula~\ref{theorem:LZZoftrivialcharacter}, the special value of the $p$-adic $L$-function cancels the error terms, and the proof is complete.

\section{An Anticyclotomic Control Theorem}\label{sec:selmer}
This section studies a class of Galois cohomology groups, namely the anticyclotomic Selmer groups. The purpose is to use the anticyclotomic Selmer groups to connect the Shafarevich--Tate group and part of the Tamagawa numbers appearing in the Birch and Swinnerton-Dyer formula.

This section is divided into three subsections. Subsection~\ref{section:controlthemsnotation} introduces the additional notation, assumptions, and preliminaries needed in this section. The work of Subsection~\ref{section:controloverK} is carried out entirely over the totally imaginary quadratic extension $K$, its main content is to use the anticyclotomic Selmer groups to give a formula involving the Shafarevich--Tate group, namely the key formula~\eqref{equation:key} in Proposition~\ref{prop:keyformula}. The work of Subsection~\ref{section:controlthmsoverL} is carried out over a particular Iwasawa $\Zp$-extension $L$, its main content is to use the control theorem~\ref{theorem:controltheorems} to give the relation between the anticyclotomic Selmer groups and the size of a certain Iwasawa module. This step introduces part of the Tamagawa numbers.

The idea of this section is inspired by the work of Jetchev--Skinner--Wan \cite{jetchev5birch}, here we provide the necessary details and clarify several points that are only briefly addressed there.

\subsection{Assumptions, Notation, and Preliminaries}\label{section:controlthemsnotation}

\subsubsection{The Iwasawa algebra}

Consider the field $K_\infty$ obtained as the compositum of all $\Zp$-extensions of $K$. Then, by class field theory, $\Gal(K_\infty/K)\cong \Zp^{d+1+\delta}$, where $\delta=\delta(F)=\delta(K)$ is the Leopoldt constant, see \cite[Theorem 13.4]{washington2012introduction}. The equality of the constants $\delta(F)=\delta(K)$ for the CM pair $K/F$ is standard. Since $K/F$ is an abelian extension, $\Gal(K/F)$ acts on $\Gal(K_\infty/K)$ by conjugation. Since $p>2$, the group $\Gal(K_\infty/K)$ splits into the parts on which the action has eigenvalues $\pm 1$. By class field theory, as a $\Gal(K/F)$-module,
\[
  \Gal(K_\infty/K)\cong \Zp^{1+\delta}\oplus \Zp^{d},
\]
where $\Zp^{1+\delta}$ is the part with eigenvalue $+1$ and $\Zp^{d}$ is the part with eigenvalue $-1$. The field $K^-=K_\infty^{\Zp^{1+\delta}}$ is called the (maximal) anticyclotomic $\Zp^d$-extension, and we write $\Gamma^-:=\Gal(K^-/K)\cong \Zp^{d}$ for its maximal anticyclotomic Galois group. Any intermediate field $L/K$ of $K^-/K$ with $\Gal(L/K)\cong \Zp$ is called an anticyclotomic $\Zp$-extension.

For a finite extension $L'/\Qp$ with ring of integers $\mathcal{O}_{L'}$, the anticyclotomic Iwasawa algebra with coefficients in $\mathcal{O}_{L'}$ is defined to be $\Lambda_-=\mathcal{O}_{L'}[[\Gamma^-]]$.

Note that the Iwasawa algebra is not canonically isomorphic to a power series ring, the isomorphism $\Lambda_-\cong \mathcal{O}_{L'}[[X_1,\dots,X_d]]$ depends on the choice of topological generators of $\Gamma^-$. Fix an intermediate field $L/K$ of $K^-/K$ with $\Gal(L/K)\cong \Zp$, and write
\[
  \Lambda_L=\mathcal{O}_{L'}[[\Gal(L/K)]]\cong \mathcal{O}_{L'}[[X]].
\]
By convention, when the context is clear, $\Lambda_L$ is simply denoted by $\Lambda$.

\subsubsection{Selmer groups}

This subsubsection introduces Selmer groups, the main objects of study of this section.

Let $G=\Gal(\bar{k}/k)$ be the absolute Galois group of a number field $k$, and let $A$ be a module on which $G$ acts continuously. Let $k_u$ be the completion of $k$ at the place $u$, and write $G_u=\Gal(\bar{k}_u/k_u)$ for the absolute Galois group of $k_u$ and $I_u$ for the corresponding inertia group. A so-called \emph{local condition} at $u$ is just a subgroup of the cohomology group $H^1(k_u,A)$.

\paragraph{Local conditions}
This paragraph introduces the local conditions for Selmer groups, in preparation for the global Selmer groups below.

For an arbitrary continuous $G$-module $A$ and any finite place $u$ of $k$, one can define the unramified local condition
\[
  H^1_{ur}(k_u,A)=\ker\left\{H^1(k_u,A)\rightarrow H^1(I_u,A)\right\}.
\]

For the $p$-adic Galois representation $V=V_p(E)$ attached to $E$, one can define its \textbf{Bloch--Kato local condition}, denoted $H^1_{\mathrm{BK}}(k_u,V)$:
\[
  H^1_{\mathrm{BK}}(k_u,V)
  =\left\{
  \begin{aligned}
    &\ker\left\{H^1(k_u,V)\rightarrow H^1(I_u,V)\right\} & u\nmid p\infty,\\
    &\ker\left\{H^1(k_u,V)\rightarrow H^1(k_u,V\otimes B_{cris})\right\} & u\mid p,\\
    &0 & u\mid \infty,
  \end{aligned}
  \right.
\]
where $B_{cris}$ is the period ring defined by Fontaine \cite{fontaine1982certains}.

By the long exact sequence induced by the short exact sequence $0\rightarrow T\rightarrow V\rightarrow W\rightarrow 0$, the Bloch--Kato local conditions for $T$ and $W$ are defined to be the \emph{preimage and the image} of $V$'s Bloch--Kato condition.

\begin{remark}
One should note that, in general, defining the Bloch--Kato conditions for $T$ and $W$ directly by imitating the definition of the Bloch--Kato condition for $V$ is not quite the same as the definition via the exact sequence.
\end{remark}

We now introduce the anticyclotomic local condition. In this case $k=K$ is a totally imaginary quadratic extension of $F$. By the assumptions of the article, the fixed place $w\mid p$ of $F$ splits in $K$ as $w=v\bar{v}$. The \emph{anticyclotomic local condition} for $V$ is defined to be
\[
  H^1_{ac}(K_u,V)
  =\left\{
  \begin{aligned}
    &H^1(K_{\bar{v}},V) & u=\bar{v},\\
    &H^1_{\mathrm{BK}}(K_u,V) & u\nmid p\infty \text{ and } u \text{ splits in } K,\\
    &\hphantom{H^1_{\mathrm{BK}}(K_u,V)} & \text{or } u\mid p,\ u\neq v,\bar{v},\\
    &0 & \text{otherwise}.
  \end{aligned}
  \right.
\]
Just as for the Bloch--Kato condition, the anticyclotomic local conditions for $T$ and $W$ are defined through the long exact sequence.

Recall the notation: $L/K$ is an anticyclotomic $\Zp$-extension and $\Lambda_L$ is the one-variable Iwasawa algebra. Consider the $G_K$-module $M:=T\otimes \Lambda_L^\vee$, where $\Lambda_L^\vee=\Hom(\Lambda_L,\Qp/\Zp)$, and $M$ carries the $G_K$-action $\rho\otimes\Psi^{-1}$, where $\Psi:G_K\rightarrow\Gal(L/K)$ is the natural projection.

For $M$, one can also define its anticyclotomic local condition. However, before giving the  definition, note that $M$ is not a $p$-adic Galois representation in the sense of finite groups or finite $\mathbb{Z}_p$-rank, so one cannot directly discuss its Bloch--Kato local condition. Nevertheless, one can obtain an $M$-version of the Bloch--Kato condition by means of Shapiro's lemma. By Shapiro's lemma, there is a natural isomorphism
\[
  H^1(K_u,T\otimes\Lambda)\cong \varprojlim_{K'}\bigoplus_{u'\mid u} H^1(K'_{u'},T),
\]
where $K'$ runs over the subfields of $L/K$ and $u'$ runs over the places of $K'$ above $u$.

For the groups $H^1(K'_{u'},T)$ one can define the Bloch--Kato local condition $H^1_{\mathrm{BK}}(K'_{u'},T)$. Using these conditions, one defines the Bloch--Kato local condition
\[
  H^1_{\mathrm{BK}}(K_u,T\otimes\Lambda)
  := \varprojlim_{K'}\bigoplus_{u'\mid u} H^1_{\mathrm{BK}}(K'_{u'},T).
\]
Then $H^1_{\mathrm{BK}}(K_u,M)$ is defined to be the local condition obtained from $H^1_{\mathrm{BK}}(K_u,T\otimes\Lambda)$ via Tate local duality.

At this point, the anticyclotomic local condition for $M$ is defined to be
\[
  H^1_{ac}(K_u,M)
  =\left\{
  \begin{aligned}
    &H^1(K_{\bar{v}},M) & u=\bar{v},\\
    &H^1_{ur}(K_u,M) & u\nmid p\infty \text{ and } u \text{ splits in } K,\\
    &H^1_{\mathrm{BK}}(K_u,M) & u\mid p,\ u\neq v,\bar{v},\\
    &0 & \text{otherwise}.
  \end{aligned}
  \right.
\]

\begin{remark}
If one puts the unramified local condition at the places $u\mid p$ with $u\neq v,\bar{v}$, then this is fine by definition. However, from a technical point of view, the proof of Proposition~\ref{prop:surjective}, especially the assertion part, may have problems.
\end{remark}

\paragraph{Selmer structures and duality theorems}
This paragraph uses the local conditions defined above to define the global Selmer groups. It then introduces the duality theorems, an important tool in the study of Selmer groups.

Let $k$ be a number field. For a $G=\Gal(\bar{k}/k)$-module $A$, a \emph{Selmer structure} $\mathcal{F}$ is a system of local conditions $H^1_{\mathcal{F}}(k_u,A)\subseteq H^1(k_u,A)$ such that for almost all places $u$, $H^1_{\mathcal{F}}(k_u,A)=H^1_{ur}(k_u,A)$. A Selmer structure determines a global Selmer group:
\[
  H^1_{\mathcal{F}}(k,A)
  :=\ker\left\{H^1(k,A)\rightarrow \bigoplus_u
    \frac{H^1(k_u,A)}{H^1_{\mathcal{F}}(k_u,A)}\right\}.
\]

For a $G$-module $A$, write $A^\ast=\Hom_{cont}(A,\Qp/\Zp(1))$ for its arithmetic dual. Tate's local duality theorem asserts that $H^1(k_u,A)$ and $H^1(k_u,A^\ast)$ are dual to each other. Hence, if $\mathcal{F}$ is a Selmer structure for $A$, then via Tate local duality one can define the dual Selmer structure $\mathcal{F}^\ast$ for $A^\ast$, namely
\[
  H^1_{\mathcal{F}^\ast}(k_u,A^\ast)
  :=\Hom\left(H^1(k_u,A)/H^1_{\mathcal{F}}(k_u,A),\Qp/\Zp\right).
\]
In other words, it is the annihilator module of $H^1_{\mathcal{F}}(k_u,A)$.

In order to state the Poitou--Tate duality theorem, we introduce one more piece of notation. Let $\mathcal{F},\mathcal{G}$ be two Selmer structures for the $G$-module $A$. If for every place $u$ we have $H^1_{\mathcal{F}}(k_u,A)\le H^1_{\mathcal{G}}(k_u,A)$, then we write $\mathcal{F}\le\mathcal{G}$.

\begin{theorem}[Poitou--Tate global duality]
Let $\mathcal{F}\le\mathcal{G}$ be two Selmer structures for $A$. Consider the exact sequences
\[
  0\rightarrow H^1_{\mathcal{F}}(k,A) \rightarrow H^1_{\mathcal{G}}(k,A)
  \xrightarrow{loc_{\mathcal{F}}^{\mathcal{G}}}
  \bigoplus \frac{H^1_{\mathcal{G}}(k_u,A)}{H^1_{\mathcal{F}}(k_u,A)}
\]
and
\[
  0\rightarrow H^1_{\mathcal{G}^\ast}(k,A^\ast) \rightarrow
  H^1_{\mathcal{F}^\ast}(k,A^\ast) \xrightarrow{loc^{\mathcal{F}^\ast}_{\mathcal{G}^\ast}}
  \bigoplus \frac{H^1_{\mathcal{F}^\ast}(k_u,A^\ast)}{H^1_{\mathcal{G}^\ast}(k_u,A^\ast)}.
\]
Then the image of $loc_{\mathcal{F}}^{\mathcal{G}}$ and the image of $loc^{\mathcal{F}^\ast}_{\mathcal{G}^\ast}$ are mutually orthogonal with respect to Tate local duality.
\end{theorem}

\begin{proof}
This is the consequence of Poitou--Tate global duality used throughout Jetchev--Skinner--Wan \cite[Theorem 2.3.2]{jetchev5birch}. See also Rubin's book \cite[Theorem I.7.3]{rubin2000euler} and Milne's book \cite[Theorem I.4.10]{milne2006arithmetic}.
\end{proof}

\begin{remark}
The Bloch--Kato local condition and the anticyclotomic local condition introduced in the previous paragraph also constitute Selmer structures, so the $\mathcal{F}$ in this paragraph can be replaced by the BK and ac conditions above.

Indeed, it suffices to note that the Bloch--Kato local condition and the anticyclotomic local condition are unramified at almost all places.
\end{remark}

For convenience, we make the following convention: $\mathcal{F}_u$ is the new Selmer structure obtained from the original Selmer structure by replacing only the local condition at $u$ by $0$, while $\mathcal{F}^u$ is the new Selmer structure obtained from the original Selmer structure by replacing only the local condition at $u$ by $H^1(k_u,-)$.

\paragraph{Results on cohomology groups with $W$- and $T$-coefficients}
For the convenience of the proofs below, this paragraph collects and proves some facts about cohomology groups. Throughout this paragraph, $T,V,W$ are those attached to an elliptic curve.

\begin{proposition}\label{prop:saitoweight}
$V=T_p(E)\otimes\Qp$ has the following properties:
\begin{itemize}
\item $V$ is a $2$-dimensional vector space.
\item $V$ is pure of weight $-1$.
\end{itemize}
\end{proposition}

\begin{proof}
See the article of Saito \cite{saito1997modular}.
\end{proof}

\begin{proposition}\label{prop:Wfinite}
For $u\nmid p$, $H_{\mathrm{BK}}^1(K_u,W)=0$ and $H^1(K_u,W)$ is a finite group.
\end{proposition}

\begin{proof}
Under this condition, since the weight of $V$ is not $0$ or $1$, we have $H^0(K_u,V)=0$. Likewise $H^0(K_u,V^{\vee}(1))=0$. By local duality, $H^2(K_u,V)=0$. By the local Euler characteristic formula, we have $H^1(K_u,V)=0$, see \cite[Section 2.2.2]{jetchev5birch}. Hence $H^1_{\mathrm{BK}}(K_u,V)=0$. By definition, $H^1_{\mathrm{BK}}(K_u,W)$ is the image of $H^1_{\mathrm{BK}}(K_u,V)=0$, hence it is also $0$.

Consider the long exact sequence induced by $0\rightarrow T\rightarrow V\rightarrow W\rightarrow 0$. Together with $H^1(K_u,V)=0$ and $H^2(K_u,V)=0$, we obtain that $H^1(K_u,W)$ is isomorphic to $H^2(K_u,T)$. But the former is a torsion group while the latter is a finitely generated $\Zp$-group. Hence both are finite groups.
\end{proof}

\begin{proposition}\label{prop:H0T=0}
For every place $u$ of $K$, we have $H^0(K_u,T)=0$.
\end{proposition}

\begin{proof}
For $u\nmid p$, this follows from $H^0(K_u,V)=0$ by argument in Proposition~\ref{prop:Wfinite}.

For $u\mid p$, this also follows from $H^0(K_u,V)=0$, but here one needs some $p$-adic Hodge theory. If $H^0(K_u,V)\neq 0$, then the Frobenius element would have eigenvalue $1$ acting on $\mathrm{WD}_u(V)^{N=0}$, contradicting that $V$ is pure of weight different from $0$ and $1$. Here $\mathrm{WD}_u(V)$ and $N$ are the Weil--Deligne representation and the monodromy operator defined via $p$-adic Hodge theory. For details, see the footnote on page~375 of \cite{jetchev5birch}.
\end{proof}

\begin{proposition}
For $V$, the anticyclotomic local condition and the Bloch--Kato local condition differ only at $v,\bar{v}$.
\end{proposition}

\begin{proof}
First, consider $u\nmid p$: (1) if $u\mid\infty$, then by definition the anticyclotomic local condition is $0$, and the Bloch--Kato local condition is also $0$ by definition.

(2) If $u\nmid p\infty$ and $u$ does not split, then the anticyclotomic local condition is $0$ by definition. The Bloch--Kato local condition is the unramified local condition, which is also $0$ by the proof of Proposition~\ref{prop:Wfinite} above.

(3) If $u\nmid \infty$ and $u$ splits, then the anticyclotomic local condition is the Bloch--Kato local condition.

In summary, away from $p$, the anticyclotomic local condition is the Bloch--Kato local condition.

Second, if $u\mid p$, then by the definition of the anticyclotomic local condition, it differs from the Bloch--Kato local condition only outside $v,\bar{v}$.
\end{proof}

\begin{corollary}\label{WBKACalmostsame}
For the $W=E[p^\infty]$ of interest to us, the anticyclotomic local condition and the Bloch--Kato local condition differ only at $v,\bar{v}$. For $T=T_p(E)$, the anticyclotomic local condition and the Bloch--Kato local condition also differ only at $v,\bar{v}$.
\end{corollary}

\begin{proof}
Since the anticyclotomic and Bloch--Kato local conditions for $T,W$ are both given, through the long exact sequence, by the corresponding images of the maps for $V$, and since the anticyclotomic and Bloch--Kato local conditions for $V$ differ only at $v,\bar{v}$, the same holds for $T$ and $W$.
\end{proof}

\begin{corollary}\label{cor:ac=acstar}
For $W$, the Selmer structures $ac$ and $ac^\ast$ differ only at $v,\bar{v}$.
\end{corollary}

\begin{proof}
By the previous corollary, for $u\neq v,\bar{v}$,
\[
  \begin{aligned}
    H^1_{(ac)^\ast}(K_u,W)
    &\cong \Hom\left(\frac{H^1(K_u,T)}{H^1_{ac}(K_u,T)},\Qp/\Zp\right)\\
    &=\Hom\left(\frac{H^1(K_u,T)}{H^1_{\mathrm{BK}}(K_u,T)},\Qp/\Zp\right)\\
    &\cong H^1_{(\mathrm{BK})^\ast}(K_u,W).
  \end{aligned}
\]
However, by the article of Bloch--Kato \cite{bloch1990functions}, these two groups coincide: $H^1_{(\mathrm{BK})^\ast}(K_u,W)=H^1_{\mathrm{BK}}(K_u,W)$. This proves the corollary.
\end{proof}

\subsection{Results over $K$: The Key Formula}\label{section:controloverK}

This subsection proves Proposition~\ref{prop:acandsha} and the key formula~\ref{prop:keyformula}. These two results connect the size of the Shafarevich--Tate group and the size of the anticyclotomic Selmer group.

Recall the convention of this article: $K/F$ is a totally imaginary quadratic extension of a totally real field. The prime $p$ is unramified in $F$, and the fixed place $w\mid p$ of $F$ splits in $K$ as $w=v\bar{v}$. Let $L'/\Qp$ be a finite extension. In this subsection, $L'$ is the coefficient field of $V$, since $V$ comes from an elliptic curve, $L'=\Qp$.

Write
\[
  \Sha_{\mathrm{BK}}(W/K)
  =H_{\mathrm{BK}}^1(K,W)/H_{\mathrm{BK}}^1(K,W)_{\mathrm{div}}
\]
for the Shafarevich--Tate group of Bloch--Kato type. Since $H_{\mathrm{BK}}^1(K,W)$ is a cofinitely generated $\Zp$-module, this group is finite.

\begin{remark}
Although it will be explained below that, when the classical Shafarevich--Tate group is finite, these two groups are essentially the same, here we do not directly identify the Shafarevich--Tate group of Bloch--Kato type with the classical Shafarevich--Tate group.
\end{remark}

\subsubsection{A formula relating the Shafarevich--Tate group}
This subsubsection proves the following Proposition~\ref{prop:acandsha}. This proposition does not involve any information about $\Zp$-extensions, so it can be understood as a result over $K$. In essence, this result is an exact comparison of different Selmer structures.

\begin{proposition}\label{prop:acandsha}
We have
\[
  \#H_{ac}^1(K,W)=\#\Sha_{\mathrm{BK}}(W/K)\cdot(\#\delta_v)^2,
\]
where
\[
  \delta_v=\coker\left\{H_{\mathrm{BK}}^1(K,T)
    \xrightarrow{\operatorname{loc}_v}
    H_{\mathrm{BK}}^1(K_v,T)/H^1(K_v,T)_{\mathrm{tor}}\right\}.
\]
In particular, since the Shafarevich--Tate group of Bloch--Kato type of an elliptic curve is finite and  $H_{ac}^1(K,W)$ is finite.
\end{proposition}

\begin{proof}
First, for the Bloch--Kato local condition, consider the exact sequence
\[
  0\rightarrow H^1_{\mathrm{BK}_v}(K,W)\rightarrow
  H^1_{\mathrm{BK}}(K,W)\rightarrow
  H^1_{\mathrm{BK}}(K_v,W)\rightarrow 0,
\]
and its dual sequence
\[
  0\rightarrow H^1_{\mathrm{BK}}(K,W^\ast)\rightarrow
  H^1_{\mathrm{BK}^v}(K,W^\ast)\rightarrow
  H^1(K_v,W^\ast)/H^1_{\mathrm{BK}}(K_v,W^\ast).
\]
Since $w=v\bar{v}$ splits in $K$ and because of the elliptic curve assumptions, the rightmost map of the first exact sequence is surjective. Here the splitting assumption ensures that both groups have corank $1$. Then, by the duality of the Poitou--Tate exact sequences, the map $H^1_{\mathrm{BK}^v}(K,W^\ast)\rightarrow H^1(K_v,W^\ast)/H^1_{\mathrm{BK}}(K_v,W^\ast)$ is in fact the zero map, hence
\[
  H^1_{\mathrm{BK}}(K,W^\ast)=H^1_{\mathrm{BK}^v}(K,W^\ast).
\]

Note that here we have already used the fact that the Bloch--Kato condition is self-dual.

Observe the maps $H^1_{\mathrm{BK}}(K,W)_{div}\xrightarrow{\alpha} H^1_{\mathrm{BK}}(K,W)\xrightarrow{\beta} H^1_{\mathrm{BK}}(K_v,W)$. By definition, $\alpha$ is injective. By assumption, $\beta\circ\alpha$ is surjective. Using homological algebra, we obtain $0\rightarrow\ker(\beta\circ\alpha)\rightarrow\ker\beta\rightarrow\coker\alpha \rightarrow 0$. From this exact sequence,
\[
  \#\ker\beta=\#\coker\alpha\cdot\#\ker(\beta\circ\alpha).
\]
Since, by definition, $\#\ker\beta=\#H^1_{\mathrm{BK}_v}(K,W)$ and $\coker\alpha=\Sha_{\mathrm{BK}}(W/K)$, we obtain
\[
  \#H^1_{\mathrm{BK}_v}(K,W)=\#\Sha_{\mathrm{BK}}(W/K)\cdot\#\ker(\beta\circ\alpha).
\]

We claim that $\#\ker(\beta\circ\alpha)=\#\delta_v$.

Indeed, consider the short exact sequence
\[
  0\rightarrow H^1_{\mathrm{BK}}(K,T)/H^1_{\mathrm{BK}}(K,T)_{tor}
  \rightarrow H^1_{\mathrm{BK}}(K_v,T)/H^1_{\mathrm{BK}}(K_v,T)_{tor}
  \rightarrow\delta_v\rightarrow 0.
\]
By our assumptions on the Mordell--Weil rank of $E/K$ (which is one, by generalized Gross--Zagier--Kolyvagin), the first two terms of the exact sequence are free $\mathbb{Z}_p$-modules of rank $1$.There are canonical isomorphisms
\[
\left(H^1_{\mathrm{BK}}(K,T)/H^1_{\mathrm{BK}}(K,T)_{\mathrm{tor}}\right)\otimes (\mathbb{Q}_p/\mathbb{Z}_p) \cong H^1(K_v,W)_{\mathrm{div}},
\]
and
\[
\left(H^1_{\mathrm{BK}}(K_v,T)/H^1_{\mathrm{BK}}(K_v,T)_{\mathrm{tor}}\right)\otimes (\mathbb{Q}_p/\mathbb{Z}_p) \cong H^1_{\mathrm{BK}}(K_v,W).
\]
After tensoring the above short exact sequence with $\mathbb{Q}_p/\mathbb{Z}_p$ and considering the left derived sequence, the term
\[
\mathrm{Tor}_1\!\left( H^1_{\mathrm{BK}}(K_v,T)/H^1_{\mathrm{BK}}(K_v,T)_{\mathrm{tor}},\; \mathbb{Q}_p/\mathbb{Z}_p \right)
\]
vanishes, because the quotient module is free and hence flat.

This completes the proof of the claim. Note that, so far, the proof has only discussed Selmer groups with the Bloch--Kato local condition. Next we discuss the comparison with the Selmer groups with the anticyclotomic local condition, thereby completing the proof.

By Corollary~\ref{WBKACalmostsame}, the Bloch--Kato and anticyclotomic local conditions for $W=E[p^\infty]$ differ only at $v,\bar{v}$. Comparing the two local conditions, we obtain the exact sequence
\[
  0\rightarrow H^1_{\mathrm{BK}_v}(K,W)\rightarrow H^1_{ac}(K,W)\xrightarrow{\alpha'}
  H^1(K_{\bar{v}},W)_{div}/H^1_{\mathrm{BK}}(K_{\bar{v}},W),
\]
and its dual exact sequence
\[
  0\rightarrow H^1_{ac^\ast}(K,W^\ast)\rightarrow H^1_{\mathrm{BK}^v}(K,W^\ast)
  \xrightarrow{\beta'}
  H^1_{\mathrm{BK}}(K_{\bar{v}},W^\ast)/H^1(K_{\bar{v}},W^\ast)_{tor}.
\]
Then, by the first exact sequence,
\[
  \#H^1_{ac}(K,W)=\#\ker\alpha'\cdot\#\operatorname{im}\alpha'.
\]
By the Poitou--Tate exact sequence theorem, $\#\operatorname{im}\alpha'=\#\coker\beta'$. However, as proved above, $H^1_{\mathrm{BK}}(K,W^\ast)=H^1_{\mathrm{BK}^v}(K,W^\ast)$, hence
\[
  \coker\beta'=
  \coker\left\{H^1_{\mathrm{BK}}(K,W^\ast)\rightarrow
  H^1_{\mathrm{BK}}(K_{\bar{v}},W^\ast)/H^1(K_{\bar{v}},W^\ast)_{tor}\right\}.
\]

By the Weil pairing, $W^\ast\cong T$. Note that $G_{K_v}=cG_{K_{\bar{v}}}c^{-1}$, which induces an isomorphism $H^1(K_v,T)\rightarrow H^1(K_{\bar{v}},W^\ast)$, $\phi\mapsto(\phi_c:g\mapsto c\phi(g)c^{-1})$. Substituting this isomorphism, we can identify $\coker\beta'$ with $\delta_v$.

Now, combining
\[
  \#H^1_{ac}(K,W)=\#H^1_{\mathrm{BK}_v}(K,W)\cdot\#\operatorname{im}\alpha',
\]
\[
  \#H^1_{\mathrm{BK}_v}(K,W)=\#\Sha_{\mathrm{BK}}(W/K)\cdot\#\delta_v,
\]
and
\[
  \#\operatorname{im}\alpha'=\#\coker\beta'=\#\delta_v,
\]
we obtain $\#H_{ac}^1(K,W)=\#\Sha_{\mathrm{BK}}(W/K)\cdot(\#\delta_v)^2$. The proposition is proved.
\end{proof}

\begin{remark}
Note that when $K_v\neq\Qp$, discussing the size of $\delta_v$ is meaningless, because it is not a finite group. Hence the assumption $F_w\cong\Qp$ in this article is necessary from the point of view of Proposition~\ref{prop:acandsha}.
\end{remark}

\begin{remark}
The proof is essentially the same as \cite[Proposition 3.2.1]{jetchev5birch}, except that we provide a more detailed explanation here.
\end{remark}

\subsubsection{The key formula}\label{subsection:keyformula}
This subsubsection translates Proposition~\ref{prop:acandsha} of the previous subsubsection and obtains the key formula~\eqref{equation:key} of the following proposition. This formula will be used to prove the Birch and Swinnerton-Dyer formula.

By the basic assumptions, $\mathbf{f}$ is a Hilbert modular form of weight $k=(2,2,\dots,2)$, connected with the elliptic curve $E$ through the modularity conjectures. The triple $(V,T,W)=(V_p(E),T_p(E),E[p^\infty])$ is as fixed above. By the convention of this section, the coefficient ring here is $L'=\Qp$, hence we write $\Qp$ below.

\begin{proposition}\label{prop:keyformula}
Notation as above, $E/F$ is the fixed semistable modular elliptic curve of this article, and $K/F$ is a totally imaginary quadratic extension. By the basic assumption of this article, $K_u\cong\Qp$ for $u\mid w$. We make the following additional assumptions:
\begin{itemize}
\item $r_{\mathrm{MW}}(E/K)=1$ and $\Sha(E/K)[p^\infty]$ is finite.
\item $E[p]$ is irreducible as a $G_K$-representation.
\end{itemize}
Assume that $P\in E(K)$ is a point of infinite order and $\omega$ is a N\'eron differential. Then Proposition~\ref{prop:acandsha} translates into the following formula:
\begin{equation}\label{equation:key}
  \#H_{ac}^1(K,W)=\#\Sha(E/K)[p^\infty]
  \cdot\left(
    \frac{\#\Zp/(\frac{1-a_w+p}{p}\log_\omega P)}
         {\#H^0(K_v,W)[E(K)\otimes\Zp:\Zp\cdot P]}
  \right)^2.
\end{equation}
\end{proposition}

\begin{remark}
This formula connects the Shafarevich--Tate group, the anticyclotomic Selmer group, the index of a Heegner point, and some error terms. In the proof of Theorem~\ref{theorem:maintheorem} one will see that, under Iwasawa-theoretic results such as the control theorem, this formula plays a key role in the proof of the Birch and Swinnerton-Dyer formula, which is the origin of its name.
\end{remark}

\begin{proof}
Recall that $E[p]$ is by assumption an irreducible $G_F$-representation. We further assume that it is irreducible as a $G_K$-representation. Then by Mordell-Weil rank assumption, $E(K)\otimes\Zp\cong H^1_{\mathrm{BK}}(K,T)\cong\Zp$.

On the other hand, locally, since by assumption $K_u\cong\Qp$ for $u\mid w$, by the structure theorem for elliptic curves over local fields, e.g. \cite[Proposition VII.6.3]{silverman2009arithmetic}, $E(K_u)/E(K_u)_{tor}\cong\Zp$. Hence $E(K_u)/E(K_u)_{tor}\otimes_{\Zp}\Qp/\Zp\cong\Qp/\Zp$. Therefore, by the nonvanishing of the map and by the fact that the image of a divisible group under a map is divisible, we obtain (in the case $u=v$) that $E(K)\otimes\Qp/\Zp\rightarrow E(K_v)\otimes\Qp/\Zp$ is surjective.

We now compute the size of $\#\delta_v$ in Proposition~\ref{prop:acandsha}.

Since $\delta_v$ is described in terms of cohomology groups with $T$-coefficients, we need to translate these cohomology groups. First, by what was said above, the global cohomology group $E(K)\otimes\Zp\cong H^1_{\mathrm{BK}}(K,T)\cong\Zp$ has been obtained. Similarly, for the local cohomology group,
\[
  E(K_v)/E(K_v)_{tor}\otimes_{\Zp}\Zp
  \cong H^1_{\mathrm{BK}}(K_v,T)/H^1_{\mathrm{BK}}(K_v,T)_{tor}.
\]
Then, by Proposition~\ref{prop:acandsha} of the previous subsubsection,
\[
  \#\delta_v=\#\coker\left\{H^1_{\mathrm{BK}}(K,T)\rightarrow
  H^1_{\mathrm{BK}}(K_v,T)/H^1_{\mathrm{BK}}(K_v,T)_{tor}\right\}.
\]
Substituting the above discussion, $\#\delta_v=\#[E(K_v)/E(K_v)_{tor}: E(K)\otimes\Zp]$. At this point one takes a quotient in $\Zp$, which is well defined.

Now assume that $P\in E(K)$ is an arbitrary point of infinite order. Using this point, one can continue to obtain
\[
  \#\delta_v=\frac{[E(K_v)/E(K_v)_{tor}:\Zp\cdot P]}
                   {[E(K)\otimes\Zp:\Zp\cdot P]}.
\]

This separates the local information (the numerator) from the global information (the denominator). Next, using geometric tools, we translate the local information into a form needed for the later proofs. Note that, by the assumption of good reduction at $w$, by Hensel's lemma there is a surjection $E(K_v)\rightarrow\tilde{E}(\kappa_v)$, where $\kappa_v$ is the residue field of $K_v$ and $\tilde{E}$ is the elliptic curve over $\kappa_v$ determined by $E$. Write $E^1(K_v)$ for the kernel of this map.

By the logarithm map of the formal group, $\log:E^1(K_v)\rightarrow p\Omega^1(E/\Zp)^\vee$, where $\Omega^1(E/\Zp)^\vee=\Hom_{\Zp}(\Omega^1(E/\Zp),\Zp)$. This formal group logarithm extends to $\log:E(K_v)/E(K_v)_{tor}\rightarrow\Omega^1(E/\Zp)^\vee$. Choose a N\'eron differential $\omega=\omega_E\in\Omega^1(E/\Zp)\otimes\Zp$ such that, as an element of $\operatorname{End}_F(E)$, every positive integer $n$ satisfies $n^\ast\omega=n\omega$. Now consider the map $\log$ evaluated at $\omega$. We obtain a $\Zp$-module homomorphism $\log_\omega:E(K_v)/E(K_v)_{tor}\otimes\Zp\rightarrow\Zp$ which maps $E^1(K_v)$ surjectively onto $p\Zp$. Moreover, since both the domain and the codomain of $\log_\omega$ are free $\Zp$-modules and this map is not the zero map, this map is injective. Then the index $[E(K_v)/E(K_v)_{tor}:\Zp\cdot P]$ can be computed on the right-hand side $\Zp$ via $\log_\omega$: $[E(K_v)/E(K_v)_{tor}:\Zp\cdot P]= \allowbreak [\log_\omega(E(K_v)/E(K_v)_{tor}):\log_\omega(P)]$. Since $\Zp$ is a discrete valuation ring,
\[
  [\log_\omega(E(K_v)/E(K_v)_{tor}):\log_\omega(P)]
  =\frac{\#(\Zp/\log_\omega P)}{\#(\Zp/\log_\omega(E(K_v)/E(K_v)_{tor}))}.
\]

Furthermore, since $E^1(K_v)$ is a subgroup of $E(K_v)/E(K_v)_{tor}$, and by the above discussion,
\[
  \#(\Zp/p\Zp)=\#
  \frac{\log_\omega(E(K_v)/E(K_v)_{tor})}{\log_\omega(E^1(K_v)\otimes\Zp)}
  \cdot\#\frac{\Zp}{\log_\omega(E(K_v)/E(K_v)_{tor})}.
\]
Since $E^1(K_v)$ is torsion-free, and by the injectivity of $\log_\omega$,
\[
  \#\left(\log_\omega(E(K_v)/E(K_v)_{tor})/\log_\omega(E^1(K_v)\otimes\Zp)\right)
  =\frac{\#(E(K_v)/E^1(K_v))}{\#(E(K_v)_{tor}\otimes\Zp)}.
\]

By the definition of $E^1(K_v)$, we have $E(K_v)/E^1(K_v)\cong\tilde{E}(\kappa_v)$. Hence $E(K_v)/E^1(K_v)\otimes\Zp\cong\tilde{E}(\kappa_v)[p^\infty]\cong \Zp/(1-a_w+p)$. Again, here we use $F_w\cong\Qp$.

For the denominator, $E(K_v)_{tor}\otimes\Zp=E(K_v)[p^\infty]=H^0(K_v,W)$.

Now, putting together the translation of the local information, we obtain
\[
  [E(K_v)/E(K_v)_{tor}:\Zp\cdot P]
  =\frac{\#\Zp/(\frac{1-a_w+p}{p}\log_\omega P)}{\#H^0(K_v,W)}.
\]
Therefore
\[
  \delta_v=\frac{\#\Zp/(\frac{1-a_w+p}{p}\log_\omega P)}
                 {\#H^0(K_v,W)[E(K)\otimes\Zp:\Zp\cdot P]}.
\]
By Proposition~\ref{prop:acandsha} of the previous subsubsection, the proposition follows.
\end{proof}

\subsection{Results over $L$: The Control Theorem}\label{section:controlthmsoverL}

This subsection proves the second main result of this section, the anticyclotomic control theorem~\ref{theorem:controltheorems}.

Recall that $L/K$ is an anticyclotomic $\Zp$-extension of $K$, $\Lambda=\Lambda_L$ is the Iwasawa algebra, $M=T\otimes\check{\Lambda}_L$, and $X_{ac}^{\Sigma}(M)=\Hom\left(H^1_{ac^{\Sigma}}(K,M),\Qp/\Zp\right)$.

In order to obtain the terms appearing in the Birch and Swinnerton-Dyer formula, consider the following proposition.

\begin{proposition}\label{prop:specialZpextension}
There exists a $\Zp$-extension $L/K$ such that a place $u$ of $K$ ramifies in $L$ if and only if $u\mid w$, that is, only $v$ and $\bar{v}$ ramify.
\end{proposition}

\begin{proof}
Consider the maximal abelian extension $k$ of $K$ unramified outside $p$. By the proof of Theorem~13.4 in Washington's book \cite{washington2012introduction}, $k$ is a finite extension of $K_\infty$. There, $\Gal(k/K)\cong\mathbb{A}_K^\times/\overline{K^\times U''}$, where $U''$ is the product of local units away from $p$. Now, further, let $U'''=\prod_{u\nmid w}U_u$ be the product of the local unit groups at all places $u\nmid w$, i.e.\ $U'''=\prod_{u\nmid w}\mathcal{O}_{K_u}^\times\times \prod_{u\mid w}\{1\}$. This is the product of local units away from $w$.

Then consider the intermediate field $k'=\mathbb{A}_K^\times/\overline{K^\times U'''}$ of $k/K$. This intermediate field is unramified away from $w$.

Furthermore, by basic algebraic number theory, $\prod_{u\nmid w}U_u$ is a finite group times $\Zp^a$, where $a=\sum_{u\mid p,\ u\nmid w}[K_u:\Qp]$ is the sum of the local degrees at the $p$-adic places of $K$ not dividing $w$. In other words, $k'$ must have a subfield whose Galois group over $K$ is isomorphic to $\Zp^{1+\delta+d-a}$.

Now consider the field $k'':=k'\cap K^-$, since the anticyclotomic part of $\Gal(K_\infty/K)$ is nontrivial of $\Zp$-rank at least $d\ge 1$, $\Gal(k''/K)$ is certainly nontrivial. This shows that $k''/K$ is a nontrivial anticyclotomic extension unramified away from $w$. At this point take $L$ to be an intermediate field of $k''/K$ with $\Gal(L/K)\cong\Zp$. This proves the proposition.
\end{proof}

Below we fix a $\Zp$-extension $L/K$ satisfying the proposition above, that is, only $w$ ramifies in it. $\Lambda$ is the Iwasawa algebra determined by $L/K$, and $V,T,W,M$ are the arithmetic objects defined in this section. At this point $L'$ is $\Qp$.

Let $S$ be a finite set of places of $K$ containing all the places where $V=V_p(E)$ ramifies and all the places above $p$. Let $S_p$ be the subset of $S$ consisting of the places not dividing $p$. Fix a subset $\Sigma\subseteq S_p$, in particular, $\Sigma$ may be empty.

Write $\mathscr{A}:=\{u\text{ a place of }K:\ u\in S_p-\Sigma\text{ and }u\text{ splits}\}$.

\begin{theorem}[Anticyclotomic control theorem]\label{theorem:controltheorems}
As a $\Lambda$-module, $X_{ac}^{\Sigma}(M)$ is a $\Lambda$-torsion module. Moreover, if $f_{ac}^\Sigma(T)$ is a generator of the characteristic ideal of $X_{ac}^{\Sigma}(M)$, then
\[
  \#\Zp/f_{ac}^\Sigma(0)=\#H^1_{ac}(K,W)\times C^\Sigma(W),
\]
where
\[
  C^\Sigma(W)=\#H^0(K_v,W)\#H^0(K_{\bar{v}},W)
  \cdot\prod_{u\in\mathscr{A}}\#H^1_{ur}(K_u,W)
  \cdot\prod_{u\in\Sigma}\#H^1(K_u,W).
\]
\end{theorem}

\begin{remark}
The proof of Theorem~\ref{theorem:controltheorems} is entirely an imitation of Theorem~3.3.1 of Jetchev--Skinner--Wan \cite{jetchev5birch}, whose idea comes from the classical article of Greenberg \cite{greenberg1999iwasawa}. In the present article, by comparison,  we provide the necessary details and clarify several points that are only briefly discussed there.
\end{remark}

For a finite set $S'$ of places of $K$, one can define the objects
\[
  P(M.S')=\prod_{u\in S'}\frac{H^1(K_u,M)}{H^1_{ac}(K_u,M)}
  \quad\text{and}\quad
  P(W.S')=\prod_{u\in S'}\frac{H^1(K_u,W)}{H^1_{ac}(K_u,W)}.
\]
Consider the two localization maps
\[
  H^1(K^S/K,M)\xrightarrow{loc_{S,M}}P(M.S)
  \quad\text{and}\quad
  H^1(K^S/K,W)\xrightarrow{loc_{S,W}}P(W.S),
\]
where $K^S$ is the maximal extension of $K$ unramified outside $S$. It can be proved that the Selmer groups defined above are subgroups of $H^1(K^S/K,W)$.

\begin{lemma}
$H^1(K^S/K,M)$ can be expressed in the language of Selmer structures. Denote this Selmer structure by $A$. If $u\in S$, then $H^1_A(K_u,M)=H^1(K_u,M)$, while if $u\notin S$, then $H^1_A(K_u,M)=H^1_{ur}(K_u,M)$.

On the other hand, $H^1(K^S/K,W)$ can be expressed in the language of Selmer structures. Denote this Selmer structure by $B$. Then, if $u\in S$, $H^1_B(K_u,W)=H^1(K_u,W)$, if $u\notin S$, $H^1_B(K_u,W)=H^1_{ur}(K_u,W)$.
\end{lemma}

\begin{proof}
Indeed, by the definition of $S$, outside $S$ the actions of $G_K$ on $M$ and $W$ are unramified, so $I_u$, for $u\notin S$, acts trivially. Then any element $x\in H^1(K,M)$ whose restriction to $\prod_{u\notin S}H^1(K_u,M)/H^1_{ur}(K_u,M)=\prod_uH^1(I_u,M)$ is trivial lies in $H^1(G_K/(\prod_uI_u),M^{(\prod_uI_u)})$. But since $I_u$ acts trivially, $M^{(\prod_uI_u)}=M$, and $G_K/(\prod_uI_u)=\Gal(K^S/K)$. The lemma follows.
\end{proof}

\begin{lemma}
The $A$-Selmer structure of $M$ can be written as $ac^S$. The $B$-Selmer structure of $W$ can be written as $ac^S$.
\end{lemma}

\begin{proof}
In fact, one only needs to pay attention to the situation outside $S$, because both $ac^S$ and $A$ take the relaxed local condition at the places of $S$. So the assertion reduces to: the local condition of $ac$ at the places outside $S$ is the unramified local condition. Since $ac$ is unramified at the places $u\nmid p\infty$ that split, one only needs to prove that, at $u\mid p\infty$ or at non-split $u$, the unramified local condition coincides with the anticyclotomic local condition, i.e.\ is trivial. Since $S$ contains the places above $p$ that ramify, the case $u\mid p$ need not be discussed.

Case one: $u\mid\infty$. Since the whole cohomology group is the cohomology of a $p$-torsion module of a group of order two, it is trivial. Hence the unramified local condition at the infinite place, as a subgroup of the trivial group, is also trivial.

Case two: $u$ is non-split in $K/F$ (i.e.\ inert or ramified). Then
\[
  H^1_{ur}(K_u,M)=H^1(K_u^{ur}/K_u,M^{I_u}).
\]
Since $M=T\otimes\Lambda^\vee$, and $I_u$ acts trivially on $T$, and $u$ is unramified in $L/K$ (i.e.\ $I_u$ acts trivially on $\Lambda$), we have $M^{I_u}=M$. Since $M$ is a $p$-torsion group and $\Gal(K_u^{ur}/K_u)\cong\hat{\mathbb{Z}}$, $H^1_{ur}(K_u,M)=H^1(K_u^{ur}/K_u,M^{I_u})\cong H^1(\Zp,M)$. Now write $M=\varinjlim T\otimes(\Lambda/p^n)^\vee$ as a direct limit of free finite corank $\Zp$-modules. In fact it suffices to prove $H^1(\Zp,T\otimes(\Lambda/p^n)^\vee)=0$. However, by a well-known result of Galois cohomology: ``if $C$ is a free finite corank $\Gamma\cong\Zp$-module with $H^0(\Gamma,C)$ finite, then $H^1(\Gamma,C)=0$.'' Since $u$ is inert, by class field theory it splits completely in $L/K$, so the projection of the group generated by Frobenius to the $p$-part acts trivially on $\Lambda$. Hence $H^0(\Zp,T\otimes(\Lambda/p^n)^\vee)=T^{G_u}\otimes(\Lambda/p^n)^\vee=0$. This is because $u\nmid p$ and by Proposition~\ref{prop:Wfinite}.

This completes the proof for $M$.

For $W$, the idea is exactly the same. In the case of $W$, one should note that, when defining the anticyclotomic local condition of $W$, one uses the Bloch--Kato local condition. But, by the remark above, since $u$ is unramified for $W$, the Bloch--Kato condition is the unramified local condition.

Case one is the same reasoning. Case two is even simpler, since $W$ itself is a free finite corank $\Zp$-module, and $E(K_u)[p^\infty]$ is finite.
\end{proof}

By this lemma, the localization maps above can be extended to
\[
  0\rightarrow H^1_{ac}(K,M)\rightarrow H^1(K^S/K,M)\xrightarrow{loc_{S,M}}P(M.S)
\]
and
\[
  0\rightarrow H^1_{ac}(K,W)\rightarrow H^1(K^S/K,W)\xrightarrow{loc_{S,W}}P(W.S).
\]

\subsubsection{The localization maps $loc_{S,W}$ and $loc_{S,M}$ are surjective}

\begin{proposition}\label{prop:surjective}
The localization maps $loc_{S,W}$ and $loc_{S,M}$ are surjective.
\end{proposition}

\begin{proof}
First note that, by the notation $M=T\otimes\Lambda^\vee$, by the Weil pairing and Pontryagin duality,
\[
  M^\ast=\Hom_{cts}(M,\Qp/\Zp(1))
  =\Hom(T,\Zp(1))\otimes\Hom_{cts}(\Lambda^\vee,\Qp/\Zp)=T\otimes\Lambda.
\]
At this point, as a $G_K$-module, $M^\ast$ carries the action $\rho\otimes\Psi$.

By the Poitou--Tate exact sequence, $\coker(loc_{S,W})$ can be identified with a quotient of $H^1_{ac^\ast}(K,M^\ast)$. Hence it suffices to prove that $H^1_{ac^\ast}(K,M^\ast)$ (or a larger group) vanishes, then the localization map $loc_{S,W}$ is surjective.

Consider  the long exact sequence induced by the short exact sequence $0\rightarrow M^\ast\xrightarrow{\times(\gamma-1)}M^\ast\rightarrow T\rightarrow 0$:
\[
  H^1(K^S/K,M^\ast)\xrightarrow{\times(\gamma-1)}
  H^1(K^S/K,M^\ast)\rightarrow H^1(K^S/K,T).
\]
Then $H^1(K^S/K,M^\ast)/(\gamma-1)H^1(K^S/K,M^\ast)\hookrightarrow H^1(K^S/K,T)$.

\textbf{Claim:} the injection above induces a map
\[
  H^1_{ac^\ast}(K,M^\ast)/(\gamma-1)H^1_{ac^\ast}(K,M^\ast)
  \hookrightarrow H^1_{(ac)_{\bar{v}}}(K,T),
\]
and this map is injective.

\textbf{Proof of the claim:} first observe the following diagram
\[
  \begin{tikzcd}
  {H^1_{ac^\ast}(K,M^\ast)} \arrow[d]
    & {H^1_{ac^\ast}(K,M^\ast)} \arrow[d]
    & {H^1_{(ac)_{\bar{v}}}(K,T)} \arrow[d] \\
  {H^1(K^S/K,M^\ast)} \arrow[d] \arrow[r]
    & {H^1(K^S/K,M^\ast)} \arrow[d] \arrow[r]
    & {H^1(K^S/K,T)} \arrow[d] \\
  {\prod_{u\in S}\frac{H^1(K_u,M^\ast)}{H^1_{ac^\ast}(K_u,M^\ast)}}
    \arrow[r, dashed]
    & {\prod_{u\in S}\frac{H^1(K_u,M^\ast)}{H^1_{ac^\ast}(K_u,M^\ast)}}
    \arrow[r, dashed]
    & {\prod_{u\in S-\bar{v}}\frac{H^1(K_u,T)}{H^1_{ac}(K_u,T)}
       \times H^1(K_{\bar{v}},T)}
  \end{tikzcd}
\]
The map we want is in fact the map between the kernels of the three vertical maps (although one cannot directly use the lemma above, by the same idea one can show that the vertical maps are well defined), and this depends on the existence of the dashed maps. This in turn reduces to the local maps at $u\in S$. By Tate's local duality and the definition of the dual Selmer structure, after taking the dual of each $u\in S-\{\bar{v}\}$ in the dashed row, it suffices to prove the existence of the map $H^1_{ac^\ast}(K_u,W)\rightarrow H^1_{ac}(K_u,M)\xrightarrow{\times(\gamma-1)} H^1_{ac}(K_u,M)$.

For $u\neq v$, $u\in S-\{\bar{v}\}$, by Corollary~\ref{cor:ac=acstar}, the anticyclotomic condition at this point coincides with its dual anticyclotomic condition, i.e.\ $ac^\ast$ can be replaced by $ac$. At this point, by the definition of the local conditions, the map exists. Hence the first part of the claim holds. Note that when $u\mid p$ and $u\neq v,\bar{v}$, their local conditions are both the Bloch--Kato ones, so the existence of the map is reasonable.

For $u=v$, since $H^1_{ac}(K_u,M)=0$, the map obviously exists.

The second part of the claim comes from a diagram chase: assume that $c\in H^1_{(ac)^\ast}(K,M^\ast)$ has trivial image in $H^1_{(ac)_{\bar{v}}}(K,T)$,  then $c=(\gamma-1)d$ for some $d\in H^1(K^S/K,M^\ast)$. Then, by the commutative diagram above, under the restriction map $(\gamma-1)d=0\in H^1(K_{\bar{v}},M^\ast)$. But, by the long exact sequence induced by the short exact sequence, $d\in H^0(K_{\bar{v}},T)\subseteq V^{G_{K_{\bar{v}}}}$. However, $V$ is pure of weight $-1$, in particular its weight is not $0$ or $1$, so $V^{G_{K_{\bar{v}}}}=0$. Hence $d$ is trivial.

Proof of the claim is done.

Then, it suffices to prove $H^1_{(ac)_{\bar{v}}}(K,T)=0$, then by Nakayama's lemma $H^1_{ac^\ast}(K,M^\ast)=0$. But, by $H^1_{(ac)_{\bar{v}}}(K,T)\otimes\Qp/\Zp\hookrightarrow H^1_{(ac)_{\bar{v}}}(K,W)$ (this map comes from taking the direct limit of the injections $H^1(K,T)/p^n\hookrightarrow H^1(K,T/p^n)$) and by Proposition~\ref{prop:acandsha}, $H^1_{(ac)_{\bar{v}}}(K,T)$ must be finite, because $H^1_{(ac)_{\bar{v}}}(K,W)$ is a subgroup of the finite group $H^1_{ac}(K,W)$. But, under the irreducibility condition, $H^1_{(ac)_{\bar{v}}}(K,T)$, as a submodule of the torsion-free module $H^1(K,T)$, is torsion-free. Hence it must be trivial.

For the second part, the surjectivity of $loc_{S,W}$ is proved similarly: identify $\coker(loc_{S,W})$ with a quotient of $H^1_{ac^\ast}(K,T)$, and consider $0\rightarrow H^1_{(ac^\ast)_{\bar{v}}}(K,T)\rightarrow H^1_{ac^\ast}(K,T)\rightarrow H^1_{ac^\ast}(K_{\bar{v}},T)$, where the first term is already $0$ by the previous proof, and the middle term is torsion-free. By the definition of the dual Selmer structure, $H^1_{ac^\ast}(K_{\bar{v}},T)= (H^1(K_{\bar{v}},W)/H^1_{ac}(K_{\bar{v}},W))^\ast$, but, by definition, $H^1_{ac}(K_{\bar{v}},W)=H^1(K_{\bar{v}},W)_{div}$, so by the cofinitely generated of $H^1(K_{\bar{v}},W)$ we get that $H^1_{ac^\ast}(K_{\bar{v}},T)$ is finite. Hence $H^1_{ac^\ast}(K,T)$, as a subgroup of a finite group, is finite, together with torsion-freeness, it is $0$.

This completes the proof of the proposition.
\end{proof}

\begin{remark}
As a trivial corollary, when $P(M;S)$ is replaced by $P(M;S-\Sigma)$, the localization map is still surjective.
\end{remark}

\subsubsection{Triviality of the coinvariants}

\begin{proposition}\label{prop:coinvariant}
$H^1(K^S/K,M)_\Gamma=0$ and $H^1_{ac^\Sigma}(K,M)_\Gamma=0$.
\end{proposition}

\begin{proof}
The long exact sequence induced by the short exact sequence $0\rightarrow W\rightarrow M\xrightarrow{\times(\gamma-1)}M\rightarrow 0$ gives $H^1(K^S/K,M)_{\Gamma}\hookrightarrow H^2(K^S/K,W)$. By this injective map, it suffices to prove $H^2(K^S/K,W)=0$.

Consider $\ker\{H^2(K^S/K,W)\rightarrow\prod_{v\in S}H^2(K_u,W)\}$. Since $H^2(K_u,W)$ is dual to $H^0(K_u,T)$, which was shown trivial earlier,
\[
  \ker\{H^2(K^S/K,W)\rightarrow\prod_{v\in S}H^2(K_u,W)\}=H^2(K^S/K,W).
\]

And,by the $K^S/K$-version of the Poitou--Tate global duality theorem (e.g.\ Theorem~I.4.10 of Milne's book \cite{milne2006arithmetic}), the kernel
\[
\ker\left\{H^2(K^S/K,W)\rightarrow\prod_{v\in S}H^2(K_u,W)\right\}
\]
is dual to
\[
\ker\left\{H^1(K^S/K,T)\rightarrow\prod_{v\in S}H^1(K_u,T)\right\}.
\]
hence it suffices to prove that the latter is trivial. And, by the reasoning used repeatedly in the previous subsubsection, this module is both torsion-free (as a submodule of a torsion-free module) and finite (its tensor product with a divisible group embeds into a finite group), hence it must be trivial. This proves the first assertion of the proposition.

Since the surjectivity was proved in the previous subsubsection, the following exact sequence is now completed:
\[
  0\rightarrow H^1_{ac}(K,M)\rightarrow H^1(K^S/K,M)\rightarrow P(M.S)\rightarrow 0.
\]

Consider the following commutative diagram:
\[
  \begin{tikzcd}
  0 \arrow[r] & {H^1_{ac^\Sigma}(K,M)} \arrow[r] \arrow[d, "\times(\gamma-1)"]
    & {H^1(K^S/K,M)} \arrow[r] \arrow[d, "\times(\gamma-1)"]
    & P(M.S-\Sigma) \arrow[r] \arrow[d, "\times(\gamma-1)"] & 0 \\
  0 \arrow[r] & {H^1_{ac^\Sigma}(K,M)} \arrow[r]
    & {H^1(K^S/K,M)} \arrow[r] & P(M.S-\Sigma) \arrow[r] & 0
  \end{tikzcd}
\]
and use the snake lemma. One obtains
\begin{multline*}
0\rightarrow H^1_{ac^\Sigma}(K,M)^\Gamma\rightarrow
H^1(K^S/K,M)^\Gamma\xrightarrow{\alpha}P(M.S-\Sigma)^\Gamma\\
\rightarrow H^1_{ac^\Sigma}(K,M)_\Gamma\rightarrow H^1(K^S/K,M)_\Gamma\rightarrow
P(M.S-\Sigma)_\Gamma\rightarrow 0.
\end{multline*}
To examine the map $\alpha$, consider the following commutative diagram:
\[
  \begin{tikzcd}
  {H^1(K^S/K,W)} \arrow[r] \arrow[d] & P(W.S-\Sigma) \arrow[d] \\
  {H^1(K^S/K,M)^\Gamma} \arrow[r, "\alpha"] & P(M.S-\Sigma)^\Gamma
  \end{tikzcd}
\]

\textbf{Claim:} $\alpha$ is surjective. By Proposition~\ref{prop:surjective} of the previous subsubsection, the upper map of the commutative diagram is surjective. Hence it suffices to prove that the right vertical map of the diagram is surjective, which reduces to each local cohomology group. Now consider the following commutative diagram:
\[
  \begin{tikzcd}
  0 \arrow[r] & {H^1_{ac}(K_u,W)} \arrow[d] \arrow[r]
    & {H^1(K_u,W)} \arrow[d, "\beta"] \arrow[r]
    & {\frac{H^1(K_u,W)}{H^1_{ac}(K_u,W)}} \arrow[d, "\theta"] \arrow[r] & 0 \\
  0 \arrow[r] & {H^1_{ac}(K_u,M)^\Gamma} \arrow[r]
    & {H^1(K_u,M)^\Gamma} \arrow[r]
    & {(\frac{H^1(K_u,M)}{H^1_{ac}(K_u,M)})^\Gamma} \arrow[r] & H
  \end{tikzcd}
\]
where $H=H^1(\Gamma,H^1_{ac}(K_u,M))$. First compute $H=H^1(\Gamma,H^1_{ac}(K_u,M))$.

Since $\Gamma\cong\Zp$, $H^1(\Gamma,H^1_{ac}(K_u,M))=H^1_{ac}(K_u,M)/(\gamma-1)H^1_{ac}(K_u,M)$, which can be regarded as part of the long exact sequence induced by a short exact sequence, hence it embeds into $H^2(K_u,W)$, which is trivial. Hence $H$ is trivial.

Moreover, since $\beta$ can also be regarded as part of the long exact sequence induced by a short exact sequence, it is automatically surjective. By the snake lemma, the map $\theta$ is surjective. Since the required surjectivity is equivalent to the surjectivity of each local map $\theta$, this proves the claim.

Since the claim holds, $\alpha$ is surjective, so $H^1_{ac^\Sigma}(K,M)_\Gamma\rightarrow H^1(K^S/K,M)_\Gamma$ is injective, and the latter is trivial by the first assertion of this proposition. Hence the former is also trivial. This completes the proof of the proposition.
\end{proof}

\subsubsection{Computing $\#\ker(r)$ and $\#H^1_{ac^\Sigma}(K,M)^\Gamma$}

Write $r:P(W;S-\Sigma)\rightarrow P(M;S-\Sigma)^\Gamma$ for the map.

\begin{proposition}\label{prop:numberofkernel}
The size of the kernel of the map $r$ is
\[
  \#\ker(r)=\#H^0(K_v,W)\#H^0(K_{\bar{v}},W)
  \cdot\prod_{u\in S_p-\Sigma\text{ and }u\text{ splits}}c_u^{p}(W),
\]
where, for $u\nmid p$, $c_u^{p}(W):=[H^1_{ur}(K_u,W):H^1_{\mathrm{BK}}(K_u,W)]$ is the $p$-part of the local Tamagawa number.
\end{proposition}

\begin{proof}
By the definition of $r$, computing $\#\ker(r)$ amounts to computing each $\#\ker(r_u)$ for
\[
  r_u:\frac{H^1(K_u,W)}{H^1_{ac}(K_u,W)}\rightarrow
  \left(\frac{H^1(K_u,M)}{H^1_{ac}(K_u,M)}\right)^\Gamma
  =\frac{H^1(K_u,M)^\Gamma}{H^1_{ac}(K_u,M)^\Gamma},
\]
where the last equality comes from the discussion of the previous subsubsection. Let $u'$ denote the restriction of the place $u$ of $K$ to $F$.

First observe the commutative diagram
\[
  \begin{tikzcd}
    & & M^{G_u}/(\gamma-1)M^{G_u} \arrow[d, hook]
    & (M^{I_u}/(\gamma-1)M^{I_u})^{G_{\kappa_u}} \arrow[d, hook] & \\
  0 \arrow[r] & {H^1_{ur}(K_u,W)} \arrow[d] \arrow[r]
    & {H^1(K_u,W)} \arrow[d] \arrow[r]
    & {H^1(I_u,M)^{G_{\kappa_u}}} \arrow[d] \arrow[r] & 0 \\
  0 \arrow[r] & {H^1_{ur}(K_u,M)^\Gamma} \arrow[r]
    & {H^1(K_u,M)^\Gamma} \arrow[r] & {H^1(I_u,M)^\Gamma} &
  \end{tikzcd}
\]
and discuss it case by case.

\textbf{Case 1.a: $u\nmid p$, $W$ is ramified at $u$, and $u'$ splits in $K$.} Recall that $I_u$ acts on $\Lambda^\vee$ through $\Psi^{-1}$. Since $u\nmid p$, $I_u$ acts trivially on $\Lambda^\vee$, so $M^{I_u}$ is still $(\gamma-1)$-divisible. Hence we get the following commutative diagram:
\[
  \begin{tikzcd}
  {H^1(K_u,W)/H^1_{ur}(K_u,W)} \arrow[d] \arrow[r]
    & {H^1(I_u,M)^{G_{\kappa_u}}} \arrow[d, hook] \\
  {H^1(K_u,M)^\Gamma/H^1_{ur}(K_u,M)^\Gamma} \arrow[r]
    & {H^1(I_u,M)^\Gamma}
  \end{tikzcd}
\]
First, the upper map is already an isomorphism, so the map from the upper left to the lower right is injective. Hence the left vertical map is injective. Since $H^1_{\mathrm{BK}}(K_u,W)\subseteq H^1_{ur}(K_u,W)$ by the natural inclusion, $\ker(r_u)=H^1_{ur}(K_u,W)/H^1_{\mathrm{BK}}(K_u,W)$. This completes the proof of Case 1.a.

\textbf{Case 1.b: $u\nmid p$, $W$ is ramified at $u$, and $u'$ does not split in $K$.} At this point, by definition, the anticyclotomic local conditions of $M$ and $W$ are trivial. Hence $\ker(r_u)=M^{G_u}/(\gamma-1)M^{G_u}$. Since $u'$ does not split, $u$ splits completely in the anticyclotomic $\Zp$-extension, which means that the whole local Galois group $G_u$ acts trivially on $\Lambda^\vee$. Hence, by divisibility, $\ker(r_u)$ is trivial.

\textbf{Case 2.a: $u\nmid p$, $W$ is unramified at $u$, and $u'$ splits in $K$.} The discussion here is the same as in Case 1.a, and one obtains $\ker(r_u)=H^1_{ur}(K_u,W)/H^1_{\mathrm{BK}}(K_u,W)$. But now $W$ is unramified. By a previous proposition, at unramified places the Bloch--Kato local condition coincides with the unramified local condition, so $\ker(r_u)=0$.

\textbf{Case 2.b: $u\nmid p$, $W$ is unramified at $u$, and $u'$ does not split in $K$.} The discussion here is the same as in Case 1.b, and $\ker(r_u)=0$.

\textbf{Case 3.a: $u=\bar{v}$.} By the anticyclotomic local condition of $M$, $H^1_{ac}(K_{\bar{v}},M)=H^1(K_{\bar{v}},M)$, hence $\ker(r_{\bar{v}})=H^1(K_{\bar{v}},W)/H^1(K_{\bar{v}},W)_{div}$. By Tate's local duality, $\ker(r_{\bar{v}})$ is dual to $H^1(K_{\bar{v}},T)_{tor}$. But, by the long exact sequence induced by $0\rightarrow T\rightarrow V\rightarrow W\rightarrow 0$,
\[
  H^1(K_{\bar{v}},T)_{tor}=\ker\{H^1(K_{\bar{v}},T)\rightarrow H^1(K_{\bar{v}},V)\}\\
  =\coker\{H^0(K_{\bar{v}},V)\rightarrow H^0(K_{\bar{v}},W)\}.
\]
But $H^0(K_{\bar{v}},V)=0$. Hence finally $\#\ker(r_{\bar{v}})=\#H^0(K_{\bar{v}},W)$.

\textbf{Case 3.b: $u=v$.} By definition, the anticyclotomic local conditions of $M$ and $W$ are trivial. Hence $\ker(r_v)\cong M^{G_{K_v}}/(\gamma-1)M^{G_{K_v}}$. Consider $P_v=\ker\Psi|_{G_{K_v}}$, and write $\Gamma_v:=G_{K_v}/P_v\hookrightarrow\Gamma$. By the choice of $L/K$, $v$ must ramify in $L$. Hence $I_v$ is nontrivial and has finite index in $\Gamma$, so $\Gamma_v$ also has finite index in $\Gamma$. Take a topological generator $\gamma_v$ of $\Gamma_v$. Since $P_v$ acts trivially on $\Lambda$, $M^{P_v}=T^{P_v}\otimes\Lambda$. Now assume that $\Gamma_v^{p^t}$ acts trivially on $T^{P_v}$ (this holds because $I_v$ has finite index in $\Gamma_v$), then
\[
  M^{G_{K_v}}\subseteq (M^{P_v})^{\Gamma_v^{p^t}}=T^{P_v}\otimes\Lambda_{\Gamma_v^{p^t}}.
\]

\textbf{Claim:} $\#T^{P_v}$ is finite. Then, by the claim, $M^{G_{K_v}}$ is finite, and at this point
\[
  \#M^{G_{K_v}}/(\gamma-1)M^{G_{K_v}}=\#M^{G_{K_v}}[\gamma-1]
  =\#M[\gamma-1]^{G_{K_v}}=\#H^0(K_v,W).
\]

\textbf{Proof of the claim:} if $T^{P_v}$ were not finite, then $V^{P_v}\neq 0$. Since $V$ is a two-dimensional semistable $p$-adic representation, $V$ has two possible forms: (1) $V$ is crystalline. (2) $V$ is not crystalline, in which case $0\rightarrow\Qp(\chi_{cyc}\alpha)\rightarrow V\rightarrow\Qp(\alpha)\rightarrow 0$ with the character $\alpha$ unramified.

In case (1), $V^{P_v}$ is a sum of one-dimensional crystalline representations of weight $-1$. Such crystalline representations are all of the form $\chi_{cyc}^a\beta$, where $\beta$ is unramified. One can prove that $\chi_{cyc}^a\beta|_{I_v}=\beta|_{I_v}$, which shows $a\equiv 0\pmod{p-1}$. Since $a$ is the Hodge--Tate weight of $V$, $a=0$ due to elliptic curves. This shows that $\beta$ factors through a finite group $\Gamma_v/\Psi(I_v)$, but $\beta(\Frob_v)$ is then a root of unity, whereas the weight $-1$ condition forces $\beta(\Frob_v)$ to be a Weil number of absolute value $p^{-1/2}$, a contradiction.

In case (2), $V|_{P_v}$ also satisfies a short exact sequence similar to that of $V$, then $V^{P_v}\neq 0$ again shows that $\alpha$ factors through $\Gamma_v$, which is again a contradiction. Hence the claim holds.

\textbf{Case 3.c: $u\mid p$, $u\neq v,\bar{v}$.} At this point, by the choice of the $\Zp$-extension $L/K$, $u$ is unramified in $L/K$. By the last case in the proof of Lemma~4.6 of another article of Greenberg \cite{MR1860044}, the kernel $\ker(r_u)$ here is trivial. Note that the language there differs from ours, but after using Shapiro's lemma the two are consistent.

Combining all seven cases, the proposition is proved.
\end{proof}

We now use the computed $\#\ker(r)$ to compute $\#H^1_{ac^\Sigma}(K,M)^\Gamma$.

\begin{proposition}\label{prop:Mfinite}
$\#H^1_{ac^\Sigma}(K,M)^\Gamma$ is finite, and
\[
  \#H^1_{ac^\Sigma}(K,M)^\Gamma=\#H^1_{ac}(K,W)\cdot C^\Sigma(W).
\]
\end{proposition}

\begin{proof}
First consider the following commutative diagram
\[
  \begin{tikzcd}[column sep=small, row sep=small]
  0 \arrow[r] & {H^1_{ac^\Sigma}(K,W)} \arrow[r] \arrow[d, "s"]
    & {H^1(K^S/K,W)} \arrow[r] \arrow[d, "h"]
    & P(W.S-\Sigma) \arrow[r] \arrow[d, "r"] & 0 \\
  0 \arrow[r] & {H^1_{ac^\Sigma}(K,M)^\Gamma} \arrow[r]
    & {H^1(K^S/K,M)^\Gamma} \arrow[r] & P(M.S-\Sigma)^\Gamma &
  \end{tikzcd}
\]
Using the snake lemma, one obtains $0\rightarrow\ker s\rightarrow\ker h\rightarrow\ker r\rightarrow\coker s\rightarrow \coker h\rightarrow\coker r$. By the previous discussion, $h$ itself comes from part of a long exact sequence induced by a short exact sequence, so it is naturally surjective, i.e.\ $\coker h=0$. Substituting $\coker h=0$, the exact sequence gives
\[
  \frac{\#\coker s}{\#\ker s}=\frac{\#\ker h}{\#\ker r}.
\]

First, from the map $s$ itself,
\[
  \frac{\#H^1_{ac^\Sigma}(K,M)^\Gamma}{\#H^1_{ac^\Sigma}(K,W)}
  =\frac{\#\coker s}{\#\ker s}.
\]

Combining these,
\[
  \frac{\#H^1_{ac^\Sigma}(K,M)^\Gamma}{\#H^1_{ac^\Sigma}(K,W)}
  =\frac{\#\ker h}{\#\ker r}.
\]

By the long exact sequence induced by the short exact sequence, $\#\ker(h)=\#M^{G_K}/(\gamma-1)M^{G_K}=\#(M[\gamma-1])^{G_K}=\#W^{G_K}$. But, by the irreducibility assumption, $\#W^{G_K}=1$.

Hence $\#H^1_{ac^\Sigma}(K,M)^\Gamma=\#H^1_{ac^\Sigma}(K,W)\cdot\#\ker(r)$. Substituting the computed $\#\ker(r)$,
\[
    \#H^1_{ac^\Sigma}(K,M)^\Gamma
    =\#H^1_{ac^\Sigma}(K,W)\cdot\#H^0(K_v,W)\#H^0(K_{\bar{v}},W)
    \cdot\prod_{u\in S_p-\Sigma,\ u\text{ splits}}c_u^{p}(W).
\]

Moreover, by the short exact sequence $0\rightarrow H^1_{ac}(K,W)\rightarrow H^1_{ac^\Sigma}(K,W)\rightarrow P(W;\Sigma)\rightarrow 0$ (note that one cannot directly cite the previous proposition here, but by the same idea this exact sequence still holds), and since $u\in\Sigma$ implies $u\nmid p$, we have $H^1_{ac}(K_u,W)=0$. Hence
\[
  \#H^1_{ac^\Sigma}(K,W)=\#H^1_{ac}(K,W)\cdot\prod_{u\in\Sigma}\#H^1(K_u,W).
\]

Combining the formulas above, one obtains
\[
  \#H^1_{ac^\Sigma}(K,M)^\Gamma=\#H^1_{ac}(K,W)\cdot C^\Sigma(W).
\]

The proposition is proved.
\end{proof}

\subsubsection{Proof of the control theorem}

First we prove the torsion property of $X_{ac}^\Sigma(M)$. By definition, $X_{ac}^\Sigma(M)=\Hom(H^1_{ac^\Sigma}(K,M),\Qp/\Zp)$. Consider $X_{ac}^\Sigma(M)_\Gamma=\Hom(H^1_{ac^\Sigma}(K,M)^\Gamma,\Qp/\Zp)$. By Proposition~\ref{prop:Mfinite}, this is a finite $\Lambda$-module. By the algebraic properties of $\Lambda$-modules, the assertion follows.

Next we compute $\#\Zp/f_{ac}^\Sigma(0)$.

\begin{proposition}
$H^1_{ac^\Sigma}(K,M)$ has no proper finite-index submodule. In other words, $X_{ac}^\Sigma(M)$ has no nontrivial pseudo-null submodule.
\end{proposition}

\begin{proof}
Assume that $X$ is a proper finite-index submodule of $H^1_{ac^\Sigma}(K,M)$. Then, by the following commutative diagram:
\[
  \begin{tikzcd}
  0 \arrow[r] & X \arrow[r] \arrow[d, "\times(\gamma-1)"]
    & {H^1_{ac^\Sigma}(K,M)} \arrow[r] \arrow[d, "\times(\gamma-1)"]
    & {H^1_{ac^\Sigma}(K,M)/X} \arrow[r] \arrow[d, "\times(\gamma-1)"] & 0 \\
  0 \arrow[r] & X \arrow[r]
    & {H^1_{ac^\Sigma}(K,M)} \arrow[r]
    & {H^1_{ac^\Sigma}(K,M)/X} \arrow[r] & 0
  \end{tikzcd}
\]
by Proposition~\ref{prop:coinvariant}, $H^1_{ac^\Sigma}(K,M)_\Gamma=0$. Then, by the snake lemma,
\[
  (H^1_{ac^\Sigma}(K,M)/X)^\Gamma\rightarrow X_\Gamma\rightarrow 0\rightarrow
  (H^1_{ac^\Sigma}(K,M)/X)_\Gamma\rightarrow 0.
\]
At this point, by the finiteness condition, the third vertical map is itself an isomorphism, so $X_\Gamma=0$. Hence the dual $X'$ of $X$, by assumption, is a finite $\Lambda$-module with $(X')^{\Gamma}=0$. By Nakayama's lemma, $X'$ is $0$, and hence $X$ is trivial.
\end{proof}

The following proposition is  standard in Iwasawa theory.

\begin{proposition}\label{prop:numberoffac0}
$\#\Zp/f_{ac}^\Sigma(0)=\#\Lambda/(T,f_{ac}^\Sigma(T))=\#H^1_{ac}(K,W)\cdot C^\Sigma(W)$.
\end{proposition}

\begin{proof}
By the proposition above, $X$ has no nontrivial pseudo-null submodule. Hence, by the basic theory of modules over Iwasawa algebras,
\[
  0\rightarrow X_{ac}^\Sigma(M)\rightarrow\oplus\Lambda/(f_i)\rightarrow C\rightarrow 0,
\]
where $C$ is finite and none of the $f_i$ is divisible by $T$.

Then observe the following exact ladder:
\[
  \begin{tikzcd}
  0 \arrow[r] & X_{ac}^\Sigma(M) \arrow[r] \arrow[d, "\times(\gamma-1)"]
    & \oplus \Lambda/(f_i) \arrow[r] \arrow[d, "\times(\gamma-1)"]
    & C \arrow[r] \arrow[d, "\times(\gamma-1)"] & 0 \\
  0 \arrow[r] & X_{ac}^\Sigma(M) \arrow[r]
    & \oplus \Lambda/(f_i) \arrow[r] & C \arrow[r] & 0
  \end{tikzcd}
\]
Using the snake lemma again,
\[
  \begin{aligned}
    \#X_{ac}^\Sigma(M)_\Gamma
    =&\#(\oplus\Lambda/(f_i(T)))_\Gamma\\
    =&\#(\oplus\Lambda/(f_i(T),T))\\
    =&\#\Lambda/(\prod_if_i(T),T)=\#\Lambda/(f_{ac}^\Sigma(T),T)
    =\#\Zp/f_{ac}^\Sigma(0).
  \end{aligned}
\]
\end{proof}

\begin{proof}[Proof of the control theorem]
The torsion property has been proved. Combining Proposition~\ref{prop:Mfinite} and Proposition~\ref{prop:numberoffac0}, the size assertion follows.
\end{proof}

\begin{remark}
The anticyclotomic control theorem can be stated for any anticyclotomic $\Zp$-extension $L/K$. However, the size formula is then not so regular, and it destroys the Tamagawa-number part predicted by Birch and Swinnerton-Dyer. Since the final goal of this article is to obtain the Birch and Swinnerton-Dyer formula, we do not state and prove the general case here, although the statement and proof of the general case are not difficult. In fact, the special choice of $L/K$ is used only in Proposition~\ref{prop:numberofkernel}.
\end{remark}

\section{An Explicit Gross--Zagier Formula}\label{sec:GZ}
This section studies the Gross--Zagier formula, with the aim of introducing the derivative of the $L$-function. The main work of this section is to obtain, on the basis of the results of Cai--Shu--Tian \cite{cai2014explicit}, a version that can be used to prove the Birch and Swinnerton-Dyer formula. This version of the Gross--Zagier formula relates the index of the Heegner point in the key formula, the derivative of the $L$-function, and the remaining Tamagawa numbers.

This section is divided into four subsections. Subsection~\ref{section:grosszagierassumption} introduces the necessary background, including Shimura curves and abelian varieties. Subsection~\ref{section:grosszagierexplicit} presents the work of Cai--Shu--Tian and reduces it to the form needed in this article, namely Theorem~\ref{Theorem:grosszaigerimaintheorem}. In the course of proving this theorem, the article introduces some notions, especially congruence numbers and Shimura degrees, the remaining two sections deal with these notions. Subsection~\ref{section:congruencenumberandmodualrnumber} discusses the relation between congruence numbers and Shimura degrees, its main result is Proposition~\ref{prop:modular=congruence}. Subsection~\ref{section:differentshimuradegree} discusses relations among different Shimura degrees, its main results are Theorem~\ref{theorem:deines} and Proposition~\ref{prop:ij}. 

\subsection{Assumptions, Notation, and Preliminaries}\label{section:grosszagierassumption}
This subsection introduces the various preliminaries needed in this section. This subsection mainly follows the book of Yuan, Zhang and Zhang \cite{yuan2013gross}.

\subsubsection{Shimura Curves}
This subsubsection introduces the background on Shimura curves, with emphasis on quaternion algebras, Hodge classes, and Heegner points.

\paragraph{Quaternion Algebras}
This paragraph introduces quaternion algebras over the ad\`ele ring, which are essentially a generalization of $2\times 2$ matrices.

A quaternion algebra $B$ over a field $k$ is a $4$-dimensional central simple   algebra over $k$.

A quaternion algebra over the ad\`ele ring $\mathbb{A}$ of a number field $k$   means a free $\mathbb{A}$-module $\mathbb{B}$ of rank $4$ whose base change to   every local field $k_u$ is a quaternion algebra over $k_u$, i.e.   $\mathbb{B}_u:=\mathbb{B}\otimes_{\mathbb{A}}k_u$ is a quaternion algebra over   $k_u$.

        \begin{proposition}
Up to isomorphism, the only quaternion algebra over $\mathbb{C}$ is    $M_2(\mathbb{C})$. The only quaternion algebras over $\mathbb{R}$ are    $M_2(\mathbb{R})$ and the Hamilton algebra. The only quaternion algebras over    a $p$-adic field $k$ are $M_2(k)$ and its unique division algebra.
        \end{proposition}

        \begin{proof}
See the book of Voight \cite[Chapter 13]{voight2021quaternion}.
        \end{proof}

For a real place $\tau$, by the classification of quaternion algebras over   $\mathbb{R}$, $\mathbb{B}_\tau$ is either the Hamilton algebra or a matrix ring.   $\mathbb{B}$ is said to be \emph{definite} at $\tau$ if $\mathbb{B}_\tau$ is the   Hamilton algebra. $\mathbb{B}$ is said to be \emph{totally definite} if it is   definite at every real place.

Consider a finite set $\Sigma$ of places of a totally real field $F$, and define   $\mathbb{B}$ to be the quaternion algebra with \textbf{ramification set   $\Sigma$}: if $u\in\Sigma$, then $\mathbb{B}_u$ is the unique division quaternion   algebra over $F_u$, if $u\notin\Sigma$, then $\mathbb{B}_u=M_2(F_u)$. In this   case $\mathbb{B}$ is called an $\mathbb{A}$-quaternion algebra ramified at   $\Sigma$.

If $\#\Sigma$ is even, then $\mathbb{B}=B\otimes\mathbb{A}$, where $B$ is a   quaternion algebra over $F$. In this case $\mathbb{B}$ is said to be   \emph{coherent}. If $\#\Sigma$ is odd, then $\mathbb{B}$ is no longer the base   change of any quaternion algebra over $F$. It is then said to be   \emph{incoherent}.

For example, when $\Sigma=\emptyset$ and $F=\mathbb{Q}$, we have   $B=M_2(\mathbb{Q})$.

        \paragraph{Shimura Curves Defined by Quaternion Algebras}
This paragraph introduces Shimura curves defined by quaternion algebras.

Let $\mathbb{B}$ be a totally definite incoherent quaternion algebra over   $\mathbb{A}$ with ramification set $\Sigma$. By total definiteness, $\Sigma$   contains all infinite places, and by incoherence, $\Sigma$ has an odd number of   elements.

For an open compact subgroup $U$ of  $\mathbb{B}_f^{\times}=(\mathbb{B}\otimes_{\mathbb{A}}\mathbb{A}_f)^\times$, one   can construct a compactified Shimura variety $X_U$. To consider the   $\mathbb{C}$-points of $X_U$, one needs to fix an embedding   $\tau:F\rightarrow\mathbb{C}$, then   $$X_{U,\tau}(\mathbb{C})=(B(\tau)^\times \backslash \mathscr{H}^\pm \times   \mathbb{B}_f^{\times}/U)\cup\{\text{cusps}\}.$$   Here $B(\tau)$ is the quaternion algebra, unique up to $F$-isomorphism, ramified   at $\Sigma-\{\tau\}$ (because the size of the ramification set is then even).   $\mathscr{H}^\pm$ is the union of the upper and lower half-planes. The group   $B(\tau)^\times$ acts on $\mathscr{H}^\pm$ through   $B(\tau)\otimes 1\hookrightarrow B(\tau)_\tau\cong M_2(\mathbb{R})$, and acts   naturally on $\mathbb{B}_f^{\times}$. Here $\mathbb{B}_f^{\times}$ can be   identified with $B(\tau)\otimes_F\mathbb{A}_f$.

        \begin{remark}
The set $\{\text{cusps}\}$ is nonempty if and only if $F=\mathbb{Q}$ and   $\Sigma=\{\infty\}$.
        \end{remark}

For two open compact subgroups $U_1\subseteq U_2$ of $\mathbb{B}_f^{\times}$,   there is a natural surjection $\pi_{U_1,U_2}:X_{U_1}\rightarrow X_{U_2}$ between   the corresponding Shimura curves. Consider the projective limit   $X=X(\mathbb{B})$ of the projective system $\{X_U\}_U$, then $X$ is a locally   Noetherian normal scheme over $F$, but it is not of finite type. Considering its   $\mathbb{C}$-points, we have   $$X_{\tau}(\mathbb{C})=(B(\tau)^\times \backslash \mathscr{H}^\pm \times   \mathbb{B}_f^{\times}/D)\cup\{\text{cusps}\},$$   where $D$ is the closure of $F^\times$ in $\mathbb{A}_f$.

The Shimura curve $X$ carries a Hecke action of $\mathbb{B}_f^{\times}$: for   $x\in\mathbb{B}$, consider the right-multiplication action $T_x$ of   $x_f\in\mathbb{B}_f$. In the language of projective systems,   $T_x:X_{xUx^{-1}}\rightarrow X_U$ gives an isomorphism of Shimura curves.

For any open compact subgroup $U$, the curve $X_U$ is connected but in general   not geometrically connected. Consider the set $\pi_0(X_U)$ of its geometric   connected components. For $\alpha\in\pi_0(X_U)$, let $X_{U,\alpha}$ denote the   connected component represented by $\alpha$. Then there is a decomposition   $$X_{U,\bar{F}}=\coprod_{\alpha\in F^\times_+\backslash   \mathbb{A}_f^\times/q(U)}X_{U,\alpha},$$   where $q:\mathbb{B}_f^{\times}\rightarrow\mathbb{A}_f^\times$ is the reduced   norm map. Considering the $\mathbb{C}$-points of $X_U$, we have   $\pi_0(X_{U,\tau}(\mathbb{C}))=B(\tau)^\times_+\backslash   \mathbb{B}_f^{\times}/U=F^\times_+\backslash   \mathbb{A}_f^\times/q(U)$, where $B(\tau)^\times_+$ denotes the elements of   $B(\tau)^\times$ with totally positive reduced norm. Consequently   $$X_{U,\tau}(\mathbb{C})\cong\coprod_{h\in B(\tau)^\times_+\backslash   \mathbb{B}_f^{\times}/U}\Gamma_h\backslash\mathscr{H^\ast},\quad   \Gamma_h:=B(\tau)^\times_+\cap hUh^{-1}.$$   Here $\mathscr{H^\ast}=\mathscr{H^+}\cup\mathbb{P}^1(\mathbb{Q})$ if   $F=\mathbb{Q}$ and $\Sigma=\{\infty\}$, in all other cases,   $\mathscr{H^\ast}=\mathscr{H^+}$.

        \paragraph{Hodge Classes}
This paragraph introduces Hodge classes, which are the generalization of the   cusps of modular curves to general Shimura curves.

Let $\Omega_{X_U/F}$ be the canonical bundle of $X_U$.   Define the line bundle $L_U\in Pic(X_U)_{\mathbb{Q}}$ on the Shimura curve   $X_U$ by
        $$L_U=\Omega_{X_U/F}+\sum_{x\in X_U(\bar{F})}(1-e_x^{-1})x,$$
and call it the Hodge class. Here $e_x$ is the ramification index of the   preimage of $x$ under the map $\mathscr{H}^\ast\rightarrow\Gamma_h\backslash   \mathscr{H}^\ast$. Although $x$ is a geometric point, $L_U$ is defined over $F$.

For any connected component $\alpha\in\pi_0(X_{U,\bar{F}})$, write   $L_{U,\alpha}:=L_U|_{X_{U,\alpha}}$ for the restriction to the connected   component $X_{U,\alpha}$. Define $\xi_{U,\alpha}=\frac{1}{\deg L_{U,\alpha}}   L_{U,\alpha}$, the normalized Hodge class on $X_{U,\alpha}$, and define   $\xi_U=\sum\xi_{U,\alpha}$, the normalized Hodge class.

        \begin{lemma}
For the Hodge class $L_U$, one has $\deg L_U=\mathrm{Vol}(X_U)$. Hence this    is a rational number independent of $\tau:F\rightarrow\mathbb{C}$.
        \end{lemma}

        \begin{proof}
See Yuan--Zhang--Zhang \cite[Lemma 3.1.1(1)]{yuan2013gross}.
        \end{proof}

    \subsubsection{Abelian Varieties and Heegner Points}

        \paragraph{Abelian Varieties Parametrized by Shimura Curves}

        \begin{definition}
Let $A$ be a simple abelian variety over $F$. $A$ is said to be parametrized    by the Shimura curve $X$ if, for some open compact subgroup $U$, there exists    a nonzero map $X_U\rightarrow A$ over $F$.
        \end{definition}

    \begin{remark}
        We assume that the parametrizations are optimal for $X_U \rightarrow E$ for elliptic curves below.
    \end{remark}

By the Eichler--Shimura theory, $A$ is of strictly $GL(2)$-type, i.e.   $M:=End_F(A)\otimes_{\mathbb{Z}}\mathbb{Q}$ is a field and $Lie(A)$ is a free   $M\otimes_{\mathbb{Q}}F$-module of rank $1$.

Define $\pi_A:=\Hom_\xi^0(X,A)=\varinjlim_{U}\Hom_{\xi_U}^{0}(X_U,A)$,   with $\Hom_{\xi_U}^{0}(X_U,A)$ a subspace of   $\Hom_F(X_U,A)\otimes_{\Z}\Q$.   An element $\phi\in\Hom_F(X_U,A)\otimes_{\mathbb{Z}}\mathbb{Q}$ belongs to   $\Hom_{\xi_U}^{0}(X_U,A)$ if and only if $\sum_i a_i\phi(x_i)=0\in   A(\bar{F})_{\mathbb{Q}}$, where $\sum_i a_i x_i$ is a representative of   $\xi_U$ as a divisor on $X_{U,\bar{F}}$.

Let $J_U$ be the Jacobian variety of $X_U$. By the properties of Jacobian   varieties, every morphism from $X_U$ to $A$ factors through $J_U$. In other   words, $\pi_A=\Hom^0(J,A)=\varinjlim_U\Hom^0(J_U,A)$.

Since $\mathbb{B}^{\times}$ acts on $X$, it automatically acts on $\pi_A$, which   makes $\pi_A$ the space of an automorphic representation of   $\mathbb{B}^{\times}$.

Let $A^\vee$ be the dual abelian variety of $A$. Then the composite map   $X\rightarrow A\rightarrow A^\vee$ shows that $A^\vee$ is also parametrized by   the Shimura curve. Hence $\pi_{A^\vee}$ can also be defined. There exists the   following perfect $\mathbb{B}^{\times}$-pairing:
        $$\pi_A\times\pi_{A^\vee}\rightarrow M$$
given explicitly by
\[
        (f_1,f_2)=\mathrm{Vol}(X_U)^{-1}(f_{1,U}\circ f_{2,U}^\vee)\quad
        f_{1,U}\in\Hom(J_U,A),\\
        f_{2,U}\in\Hom(J_U,A^\vee).
\]
Here $f_{2,U}^\vee:A\rightarrow J_U^\vee\cong J_U$ is the composite of maps. This shows that $\pi_{A^\vee}$ is the dual representation of $\pi_A$ as a representation of $\mathbb{B}^{\times}$.

For the elliptic curves $E/F$ considered in this article, $M=\mathbb{Q}$ and   $E\cong E^\vee$, so $\pi_E$ is self-dual. In this case   $(f,f)=\mathrm{Vol}(X_U)^{-1}\deg f_U$, where $\deg f_U$ is the degree of the   parametrization morphism $f_U:X_U\rightarrow E$.

    \paragraph{CM Points and Heegner Points}

For a quadratic totally imaginary extension $K/F$, fix an   $\mathbb{A}$-embedding $K_{\mathbb{A}}\hookrightarrow\mathbb{B}$. Through this   embedding, $K_{\mathbb{A}}^\times$ acts on the Shimura curve $X$ on the right.   Let $X^{K^\times}$ be the subscheme of $X$ fixed by $K^\times$. It is defined   over $F$. Complex multiplication theory asserts that   $X^{K^\times}(\bar{F})=X^{K^\times}(K^{ab})$, and that the Galois action is   precisely the Hecke action, placed inside $\mathbb{B}$ through the reciprocity   map.

Fix a point $P\in X^{K^\times}(K^{ab})$. Then for any open compact subgroup $U$,   the point $P$ induces a point $P_U\in X_U^{K^\times}(K^{ab})$.

Let $E$ be an elliptic curve parametrized by the Shimura curve $X$, then   $M=\mathbb{Q}$. Let $\chi:\Gal(K^{ab}/K)\rightarrow k^\times$ be a character of   finite order, where $k$ is a finite extension of $\mathbb{Q}$. For any   $f\in\pi_E$, the image of $P$ under $f$ is $f(P)\in E(K^{ab})_{\mathbb{Q}}$.   Consider the Heegner point   $$P_{\chi}(f):=\int_{\Gal(\bar{K}/K)} f(P^\tau)\otimes\chi(\tau)\,d\tau   \in E(K^{ab})_{\mathbb{Q}}\otimes k,$$   where the Haar measure on $\Gal(\bar{K}/K)$ is normalized to have total volume   $1$. Since $\chi$ has finite order, the integral is in essence a finite sum. One   sees that $P_\chi(f)\in E(\chi):=(E(K^{ab})_{\mathbb{Q}}\otimes k_\chi   )^{\Gal(K^{ab}/K)}$. Here $k_\chi$ is the vector space $k$ equipped with the   action of $\Gal(\bar{K}/K)$ given by $\chi$. One sees that $P_\chi(f)\neq 0$ if   and only if $\omega_E\cdot\chi|_{\mathbb{A}_f^\times}=1$, where $\omega_E$ is   the central character of $\pi_E$.

        \begin{remark}
This is the definition in Yuan--Zhang--Zhang \cite{yuan2013gross}. For   the conventions of different references on Heegner points, see   Table~\ref{table:heegnertable} in Appendix~\ref{app:conventions}.
        \end{remark}

Define the Heegner point of Cai--Shu--Tian \cite[page 2531]{cai2014explicit}   $$P_\chi^0(f):=\sum_{t\in Pic_{K/F}(\mathcal{O}_{c_1})}f(P)^{\sigma_t}\chi(t)   \in A(K^{ab})_{\mathbb{Q}}\otimes k$$

where $Pic_{K/F}(\mathcal{O}_{c_1})=\hat{K}^\times/K^\times\hat{F}^\times   \hat{\mathcal{O}}_{c_1}^\times$. For the relation between the two Heegner points \(P_\chi(f)\) and \(P_\chi^0(f)\), see \cite[Lemma~2.3]{cai2014explicit}.

\subsection{The Explicit Gross--Zagier Formula}\label{section:grosszagierexplicit}
This subsection gives the main theorem of this section, Theorem~\ref{Theorem:grosszaigerimaintheorem}. It first uses the theorem of Cai--Shu--Tian \cite{cai2014explicit} to obtain the Gross--Zagier formula up to $p$-adic units, and then uses the relation between congruence numbers and Shimura degrees, and the relations among different Shimura degrees, to add the Tamagawa numbers to this formula. For the sake of a natural and readable exposition, this section first gives the derivation of the desired formula, while the key propositions are proved in the following sections.

    \subsubsection{Reduction to the Elliptic-Curve Case}
Keeping the notation of this section, we have the following theorem:

    \begin{theorem}[Cai--Shu--Tian]\label{theorem:explicitGZ}
$F/\mathbb{Q}$ is a totally real field of degree $d$, $A/F$ is an abelian   variety parametrized by a Shimura curve $X$ over $F$, and $\mathbf{f}$ is the   Hilbert newform of weight $(2,\cdots,2)$ on $GL_2(\mathbb{A})$ attached to $A$.   $K/F$ is a quadratic totally imaginary extension with relative discriminant   ideal $D$ and absolute discriminant $D_K$. $\chi:{\mathbb{A}}_K^\times/K^\times   \rightarrow k^\times$ is a Hecke character of finite order, with values in a   finite extension $k$ of $M=End^0(A)$, and with conductor ideal $c$.

Further assume:
        \begin{itemize}
\item 1. $\omega_A\cdot\chi|_{\mathbb{A}^\times}=1$, where $\omega_A$ is the    central character of $\pi_A$.

\item 2. for every place $u'$ of $F$,    $\epsilon(\pi_{A,u'},\chi_{u'})=\chi_{u'}\eta_{u'}(-1)\epsilon(\mathbb{B}_{u'})$,    where $\eta$ is the character attached to $K/F$ and $\epsilon$ is the local    root number.
        \end{itemize}

Let $V(\pi_A,\chi)$ and $V(\pi_{A^\vee},\chi^{-1})$ be the test-vector spaces of Cai--Shu--Tian. Then for any nonzero $f_1\in V(\pi_A,\chi)$ and $f_2\in V(\pi_{A^\vee},\chi^{-1})$, for the Rankin--Selberg $L$-function $L(s,A,\chi)$, the following equality holds in $k\otimes_{\mathbb{Q}} \mathbb{C}$:

$$L'^{(\Sigma)}(1,A,\chi)=2^{-\#\Sigma_D}\cdot   \frac{(8\pi^2)^d(\mathbf{f},\mathbf{f})_{U_0(\mathfrak{N})}}   {u_1^2\sqrt{|D_K|\|c_1^2\|}}\cdot   \frac{\left\langle P^0_\chi(f_1),P^0_{\chi^{-1}}(f_2)\right\rangle _{NT}}   {(f_1,f_2)_{R^\times}}.$$

Here   $\Sigma:=\{u'\mid(\mathfrak{N},Dc)\text{ if }u'||\mathfrak{N}\text{ then }   \operatorname{ord}_{u'}(c/\mathfrak{N})\ge 0\}$, and on the other hand   $\Sigma_D:=\{u'\mid(\mathfrak{N},D)\text{ and }   \operatorname{ord}_{u'}(c/\mathfrak{N})<0\}$.

Here $c_1\mid c$ is the part of $c$ prime to $\Sigma_1$, where   $\Sigma_1=\{u'\mid\mathfrak{N}: u'\text{ nonsplit and }   \operatorname{ord}_{u'}(c/\mathfrak{N})<0\}$. Set
        \[
        u_1=\#\ker\{Pic(\mathcal{O}_F)\rightarrow
        Pic(\mathcal{O}_{c_1})\}\,[\mathcal{O}_{c_1}^\times:\mathcal{O}^\times].
        \]
the kernel $\#\ker\{Pic(\mathcal{O}_F)\rightarrow Pic(\mathcal{O}_{c_1})\}$   has size $1$ or $2$. $(\mathbf{f},\mathbf{f})_{U_0(\mathfrak{N})}$ is the   Petersson inner product.
    \end{theorem}

    \begin{proof}
See Cai--Shu--Tian \cite[Theorem 1.5]{cai2014explicit}.
    \end{proof}

To obtain a version that can be used to prove the Birch and Swinnerton-Dyer formula, one only needs the following special case of Theorem~\ref{theorem:explicitGZ}:

    \begin{proposition}\label{prop:explicitGZ}
Under the assumptions of Theorem~\ref{theorem:explicitGZ}, take $A=E$ to be the elliptic curve mainly studied in this article, i.e. a semistable elliptic curve over $F$ attached to the Hilbert newform $\mathbf{f}$, and take $\chi=\mathbb{1}$ to be the trivial character. Let $\phi\in\pi_E$ have the factorization $\phi:X_U\rightarrow E$.

Then, up to a $p$-adic unit,   $$w_K^2\sqrt{|D_K|}\frac{L'^{(\Sigma)}(1,E/K)}   {(8\pi^2)^d(\mathbf{f},\mathbf{f})_{U_0(\mathfrak{N})}}   =\frac{1}{\deg\phi}\cdot   \left\langle P^0_{\mathbb{1}}(\phi),P^0_{\mathbb{1}}(\phi)\right\rangle _{NT}.$$   Here $w_K$ is the number of roots of unity in $K$.
    \end{proposition}

    \begin{proof}
This article always assumes that $p$ is an odd prime, hence $2^{-\#\Sigma_D}$ is a $p$-adic unit.

Since $\chi$ is the trivial character, $L^{(\Sigma)}(1,E,\chi)=L^{(\Sigma)}(1,E/K)$. Again since $\chi$ is trivial, we have $c=c_1=\mathcal{O}_K$, and since $\#\ker\{Pic(\mathcal{O}_F)\rightarrow Pic(\mathcal{O}_{c_1})\}=1$ or $2$, $u_1$ differs from $[\mathcal{O}_K^\times:\mathcal{O}_F^\times]$ by a $p$-adic unit.

Since $[\mathcal{O}_K^\times:\mathcal{O}_F^\times]= [\mathcal{O}_K^\times:\mu(K)\mathcal{O}_F^\times] [\mu(K)\mathcal{O}_F^\times:\mathcal{O}_F^\times]$, and by the book of Washington \cite[Theorem 4.12]{washington2012introduction} we have $[\mathcal{O}_K^\times:\mu(K)\mathcal{O}_F^\times]=1$ or $2$, while $[\mu(K)\mathcal{O}_F^\times:\mathcal{O}_F^\times]= [\mu(K):\{\pm1\}]=w_K/2$. Hence $u_1/w_K=\#\kappa_{\mathcal{O}_K}\,[\mathcal{O}_K^\times:\mu(K) \mathcal{O}_F^\times]/2\in\{1/2,1,2\}$, so $u_1^{-2}$ and $w_K^{-2}$ differ by a $p$-adic unit.

Finally, by assumption $\phi:X_U\rightarrow E$ is a parametrization of the elliptic curve $E$, by the discussion on page 2530 of Cai--Shu--Tian \cite{cai2014explicit}, $(\phi,\phi)_{R^\times}=\deg\phi$. This completes the proof of the proposition.
    \end{proof}

    \begin{remark}
If one further chooses $K$ so that its relative discriminant ideal satisfies $(D,\mathfrak{N})=1$, then $\Sigma=\emptyset$ in Theorem~\ref{theorem:explicitGZ}.
    \end{remark}

    \subsubsection{Overview of the Theory over $\mathbb{Q}$}
This subsubsection gives a brief overview of how the Gross--Zagier formula used in the proof of the Birch and Swinnerton-Dyer formula when $F=\mathbb{Q}$.

In this case, the conductor of $E$ is $\mathfrak{N}=N$, where $N$ is a squarefree positive integer. $\mathbf{f}=f$ is a usual newform. Assume that $N,K$ satisfy the following generalized Heegner hypothesis:
    \begin{equation}
    \begin{split}
    \bullet & \ \text{$N=N^+N^-$, where $(N^+,N^-)=1$.} \\
    \bullet & \ \text{$\ell\mid N^+\iff\ell$ splits or ramifies in $K$.} \\
    \bullet & \ \text{$\ell\mid N^-\iff\ell$ is inert in $K$.} \\
    \bullet & \ \text{$N^-$ has an even number of prime factors.}
    \end{split}
    \end{equation}

Consider the Shimura curve $X_{N^+,N^-}$ attached to the incoherent quaternion  algebra with ramification set $\Sigma=\{p:p\mid N^-\infty\}$, where $N^+$ comes  from the choice of an Eichler order. Suppose this Shimura curve parametrizes $E$.  Then, by the Gross--Zagier formula of Proposition~\ref{prop:explicitGZ}, up to a  $p$-adic unit we have  $$w_K^2\sqrt{|D_K|}\frac{L'(1,E/K)}{(8\pi^2)(f,f)_{U_0(N)}}  =\frac{1}{\delta(N^+,N^-)}\cdot  \left\langle P^0_1(\phi),P^0_1(\phi)\right\rangle _{NT}.$$  Here $\delta(N^+,N^-)$ is the Shimura degree of the parametrization of $E$ by the  Shimura curve $X_{N^+,N^-}$ (in analogy with the modular degree), i.e.  $\delta(N^+,N^-)=\deg(X_{N^+,N^-}\rightarrow E)$.  Hida defined the  congruence period
    $$\Omega^{\mathrm{cong}}_f:=\frac{(8\pi^2)(f,f)_{U_0(N)}}{\eta_f},$$
where $\eta_f$ is the congruence number of $f$, defined in detail below.

At the same time, consider the parametrization of $E$ by the modular curve  $X_0(N)$, and write $\delta(N,1)$ for the modular degree of this parametrization.  By Theorem 2.1 of Agashe--Ribet--Stein \cite{agashe2012modular},  $\operatorname{ord}_p\delta(N,1)=\operatorname{ord}_p\eta_f$. Substituting this,  the Gross--Zagier formula above becomes  $$w_K^2\sqrt{|D_K|}\frac{L'(1,E/K)}{\Omega^{\mathrm{cong}}_f}  =\frac{\delta(N,1)}{\delta(N^+,N^-)}\cdot  \left\langle P^0_{\mathbb{1}}(\phi),P^0_{\mathbb{1}}(\phi)\right\rangle _{NT}.$$

By Theorem 2 of Ribet--Takahashi \cite{ribet1997parametrizations} and  the article of Takahashi \cite{takahashi2001degrees}, up to a $p$-adic unit one has
    $$\frac{\delta(N,1)}{\delta(N^+,N^-)}=\prod_{\ell\mid N^-}c_\ell(E/K)$$
where $c_\ell(E/K)$ is the Tamagawa number of $E$ over $K$ at $\ell$.

In summary, up to $p$-adic units, one has the formula  $$w_K^2\sqrt{|D_K|}\frac{L'(1,E/K)}{\Omega^{\mathrm{cong}}_f}  =\prod_{\ell\mid N^-}c_\ell(E/K)\cdot  \left\langle P^0_{\mathbb{1}}(\phi),P^0_{\mathbb{1}}(\phi)\right\rangle _{NT}.$$

    \subsubsection{The Theory over $F$}
The purpose of this subsubsection is to derive a formula analogous to the one over  $\mathbb{Q}$.

Recall that $E$ is a semistable elliptic curve over $F$ with conductor  $\mathfrak{N}=\prod_{i=1}^s\mathfrak{q}_i$. Unlike the situation over  $\mathbb{Q}$, here $\mathfrak{N}$ cannot be written as an element of the ring of  integers, but is expressed by ideals. $K/F$ is a quadratic totally imaginary  extension. Assume that $(\mathfrak{N},K)$ satisfies the following Heegner  hypothesis:

    \begin{equation}\label{Heegnerassumption}
    \begin{split}
    \bullet & \ \text{$\mathfrak{N}=\mathfrak{M}\mathfrak{D}$, where
    $(\mathfrak{M},\mathfrak{D})=1$.} \\
    \bullet & \ \text{$u\mid\mathfrak{M}\iff u$ splits in $K$.} \\
    \bullet & \ \text{$u\mid\mathfrak{D}\iff u$ is inert in $K$.} \\
    \bullet & \ \text{the number of prime ideals of $\mathfrak{D}$ has parity
    different from $d=[F:\mathbb{Q}]$.}
    \end{split}\tag{{H}}
    \end{equation}

When $d=[F:\mathbb{Q}]$ is odd, consider the totally definite quaternion algebra  $\mathbb{B}$ with ramification set $\Sigma=\{v:v\mid\mathfrak{D}\infty\}$. The  number of prime ideals of $\mathfrak{D}$ is $2t$ with $t\ge 0$, so  $\#\Sigma=2t+d$, an odd number. By Yuan--Zhang--Zhang  \cite{yuan2013gross} and Cai--Shu--Tian \cite{cai2014explicit}, one can then use  the totally definite incoherent quaternion algebra $\mathbb{B}$ with ramification  set $\Sigma$ to construct a family of Shimura curves $\{X_U\}$, where  $U\subseteq\mathbb{B}_f^{\times}$ is an open compact subgroup. Take  $R\subseteq\mathcal{O}_{\mathbb{B}}\subseteq\mathbb{B}$, where  $\mathcal{O}_{\mathbb{B}}$ is a maximal order and $R$ is an Eichler order of level  $\mathfrak{M}$. Let $U=R\otimes(F\otimes\hat{\mathbb{Z}})$ be an open compact  subgroup. The Shimura curve constructed in this way is denoted by  $X_{\mathfrak{M},\mathfrak{D}}$. It is an algebraic curve defined over $F$.

When $d$ is even, similarly consider the totally definite quaternion algebra  $\mathbb{B}$ with ramification set $\Sigma=\{v:v\mid\mathfrak{D}\infty\}$. The  number of prime ideals of $\mathfrak{D}$ is $2t+1$ with $t\ge 0$, so  $\#\Sigma=2t+1+d$, an odd number. Then the Shimura curve $X_{\mathfrak{M},\mathfrak{D}}$  can also be constructed.

Now assume that $E$ is parametrized by $X_{\mathfrak{M},\mathfrak{D}}$, in other  words, there exists a map
    $$\phi:X_{\mathfrak{M},\mathfrak{D}}\rightarrow E,$$
which factors through the Jacobian variety of $X_{\mathfrak{M},\mathfrak{D}}$, i.e.  $\phi':X_{\mathfrak{M},\mathfrak{D}}\rightarrow J_{\mathfrak{M},\mathfrak{D}}$.

Recall that $E$ is modular and is attached to the Hilbert newform $\mathbf{f}$.  Then, by Proposition~\ref{prop:explicitGZ},  $$w_K^2\sqrt{|D_K|}\frac{L'(1,E/K)}  {(8\pi^2)^d(\mathbf{f},\mathbf{f})_{U_0(\mathfrak{N})}}  =\frac{1}{\delta(\mathfrak{M},\mathfrak{D})}\cdot  \left\langle P^0_{\mathbb{1}}(\phi),P^0_{\mathbb{1}}(\phi)\right\rangle _{NT}.$$

Here $\delta(\mathfrak{M},\mathfrak{D}):=\deg(\phi)$ is the Shimura degree of the  parametrization by the curve $X_{\mathfrak{M},\mathfrak{D}}$.

\begin{definition}[Congruence Period]\label{def:conperiod}
    Define the congruence period of $\mathbf{f}$, $$\Omega^{\mathrm{cong}}_{\mathbf{f}}:=\frac{(8\pi^2)^d(\mathbf{f},\mathbf{f})_{U_0(\mathfrak{N})}}  {\eta_{\mathbf{f}}}.$$
\end{definition}

    \begin{theorem}[Gross--Zagier Formula]\label{Theorem:grosszaigerimaintheorem}
Keep the notation and assumptions of this article, and assume in addition:
        \begin{itemize}
\item when $d$ is even, there exists a $\mathfrak{q}\mid\mathfrak{D}$ with $p\nmid\mathrm{Norm}(\mathfrak{q})+1$ such that, up to a $p$-adic unit, $$\eta_{\mathbf{f}}(\mathfrak{MD},1) =\eta_{\mathbf{f}}(\mathfrak{MD}/\mathfrak{q},\mathfrak{q}) \times c_{\mathfrak{q}}(E/K).$$

\item if $u$ is a place of $F$ ramified in $E[p]$ with $\mathrm{Norm}(u)\equiv -1\bmod p$, then either $E[p]|_{I_u}$ is irreducible or $E[p]|_{G_u}$ is absolutely reducible.

\item $E[p]|_{G_{F(\zeta_p)}}$ is absolutely irreducible.
        \end{itemize}

Then, up to a $p$-adic unit, the following Gross--Zagier formula holds:   $$w_K^2\sqrt{|D_K|}\frac{L'(1,E/K)}{\Omega^{\mathrm{cong}}_{\mathbf{f}}}   =\prod_{u\mid\mathfrak{D}}c_u(E/K)\cdot   \left\langle P^0_{\mathbb{1}}(\phi),P^0_{\mathbb{1}}(\phi)\right\rangle _{NT}.$$

    \end{theorem}

    \begin{proof}
First, by the definition of the congruence period and   Proposition~\ref{prop:explicitGZ}, up to a $p$-adic unit one immediately gets

$$w_K^2\sqrt{|D_K|}\frac{L'(1,E/K)}{\Omega^{\mathrm{cong}}_{\mathbf{f}}}   =\frac{\eta_{\mathbf{f}}}{\delta(\mathfrak{M},\mathfrak{D})}\cdot   \left\langle P^0_{\mathbb{1}}(\phi),P^0_{\mathbb{1}}(\phi)\right\rangle _{NT}.$$

Consider first the case where $d$ is odd. By the assumptions of this article, $E$   is also parametrized by $X_{\mathfrak{MD},1}$. Since the assumptions of   Proposition~\ref{prop:modular=congruence} are satisfied, that proposition   holds. Taking its $\mathfrak{M}$ to be the $\mathfrak{MD}$ here and its   $\mathfrak{D}$ to be $\mathcal{O}_F$, one gets, up to a $p$-adic unit,
        $$\delta(\mathfrak{MD},1)=\eta_{\mathbf{f}}(\mathfrak{MD},1)=\eta_{\mathbf{f}}.$$

Substituting this, up to a $p$-adic unit one has   $$w_K^2\sqrt{|D_K|}\frac{L'(1,E/K)}{\Omega^{\mathrm{cong}}_{\mathbf{f}}}   =\frac{\delta(\mathfrak{MD},1)}{\delta(\mathfrak{M},\mathfrak{D})}\cdot   \left\langle P^0_{\mathbb{1}}(\phi),P^0_{\mathbb{1}}(\phi)\right\rangle _{NT}.$$

Then, by the irreducibility of $E[p]$ and the fact that $\mathfrak{D}$ is a   product of an even number of prime ideals, iterating the generalized   Ribet--Takahashi results, i.e. Theorem~\ref{theorem:deines} and   Proposition~\ref{prop:ij}, one gets   $$\frac{\delta(\mathfrak{MD},1)}{\delta(\mathfrak{M},\mathfrak{D})}   =\prod_{u\mid\mathfrak{D}}\#\Phi_u(E).$$

Finally, $\#\Phi_u(E)=c_u(E/K)$ is the Tamagawa number. Substituting this into   the original formula, one gets   $$w_K^2\sqrt{|D_K|}\frac{L'(1,E/K)}{\Omega^{\mathrm{cong}}_{\mathbf{f}}}   =\prod_{u\mid\mathfrak{D}}c_u(E/K)\cdot   \left\langle P^0_{\mathbb{1}}(\phi),P^0_{\mathbb{1}}(\phi)\right\rangle _{NT}.$$

For $d$ even, $X_{\mathfrak{MD},1}$ is not a curve. One therefore by assumption fixes a   $\mathfrak{q}\mid\mathfrak{D}$ and considers the curve   $X_{\mathfrak{MD}/\mathfrak{q},\mathfrak{q}}$. At this point, by the second   half of the first assumption of the theorem, first up to a $p$-adic unit one   has   $\eta_{\mathbf{f}}=\eta_{\mathbf{f}}(\mathfrak{MD}/\mathfrak{q},\mathfrak{q})   \times c_{\mathfrak{q}}(E/K)$.

At this point, the first half of the first assumption of the theorem together   with the remaining assumptions makes the assumptions of   Proposition~\ref{prop:modular=congruence} satisfied again. Taking its   $\mathfrak{M}$ to be $\mathfrak{MD}/\mathfrak{q}$ and its $\mathfrak{D}$ to be   $\mathfrak{q}$, up to a $p$-adic unit one has   $$\eta_{\mathbf{f}}(\mathfrak{MD}/\mathfrak{q},\mathfrak{q})   =\delta(\mathfrak{MD}/\mathfrak{q},\mathfrak{q}).$$   Since $\mathfrak{D}/\mathfrak{q}$ has an even number of prime factors, exactly   as in the case of odd $d$, comparing the $p$-adic valuations of   $\delta(\mathfrak{M},\mathfrak{D})$ and $\delta(\mathfrak{MD}/\mathfrak{q},   \mathfrak{q})$ by the Ribet--Takahashi results gives the theorem for even $d$.
    \end{proof}

    \begin{remark}
In the next subsection, Subsection~\ref{section:congruencenumberandmodualrnumber}, Proposition~\ref{prop:modular=congruence} holds for every reasonable $\mathfrak{D}$, but in the proof of this proposition one needs a set of assumptions of Manning, which require $p\nmid\mathrm{Norm}(\mathfrak{q})+1$ for every prime factor $\mathfrak{q}$ of $\mathfrak{D}$. In other words, the more complicated $\mathfrak{D}$ is, the more extra conditions are needed. The strategy of the proof in this article is therefore to pass from congruence numbers to Shimura degrees as early as possible, when $\mathfrak{D}$ is simple enough. Note that the conditions for comparing Shimura degrees are weaker.
    \end{remark}

\subsection{Congruence Numbers and Shimura Degrees}
\label{section:congruencenumberandmodualrnumber}
The purpose of this subsection is to generalize Theorem 2.1 of Agashe--Ribet--Stein \cite{agashe2012modular} to the Shimura curves over totally real fields determined by quaternion algebras.

\begin{remark}
    Note that, by the basic assumptions of this article, $E$ has good reduction at all prime ideals above $p$, hence this subsection only generalizes the good-reduction part of Agashe--Ribet--Stein, and does not treat the multiplicative-reduction results.
\end{remark}

    \subsubsection{Statement of the Proposition}

First we introduce the congruence number of Hilbert modular forms. Let $\mathbf{S}_2(\mathbb{Z})$ be the group of Hilbert modular forms of weight $(2,\cdots,2)$, level $\mathfrak{N}$, and with Fourier coefficients in $\mathbb{Z}$.

    \begin{definition}\label{definition:congruencenumber}
For a Hilbert newform $\mathbf{f}\in\mathbf{S}_2(\mathbb{Z})$, define its congruence number to be the order of the group $$\frac{\mathbf{S}_2(\mathbb{Z})}{\mathbb{Z}\mathbf{f} +(\mathbb{Z}\mathbf{f})^\bot}.$$ 

Here $(\mathbb{Z}\mathbf{f})^\bot$ denotes the orthogonal complement of $\mathbb{Z}\mathbf{f}$ with respect to the Petersson inner product. The order  is denoted by $\eta_{\mathbf{f}}$.
    \end{definition}

More generally, for a decomposition $\mathfrak{N}=\mathfrak{MD}$, consider $\mathbf{S}_2(\mathfrak{M},\mathfrak{D})$, the space of $\mathfrak{D}$-new forms with coefficients in $\mathbb{Z}$, then $\eta_{\mathbf{f}}(\mathfrak{M},\mathfrak{D})$ is defined to be the order of the group $$\frac{\mathbf{S}_2(\mathfrak{M},\mathfrak{D})}{\mathbb{Z}\mathbf{f} +(\mathbb{Z}\mathbf{f^\bot})}.$$

In this sense, $\eta_{\mathbf{f}}=\eta_{\mathbf{f}}(\mathfrak{N},1)$.

The following proposition is the main proposition of this section.

    \begin{proposition}\label{prop:modular=congruence}
Let $E$ be the elliptic curve attached to the Hilbert newform $\mathbf{f}$, and suppose that $\phi:X_{\mathfrak{M},\mathfrak{D}}\rightarrow E$ is an optimal parametrization of $E$ by the Shimura curve $X_{\mathfrak{M},\mathfrak{D}}$. Assume $E[p]$ is $G_K$-irreducible. We make further assumptions on $E[p]$:
    \begin{itemize}
        \item For any prime ideal $\mathfrak{q}$ of $\mathfrak{D}$, we have $p \nmid \mathrm{Norm}(\mathfrak{q}) + 1$.
        \item If $u$ is a prime of $F$ ramified in $E[p]$ and $\mathrm{Norm}(u) \equiv -1 \pmod{p}$, then either $E[p]|_{I_u}$ is irreducible, or $E[p]|_{G_u}$ is absolutely reducible.
        \item $E[p]|_{G_{F(\zeta_p)}}$ is absolutely irreducible.
    \end{itemize}

Then, up to a $p$-adic unit,
        $$\deg\phi=\eta_{\mathbf{f}}(\mathfrak{M},\mathfrak{D}).$$
    \end{proposition}

    \begin{proof}
First, Proposition~\ref{prop:oldnew} translates these two objects into the \textbf{modular exponent} and the \textbf{congruence exponent}. The proof is then a combination of Proposition~\ref{prop:modulardividescongruence} and Proposition~\ref{prop:congruencedividesmodular}, which give one divisibility in each direction.
    \end{proof}

    \subsubsection{Modular Exponents and Congruence Exponents}

Recall that $\mathbb{T}$ is the Hecke algebra with $\mathbb{Z}$ coefficients acting on $\mathbf{S}_2(\mathbb{Z})$. Now consider $\mathbb{T}(\mathfrak{M},\mathfrak{D})$, the Hecke algebra with $\mathbb{Z}$ coefficients acting on $\mathbf{S}_2(\mathfrak{M},\mathfrak{D})$.

There is a natural algebra map $\mathbb{T}(\mathfrak{M},\mathfrak{D})\rightarrow \mathbb{Z}$ sending $T_{\mathfrak{q}}\mapsto C(\mathfrak{q},\mathbf{f})$. Let $I$ be the kernel of this map.

Let $E$ be the elliptic curve attached to the Hilbert newform $\mathbf{f}$, and suppose that $\phi:X_{\mathfrak{M},\mathfrak{D}}\rightarrow E$ is a parametrization of $E$ by the Shimura curve $X_{\mathfrak{M},\mathfrak{D}}$. Then $\phi$ is in fact given by a map $X_{\mathfrak{M},\mathfrak{D}}\rightarrow J=J_{\mathfrak{M},\mathfrak{D}} \rightarrow E$. The Hecke algebra $\mathbb{T}(\mathfrak{M},\mathfrak{D})$ then acts naturally on the Jacobian variety $J_{\mathfrak{M},\mathfrak{D}}$ of $X_{\mathfrak{M},\mathfrak{D}}$.

Further assume that $E=J_{\mathfrak{M},\mathfrak{D}}/IJ_{\mathfrak{M},\mathfrak{D}}$. This is analogous to the situation of modular curves and modular Jacobian varieties over $\mathbb{Q}$. Since $E$ and $J$ are self-dual, $E$ can be regarded as a subvariety of $J$. Since $J$ is an abelian variety, by the decomposition theorem for abelian varieties one can find another subvariety $B$ such that $E+B=J$ and $E\cap B$ is finite.

    \begin{lemma}\label{lemma:invariant}
$E$ and $B$ are stable under the action of $\mathbb{T}(\mathfrak{M},\mathfrak{D})$.
    \end{lemma}

    \begin{proof}
In fact, one can show that $E$ and $B$ are stable under the action of $End(J)$. See Agashe--Ribet--Stein \cite[Lemma 3.1]{agashe2012modular}. Although the situation there concerns rational modular forms, the proof there uses only the fact that if two normalized eigenforms have the same Fourier coefficients at almost all primes, then the two modular forms are equal. This fact also holds for Hilbert modular forms.
    \end{proof}

By Lemma~\ref{lemma:invariant}, let $T_E$ be the image of $\mathbb{T}(\mathfrak{M},\mathfrak{D})$ in $End(E)$, and $T_B$ its image in $End(B)$. Set $S_E=\Hom(T_E,\mathbb{Z})$ and $S_B=\Hom(T_B,\mathbb{Z})$.

The following definitions are those of Definition 3.2 and Definition 3.4 of Agashe--Ribet--Stein \cite{agashe2012modular}.

    \begin{definition}
For the decomposition $E+B$ of $J_{\mathfrak{M},\mathfrak{D}}$, define $\tilde{n}_E$ to be the exponent of the group $E\cap B$, called the \textbf{modular exponent}. $n_E$ is the order of the group $E\cap B$, called the \textbf{modular number}.
    \end{definition}

\begin{definition}
    For the decomposition $E+B$ of $J_{\mathfrak{M},\mathfrak{D}}$, define $\tilde{r}_E$ to be the exponent of the group $\frac{\mathbf{S}_2(\mathfrak{M},\mathfrak{D})}{S_E+S_B}$, called the \textbf{congruence exponent}. $r_E$ is the order of the group $\frac{\mathbf{S}_2(\mathfrak{M},\mathfrak{D})}{S_E+S_B}$, called the \textbf{congruence number}.
\end{definition}

    \begin{remark}
In fact, for any decomposition $A+B=J$ with $A\cap B$ finite, one can define analogous modular exponents, modular numbers, congruence exponents, and congruence numbers. Since this article does not need them, this remark takes the place of a discussion.
    \end{remark}

The following propositions reveal that the new definitions above are the same objects in disguise.

    \begin{proposition}\label{prop:oldnew}
Keeping the notation above, $\tilde{n}_E=\deg\phi$ and $n_E=(\deg\phi)^2$, i.e. the modular exponent is exactly the Shimura degree, and the modular number is the square of the Shimura degree.

$\tilde{r}_E=r_E=\eta_{\mathbf{f}}(\mathfrak{M},\mathfrak{D})$, i.e. the new defined congruence exponent and congruence number coincide with the previously defined congruence number.
    \end{proposition}

    \begin{proof}
First we discuss the relation between the modular exponent and the Shimura degree. This is in fact immediate: by assumption of this section, $E\cap B=\ker\{[\deg\phi]:E\rightarrow E\}$, where $[\deg\phi]$ is the multiplication-by-$\deg\phi$ map. Hence $E\cap B=(\mathbb{Z}/\deg(\phi) \mathbb{Z})^2$, so the modular exponent equals the Shimura degree, i.e. $\tilde{n}_E=\deg(\phi)$, and the modular number is the square of the Shimura degree.

Next we consider the relation between the two congruence numbers.

In fact, here $B=IJ_{\mathfrak{M},\mathfrak{D}}$, so $T_E=\mathbb{T}(\mathfrak{M},\mathfrak{D})/I$, i.e. $S_E=\mathbf{S}_2(\mathfrak{M},\mathfrak{D})[I]$. But since $E$ here is an elliptic curve attached to the Hilbert modular form $\mathbf{f}$, we have $S_E=\mathbb{Z}\mathbf{f}$.

Moreover, $S_B=\Hom(T_B,\mathbb{Z})$ is the unique saturated Hecke-stable complement of $S_E=\mathbf{S}_2(\mathfrak{M},\mathfrak{D})[I]$ in $\mathbf{S}_2(\mathfrak{M},\mathfrak{D})$ (see \cite[Section 3]{agashe2012modular}). Hence it equals the Petersson orthogonal complement $(\mathbb{Z}\mathbf{f})^\bot$. This gives $r_E=\eta_{\mathbf{f}}(\mathfrak{M},\mathfrak{D})$.

Furthermore, the map $g\mapsto\langle g,\mathbf{f}\rangle/\langle\mathbf{f},\mathbf{f}\rangle \in\mathbb{Q}$ identifies $\mathbf{S}_2(\mathfrak{M},\mathfrak{D})/(\mathbb{Z}\mathbf{f})^\bot$ with a finitely generated subgroup of $\mathbb{Q}$ containing $1$, hence with $\mathbb{Z}$, the quotient $\frac{\mathbf{S}_2(\mathfrak{M},\mathfrak{D})}{S_E+S_B}$ is therefore cyclic, its exponent equals its order, and the congruence exponent and the congruence number coincide.

This completes the proof of the proposition.
    \end{proof}

    \subsubsection{The Modular Exponent Divides the Congruence Exponent}
The purpose of this subsubsection is to prove one part of Proposition~\ref{prop:modular=congruence}.

    \begin{proposition}\label{prop:modulardividescongruence}
Keeping the notation of this section, one has
        $$\tilde{n}_E\mid\tilde{r}_E.$$
    \end{proposition}

First we prove Lemma~\ref{lemma:anisomorphism}, which translates the problems of the newly defined congruence numbers and congruence exponents to the Hecke algebra. The proposition is then obtained by diagram chasing. The idea is standard, see \cite[Lemma 3.3]{agashe2012modular} for the isomorphism and \cite[Lemmas 4.1--4.2]{agashe2012modular} for the order computations.

    \begin{lemma}\label{lemma:anisomorphism}
There exists a canonical $\mathbb{T}(\mathfrak{M},\mathfrak{D})$-module isomorphism $$Ext^1(T_E\oplus T_B/\mathbb{T}(\mathfrak{M},\mathfrak{D}),\mathbb{Z}) \cong\mathbf{S}_2(\mathfrak{M},\mathfrak{D})/(S_E+S_B).$$ Consequently, there exists a group isomorphism $$\mathbf{S}_2(\mathfrak{M},\mathfrak{D})/(S_E+S_B)\cong (T_E\oplus T_B)/\mathbb{T}(\mathfrak{M},\mathfrak{D}).$$
    \end{lemma}

    \begin{proof}
This is standard homological algebra. First consider the following short exact sequence: $$0\rightarrow\mathbb{T}(\mathfrak{M},\mathfrak{D})\rightarrow T_E\oplus T_B \rightarrow (T_E\oplus T_B)/\mathbb{T}(\mathfrak{M},\mathfrak{D}) \rightarrow 0.$$

The map on the left is injective because $E+B=J_{\mathfrak{M},\mathfrak{D}}$, by the definitions of $T_E,T_B$. Applying $\Hom(-,\mathbb{Z})$, by derived functors one gets
        \begin{multline*}
        0\rightarrow\Hom((T_E\oplus T_B)/\mathbb{T}(\mathfrak{M},\mathfrak{D}),
        \mathbb{Z})\rightarrow\Hom(T_E\oplus T_B,\mathbb{Z})\\
        \rightarrow\Hom(\mathbb{T}(\mathfrak{M},\mathfrak{D}),\mathbb{Z})
        \rightarrow Ext^1((T_E\oplus T_B)/\mathbb{T}(\mathfrak{M},\mathfrak{D}),\mathbb{Z})
        \rightarrow Ext^1(T_E\oplus T_B,\mathbb{Z}).
        \end{multline*}

First, since $E\cap B$ is finite, $(T_E\oplus T_B)/\mathbb{T}(\mathfrak{M},\mathfrak{D})$ is finite, so the first term of the exact sequence above is zero. Moreover, since $T_E\oplus T_B$ is a free $\mathbb{Z}$-module, the last term of the exact sequence above is $0$.

This turns the exact sequence above into
        \[
        0\rightarrow\Hom(T_E\oplus T_B,\mathbb{Z})
        \rightarrow\Hom(\mathbb{T}(\mathfrak{M},\mathfrak{D}),\mathbb{Z})
        \rightarrow Ext^1((T_E\oplus T_B)/\mathbb{T}(\mathfrak{M},\mathfrak{D}),\mathbb{Z})
        \rightarrow 0.
        \]

Next consider the pairing $\mathbb{T}(\mathfrak{M},\mathfrak{D})\times\mathbf{S}_2(\mathfrak{M},\mathfrak{D}) \rightarrow\mathbb{Z}$, $(t,\mathbf{g})\mapsto C(\mathcal{O},t(\mathbf{g}))$,
this is a perfect pairing, so the short exact sequence above becomes
$$0\rightarrow S_E\oplus S_B\rightarrow \mathbf{S}_2(\mathfrak{M},\mathfrak{D})\rightarrow Ext^1((T_E\oplus T_B)/\mathbb{T}(\mathfrak{M},\mathfrak{D}),\mathbb{Z}) \rightarrow 0.$$

Finally, since for a finite group $G$ one has $Ext^1(G,\mathbb{Z})\cong G$, the lemma follows.
    \end{proof}

Now observe the following commutative diagram, on which the proof of this section strongly depends.

    \begin{center}
        \begin{tikzcd}
        0 \arrow[r] & {\mathbb{T}(\mathfrak{M},\mathfrak{D})} \arrow[r] \arrow[d]
        & T_E \oplus T_B \arrow[r] \arrow[d]
        & {\frac{T_E \oplus T_B}{\mathbb{T}(\mathfrak{M},\mathfrak{D})}} \arrow[r]
        \arrow[d] & 0 \\
        0 \arrow[r] & End(J_{\mathfrak{M},\mathfrak{D}}) \arrow[r]
        & End(E)\oplus End(B) \arrow[r]
        & \frac{End(E)\oplus End(B)}{End(J_{\mathfrak{M},\mathfrak{D}})} \arrow[r]
        & 0
        \end{tikzcd}
    \end{center}

Again, since $E\cap B$ is finite, $\frac{End(E)\oplus End(B)}{End(J_{\mathfrak{M},\mathfrak{D}})}$ is finite.

Consider $e=(1,0)\in T_E\oplus T_B$, its image $e_1$ in $\frac{T_E\oplus T_B}{\mathbb{T}(\mathfrak{M},\mathfrak{D})}$, and its image $e_2$ in $\frac{End(E)\oplus End(B)}{End(J_{\mathfrak{M},\mathfrak{D}})}$. The proof of Proposition~\ref{prop:modulardividescongruence} compares the orders of $e_1$ and $e_2$.

    \begin{lemma}\label{lemma:e1order}
The order of the element $e_1$ is $\tilde{r}_E$.
    \end{lemma}

    \begin{proof}
Let $r$ be the order of $e_1$. Note that $e_1$ is an element of $\frac{T_E\oplus T_B}{\mathbb{T}(\mathfrak{M},\mathfrak{D})}$. By Lemma~\ref{lemma:anisomorphism} and the definition of the congruence exponent, one gets $\tilde{r}_E e_1=0$, i.e. $r\mid\tilde{r}_E$. On the other hand, consider $x\in\frac{T_E\oplus T_B}{\mathbb{T}(\mathfrak{M},\mathfrak{D})}$ and
$(a,b)\in T_E\oplus T_B$ with $\overline{(a,b)}=x$. Then $rx=\overline{r(a,b)}=\overline{(ra,0)+(0,rb)} =r\overline{(1,0)}+\overline{(0,rb)}=\overline{(0,rb)}$. But $\overline{(0,rb)}=\overline{(rb,rb)}-r\overline{(b,0)}=0-0=0$. The first term is $0$ because it comes from $\mathbb{T}(\mathfrak{M},\mathfrak{D})$, and the latter is $0$ because $r$ is the order of $e_1$. This shows that $rx=0$ for every $x$, hence $\tilde{r}_E\mid r$.

The two divisibilities together prove the lemma.
    \end{proof}

    \begin{lemma}\label{lemma:e2order}
The order of the element $e_2$ is $\tilde{n}_E$.
    \end{lemma}

    \begin{proof}
Let $n$ be the order of $e_2$.

Consider $End^0(J_{\mathfrak{M},\mathfrak{D}}):= End(J_{\mathfrak{M},\mathfrak{D}})\otimes\mathbb{Q}$. For an element $\beta\in End^0(J_{\mathfrak{M},\mathfrak{D}})$, the denominator of $\beta$ is the smallest positive integer $n_\beta$ such that $n_\beta\beta\in End(J)$.

Let $\pi_E,\pi_B\in End^0(J_{\mathfrak{M},\mathfrak{D}})$ be the projections of $J_{\mathfrak{M},\mathfrak{D}}$ onto $E$ and $B$. By definition, $n$ is the denominator of $\pi_E$, because the image of $n\pi_E$ in $End(E)\oplus End(B)$ is $(1,0)$.

Let $i_E,i_B$ be the embeddings of $E,B$ into $J_{\mathfrak{M},\mathfrak{D}}$, and let $i_E+i_B$ be the map $E\times B\rightarrow J$ induced by the two embeddings. Let $\beta=(n\pi_E,n\pi_B)\in\Hom(J_{\mathfrak{M},\mathfrak{D}}, E\times B)$. Then $\beta\circ(i_E+i_B)=[n]_{E\times B}$.

Now consider the first row of the following exact ladder:

        \begin{center}
            \begin{tikzcd}
            0 \arrow[r] & E\cap B \arrow[r, "{x\mapsto (x,-x)}"]
            \arrow[d, "{[\tilde{n}_E]}"]
            & E\times B \arrow[r, "i_E + i_B"] \arrow[d, "{([\tilde{n}_E],0)}"]
            & J_{\mathfrak{M},\mathfrak{D}} \arrow[r] \arrow[d, "\alpha"] & 0 \\
            0 \arrow[r] & E\cap B \arrow[r]
            & E\times B \arrow[r]
            & J_{\mathfrak{M},\mathfrak{D}} \arrow[r] & 0
            \end{tikzcd}
        \end{center}

Let $\Delta=\image{(E\cap B\rightarrow E\times B)}$. Then $[n]_{E\times B}\Delta=(\beta\circ(i_E+i_B)(\Delta))=\beta(\{0\})=0$. This shows $\tilde{n}_E\mid n$.

On the other hand, by the exact ladder, $\alpha=\tilde{n}_E\pi_E\in End(J_{\mathfrak{M},\mathfrak{D}})$, by the definition of the denominator, this shows $n\mid\tilde{n}_E$. The lemma follows.
    \end{proof}

    \begin{proof}[Proof of Proposition~\ref{prop:modulardividescongruence}]
Note that the image of $e_1$ is $e_2$, hence the order of the former is divisible by the order of the latter. By Lemma~\ref{lemma:e1order} and Lemma~\ref{lemma:e2order}, the orders of the former and the latter translate into exactly the congruence exponent and the modular exponent of the proposition, and the proposition follows.
    \end{proof}

    \subsubsection{The Congruence Exponent Divides the Modular Exponent}
    \label{subsection:congruencedividesmodular}
This subsubsection proves that, up to $p$-adic units, the congruence exponent divides the modular exponent.

    \begin{proposition}\label{prop:congruencedividesmodular}
Keeping the notation of this section, assume $E[p]$ is $G_K$-irreducible. We make further assumptions on $E[p]$:
    \begin{itemize}
        \item For any prime ideal $\mathfrak{q}$ of $\mathfrak{D}$, we have $p \nmid \mathrm{Norm}(\mathfrak{q}) + 1$.
        \item If $u$ is a prime of $F$ ramified in $E[p]$ and $\mathrm{Norm}(u) \equiv -1 \pmod{p}$, then either $E[p]|_{I_u}$ is irreducible, or $E[p]|_{G_u}$ is absolutely reducible.
        \item $E[p]|_{G_{F(\zeta_p)}}$ is absolutely irreducible.
    \end{itemize} one has
        $$\operatorname{ord}_p(\tilde{r}_E)\le\operatorname{ord}_p(\tilde{n}_E).$$
    \end{proposition}

To prove this proposition, we will introduce some more facts about Hecke algebras. For convenience, in this subsubsection $\mathbb{T}$ denotes $\mathbb{T}(\mathfrak{M},\mathfrak{D})$.

        \paragraph{The Saturation of the Hecke Algebra}
This paragraph reduces Proposition~\ref{prop:congruencedividesmodular} to a proposition about the Hecke algebra and its saturation.

Recall the notation of this article: $\mathbb{T}$ is the Hecke algebra with $\mathbb{Z}$-coefficients, regarded as a subalgebra of $End(J_{\mathfrak{M},\mathfrak{D}})$. Define the saturation $\mathbb{T}'$ of $\mathbb{T}$ by $$\mathbb{T}':=End(J_{\mathfrak{M},\mathfrak{D}})\cap(\mathbb{T}\otimes \mathbb{Q}).$$

By the assumptions of this article, both $\mathbb{T}$ and $End(J_{\mathfrak{M},\mathfrak{D}})$ are finitely generated over $\mathbb{Z}$ of the same rank, hence $\mathbb{T}'/\mathbb{T}$ is a finitely generated torsion group, and therefore a finite group.

Continue to consider the decomposition $J_{\mathfrak{M},\mathfrak{D}}=E+B$, and write $\pi_E:\mathbb{T}\rightarrow T_E$ and $\pi_B:\mathbb{T}\rightarrow T_B$. Note that the symbols $\pi_E,\pi_B$ of this paragraph have the same symbols as those of the previous subsubsection but different meanings, no confusion will arise from the context.

By the decomposition, define the \textbf{$(E,B)$-congruence ideal}
by
        $$R:=\pi_E(\ker\pi_B)\subseteq T_E,$$
and define the \textbf{$(E,B)$-intersection ideal} by
        $$S:=Ann_{T_E}(E\cap B).$$
        \begin{lemma}\label{lemma:RinS}
The $(E,B)$-congruence ideal is contained in the $(E,B)$-intersection ideal, i.e.
            $$R\subseteq S.$$
        \end{lemma}

        \begin{proof}
Take any element $t\in R$. By definition it suffices to show that $t$ annihilates $E\cap B$. But a preimage of $t$ in $\mathbb{T}$ already annihilates $B$, so in $T_E$ it annihilates $B\cap E$, this shows $t\in S$.
        \end{proof}

        \begin{proposition}\label{prop:localizationimplyequal}
For a prime $p$, if the localizations of $R$ and $S$ at $p$ coincide, then
            $$\operatorname{ord}_p(\tilde{r}_E)\le\operatorname{ord}_p(\tilde{n}_E).$$
        \end{proposition}

        \begin{proof}
By Lemma~\ref{lemma:RinS}, there is a surjection $T_E/R\rightarrow T_E/S$. Assume that the localizations of $R$ and $S$ at $p$ coincide, then the surjection above becomes an isomorphism.

By the  isomorphism theorem of modules, one immediately gets $$\frac{T_E}{R}\cong\frac{\mathbb{T}}{\ker(\pi_E)+\ker(\pi_B)} \cong\frac{T_E\oplus T_B}{\mathbb{T}}.$$

By this isomorphism and the definitions above, the exponent of $T_E/R$ is exactly the congruence exponent $\tilde{r}_E$. Again by the definition of $S$, multiplying an element $t\in T_E$ by the modular exponent $\tilde{n}_E$ always annihilates $E\cap B$. This makes the modular exponent a multiple of the exponent of $T_E/S$.

Now, since the two are isomorphic after localization, the $p$-part of the modular exponent is a multiple of the $p$-part of the exponent of the latter, and hence naturally a multiple of the $p$-part of the exponent of the former. The $p$-part of the former is exactly the $p$-part of the congruence exponent. This proves the proposition.
        \end{proof}

        \begin{proposition}\label{prop:embedintosaturation}
There exists an injection of $\mathbb{T}$-modules $S/R\rightarrow\mathbb{T}'/\mathbb{T}$.
        \end{proposition}

        \begin{proof}
Since $E\cap B$ is finite, $\mathbb{T}\otimes\mathbb{Q}\cong T_E\otimes\mathbb{Q}\oplus T_B\otimes\mathbb{Q}\subseteq End(J_{\mathfrak{M},\mathfrak{D}})\otimes \mathbb{Q}$. Below we regard all objects as subobjects of $End(J_{\mathfrak{M},\mathfrak{D}})\otimes\mathbb{Q}$.

From this point of view, $R=\mathbb{T}\cap T_E\subseteq End(J_{\mathfrak{M},\mathfrak{D}})\otimes\mathbb{Q}$. On the other hand, $S=T_E\cap End(J_{\mathfrak{M},\mathfrak{D}})\subseteq End(J_{\mathfrak{M},\mathfrak{D}})\otimes\mathbb{Q}$.

Then $S/R=S/S\cap\mathbb{T}\cong(S+\mathbb{T})/\mathbb{T}\rightarrow \mathbb{T}'/\mathbb{T}$. This proves the proposition.
        \end{proof}

By Proposition~\ref{prop:localizationimplyequal} and Proposition~\ref{prop:embedintosaturation}, to prove Proposition~\ref{prop:congruencedividesmodular} it remains to study the $p$-part of $\mathbb{T}'/\mathbb{T}$. This is the work of the next paragraph.

        \paragraph{The $p$-Part of the Congruence Number and the $p$-adic
Completion of the Hecke Algebra} This paragraph mainly uses Theorem 1.2 of Manning \cite{manning2021patching} to finally settle Proposition~\ref{prop:modular=congruence}. To this end, this paragraph will spend some words on notation.

Recall that, in defining the Shimura curve $X_{\mathfrak{M},\mathfrak{D}}$, this article uses the language of the incoherent totally definite quaternion algebra $\mathbb{B}$ over the ad\`ele ring with ramification set $\Sigma:=\{\mathfrak{q}:\mathfrak{q}\mid\mathfrak{D}\infty\}$. The open compact subgroup needed to define this curve is $U_0(\mathfrak{N})$, denoted simply by $U$ below.

For a chosen place $w\mid p$ of $F$, write $L'|F_w$ for a finite extension and $\mathcal{O}_{L'}$ for its ring of integers. Consider the object
$$\mathcal S:=H^1(X_{\mathfrak{M},\mathfrak{D}}(\mathbb{C}),\mathbb{Z})\otimes \mathcal{O}_{L'}.$$

Of course, the $\mathbb{C}$-points of $X_{\mathfrak{M},\mathfrak{D}}$ require a choice of an embedding of $F$, but this is not essential here, so the related statement is omitted.

In the article of Manning \cite{manning2021patching}, the notation $\mathcal S$ is written as $S^D(K)$, where $D$ denotes the quaternion algebra in that article and $K$ the open compact subgroup. Since the quaternion algebra is fixed in this section, the notation is simplified.

For a place $u\nmid\mathfrak{N}=\mathfrak{M}\mathfrak{D}$ of $F$, consider the double-coset operators $T_u$, $S_u:\mathcal S\rightarrow\mathcal S$ defined by
$$T_u=\left[U_0(\mathfrak{N})\left(\begin{array}{cc} \varpi_u & 0\\ 0 & 1
        \end{array}\right)U_0(\mathfrak{N})\right],\quad
S_u=\left[U_0(\mathfrak{N})\left(\begin{array}{cc} \varpi_u & 0\\ 0 & \varpi_u
        \end{array}\right)U_0(\mathfrak{N})\right].$$

Consider the Hecke algebra $$\mathbb{T}_{\mathfrak{N}}:=\mathcal{O}_{L'} [T_u,S_u,S_u^{-1}\mid u\nmid\mathfrak{N}]\subseteq End_{\mathcal{O}_{L'}}(S).$$

Let $\mathfrak{m}$ be a maximal ideal of $\mathbb{T}_{\mathfrak{N}}$. By the article of Carayol \cite{carayol1986representations}, such ideals correspond to two-dimensional semisimple Galois representations
$$\bar{\rho}_{\mathfrak{m}}:G_F\rightarrow GL_2(\mathbb{T}_{\mathfrak{N}} /\mathfrak{m}).$$

        \begin{theorem}[Manning]\label{theorem:manning}
If $\bar{\rho}_{\mathfrak{m}}$ satisfies the following conditions:
            \begin{itemize}
\item $\bar{\rho}_{\mathfrak{m}}$ is automorphic for the quaternion algebra $\mathbb{B}$.

\item $\bar{\rho}_{\mathfrak{m}}|_{G_u}$ is finite flat for $u\mid p$.

\item if $u$ is a place of $F$ ramified in $\bar{\rho}_{\mathfrak{m}}$ and $\mathrm{Norm}(u)\equiv -1\bmod p$, then either $\bar{\rho}_{\mathfrak{m}}|_{I_u}$ is irreducible or $\bar{\rho}_{\mathfrak{m}}|_{G_u}$ is absolutely reducible.

\item if $u\mid\mathfrak{D}$, then $\bar{\rho}_{\mathfrak{m}}$ is Steinberg at $u$ and $\mathrm{Norm}(u)\neq -1\bmod p$.

\item $\bar{\rho}_{\mathfrak{m}}|_{G_{F(\zeta_p)}}$ is absolutely irreducible,
            \end{itemize}

then the following map is an isomorphism:
$$\mathbb{T}_{\mathfrak{N},\mathfrak{m}}\rightarrow End_{\mathbb{T}_{\mathfrak{N},\mathfrak{m}}}(\mathcal S_{\mathfrak{m}}),$$ where $\mathbb{T}_{\mathfrak{N},\mathfrak{m}}:=(\mathbb{T}_{\mathfrak{N}})_{\mathfrak{m}}$.
        \end{theorem}

        \begin{proof}
See Manning \cite[Theorem 1.2]{manning2021patching}.
        \end{proof}

        \begin{remark}
By the modularity conjectures, for some $\mathfrak{m}$, $E[p]$ is isomorphic to the $\bar{\rho}_{\mathfrak{m}}$ above. Following \cite[Section 1.6]{manning2021patching}, a local representation is Steinberg at $v$ if it is upper triangular with diagonal entries $\chi\varepsilon_p$ and $\chi$ for an unramified character $\chi$ (the off-diagonal entry is arbitrary). Since $E$ is semistable, $E[p]$ has multiplicative reduction at every $v\mid\mathfrak{N}$. Hence $E[p]|_{G_v}$ has semisimplification $\varepsilon_p\oplus 1$ (split multiplicative) or $\chi\varepsilon_p\oplus\chi$ with $\chi$ the unramified quadratic character (nonsplit multiplicative), so it is Steinberg at $v$ in the sense of \cite{manning2021patching}. Together with good (ordinary) reduction at every $v\mid p$ and the assumptions of the article, the first, second, and fourth conditions of Theorem~\ref{theorem:manning} are automatically satisfied (note that $p$ is unramified in $F$ by the standing assumptions), this is also why the theorem of this article keeps only the third, the fifth, and the second half of the fourth assumptions.
        \end{remark}

The following proposition deals with the relation between the two Hecke algebras appearing in this section:

        \begin{proposition}\label{prop:Heckealgebrathesame}
There exists an isomorphism $$\mathbb{T}(\mathfrak{M},\mathfrak{D})\otimes\mathcal{O}_{L'} =\mathbb{T}_{\mathfrak{N}}.$$
        \end{proposition}

\begin{proof}
By assumptions,  $\mathbb{T}(\mathfrak{M},\mathfrak{D})$ is a finitely generated $\mathbb{Z}$-algebra, so $\mathbb{T}(\mathfrak{M},\mathfrak{D})\otimes\mathcal{O}_{L'}$ is obtained by changing the coefficient ring. It becomes a finitely generated algebra over $\mathcal{O}_{L'}$, whose generators are precisely the Hecke operators, and this is exactly the definition of $\mathbb{T}_{\mathfrak{N}}$.
\end{proof}

        \paragraph{Proof of Proposition~\ref{prop:congruencedividesmodular}}

        \begin{proof}[Proof of Proposition~\ref{prop:congruencedividesmodular}]
By Proposition~\ref{prop:localizationimplyequal} and Proposition~\ref{prop:embedintosaturation}, to prove Proposition~\ref{prop:congruencedividesmodular} it suffices to study the $p$-part of $\mathbb{T}'/\mathbb{T}$.

Since $\mathcal{O}_{L'}$ is a free $\mathbb{Z}_p$-module, it suffices to show that $\mathbb{T}'/\mathbb{T}\otimes\mathcal{O}_{L'}$ is trivial. Hence it suffices to show $\mathbb{T}\otimes\mathcal{O}_{L'}=\mathbb{T}'\otimes\mathcal{O}_{L'}$.

Consider the following chain of inclusions $$End(J_{\mathfrak{M},\mathfrak{D}})\subseteq End(J_{\mathfrak{M},\mathfrak{D}}(\mathbb{C}))\subseteq End(\mathcal S).$$ The first inclusion is self-evident, the second inclusion comes from the definition of the Jacobian variety over $\mathbb{C}$.

Consider the ring $\mathbb{T}''$ larger than $\mathbb{T}'$, defined by $\mathbb{T}'':=End(\mathcal S)\cap(\mathbb{T}\otimes\mathbb{Q})$.

$\mathbb{T}''\otimes \mathcal{O}_{L'}$  is larger than $\mathbb{T}'\otimes\mathcal{O}_{L'}$, and hence naturally larger than $\mathbb{T}\otimes\mathcal{O}_{L'}$. But note that $\mathbb{T}''$ commutes with $\mathbb{T}$ inside itself. In the totally definite setting there is no Galois action on the  cohomology $\mathcal S$, so $\otimes \mathcal{O}_{L'}$ must be contained in $End_{\mathbb{T}_{\mathfrak{N,m}}}(\mathcal S_{\mathfrak{m}})$. But Theorem~\ref{theorem:manning} asserts that the latter ring equals the minimal ring $\mathbb{T}\otimes\mathcal{O}_{L'}$. Hence all the inclusions above are equalities.

Thus the $p$-part of $\mathbb{T}'/\mathbb{T}$ is trivial, which completes the proof of Proposition~\ref{prop:congruencedividesmodular}.
        \end{proof}

\subsection{Relations between Shimura Degrees}\label{section:differentshimuradegree}
The purpose of this subsection is to generalize the work of Ribet--Takahashi \cite[Theorem 2]{ribet1997parametrizations} and of Takahashi \cite{takahashi2001degrees}, relating the degrees of the morphisms parametrizing $E$ by different Shimura curves to Tamagawa numbers. To this end, this section still needs some preparations.

    \subsubsection{Component Groups, Character Groups, and Monodromy Pairings}
Fix a finite place $u$ of $F$. Let $\mathcal{A}$ be the N\'eron model over $\mathcal{O}_{F_u}$ of an abelian variety $A$ over $F$. Write $\mathcal{A}_u$ for the special fiber of the model $\mathcal{A}$. It is a commutative group scheme, and $\mathcal{A}_u^0$ denotes its connected component containing the identity. The group scheme $\Phi_u(A):=\mathcal{A}_u/\mathcal{A}_u^0$ is called the component group of the abelian variety $A$ at $u$. It is a finite group scheme defined over the residue field $\kappa_u$, and $\#\Phi_u(A)(\kappa_u)=c_u(A)$ is the Tamagawa number.

The Chevalley structure theorem \cite{chevalley1960demonstration} asserts that, for a smooth connected algebraic group over a perfect field, there exists a unique normal smooth connected affine algebraic subgroup such that the quotient is an abelian variety. Applying the Chevalley construction to $\mathcal{A}_u^0$, write $\mathcal{C}$ for that unique normal smooth connected affine algebraic subgroup. Consider the structural short exact sequence of $\mathcal{C}$:
    $$0\rightarrow\mathcal{T}_u(A)\rightarrow C\rightarrow\mathcal{U}\rightarrow 0,$$
where $\mathcal{T}_u(A)$ is the torus part of $A$ over $\kappa_u$ and $\mathcal{U}$ is the unipotent group.

Define the character group of the abelian variety $A$ at $u$ to be $$\mathcal{X}_u(A):=\Hom_{\bar{\kappa_u}}(\mathcal{T}_u(A)_{\bar{\kappa_u}}, \mathbb{G}_{m,\bar{\kappa_u}}).$$ Since the torus part becomes a direct sum of several copies of the multiplicative group after base change to an algebraically closed field, the character group $\mathcal{X}_u(A)$ is in fact a free abelian group of finite rank.

If $A$ has semistable reduction at $u$, then there exists the following monodromy pairing
    $$m_{A,u}:\mathcal{X}_u(A)\times\mathcal{X}_u(A^\vee)\rightarrow\mathbb{Z}$$
and a short exact sequence $$0\rightarrow\mathcal{X}_u(A^\vee)\xrightarrow{\alpha} \Hom(\mathcal{X}_u(A),\mathbb{Z})\rightarrow\Phi_u(A)\rightarrow 0,$$ where $(\alpha(x))(y)=m_{A,u}(x,y)$.

Consider the Jacobian variety $J=J(X)$ of a curve $X$ over $F$ and an elliptic curve $E$, assume that $E$ is an optimal quotient of $J(X)$, and write $\pi:J(X)\rightarrow E$ for the morphism. Considering respectively the component groups and the character groups of $J(X)$ and $E$, and using that they are both self-dual, one has the following commutative diagram:

        \begin{center}
        \begin{tikzcd}
        0 \arrow[r] & \mathcal{X}_u(J) \arrow[r] \arrow[d, "\pi_\ast"]
        & {\Hom(\mathcal{X}_u(J),\mathbb{Z})} \arrow[r] \arrow[d]
        & \Phi_u(J) \arrow[r] \arrow[d, "\pi_\ast"] & 0 \\
        0 \arrow[r] & \mathcal{X}_u(E) \arrow[r]
        & {\Hom(\mathcal{X}_u(E),\mathbb{Z})} \arrow[r]
        & \Phi_u(E) \arrow[r] & 0
        \end{tikzcd}
        \end{center}

        \begin{definition}\label{definition:cij}
For the morphism $\pi_\ast:\Phi_u(J)\rightarrow\Phi_u(E)$, define three    constants:    $$\bar{c}_u(\pi):=\#\Phi_u(E)\qquad    i_u(\pi):=\#\image(\pi_\ast)\qquad    j_u(\pi):=\#\coker(\pi_\ast).$$
        \end{definition}

Consider a squarefree ideal $\mathfrak{N}=\mathfrak{D}\mathfrak{a}\mathfrak{b}   \mathfrak{M}$ of $\mathcal{O}_F$ such that the number of prime ideals of   $\mathfrak{D}$ has parity different from $d=[F:\mathbb{Q}]$. Consider the two   Shimura curves $X_1=X_{\mathfrak{M}\mathfrak{a}\mathfrak{b},\mathfrak{D}}$ and   $X_2=X_{\mathfrak{M},\mathfrak{a}\mathfrak{b}\mathfrak{D}}$ and their Jacobian   varieties $J_1=J(X_{\mathfrak{M}\mathfrak{a}\mathfrak{b},\mathfrak{D}})$ and   $J_2=J(X_{\mathfrak{M},\mathfrak{a}\mathfrak{b}\mathfrak{D}})$. Let $E_1$ and   $E_2$ be the optimal quotients of $J_1$ and $J_2$ respectively, and write   $\pi_1:J_1\rightarrow E_1$ and $\pi_2:J_2\rightarrow E_2$ for the corresponding   morphisms. By Faltings' isogeny theorem, $E_1$ and $E_2$ are isogenous.

Then the Ph.D. thesis of Alyson Deines gives the following result:

        \begin{theorem}[Alyson Deines]\label{theorem:deines}
Write $\delta_1=\delta(\mathfrak{M}\mathfrak{a}\mathfrak{b},\mathfrak{D})    =\deg(\pi_1)$ and    $\delta_2=\delta(\mathfrak{M},\mathfrak{a}\mathfrak{b}\mathfrak{D})    =\deg(\pi_2)$. Then
$$\frac{\delta_1}{\delta_2}    =\frac{\bar{c}_{\mathfrak{a}}(\pi_2)\bar{c}_{\mathfrak{b}}(\pi_1)}    {i_{\mathfrak{b}}(\pi_1)^2j_{\mathfrak{a}}(\pi_2)^2}.$$
        \end{theorem}

        \begin{proof}
See the Ph.D. thesis of Deines \cite[Theorem 3.2.7]{deines2014shimura}.
        \end{proof}

        \begin{proposition}\label{prop:ij}
If $E_1[p]$ and $E_2[p]$ are irreducible representations, then for any prime ideals $\mathfrak{a},\mathfrak{b}\nmid p$, $i_{\mathfrak{a}}(\pi_1)^2j_{\mathfrak{b}}(\pi_2)^2$ is a $p$-adic unit.
        \end{proposition}

        \begin{proof}
The proof of this proposition will be given in two parts, the $i$-part and the $j$-part. Combining Lemma~\ref{lemma:ipart} and Lemma~\ref{lemma:jpart} gives the proposition.
        \end{proof}

        \begin{lemma}\label{lemma:ipart}
Keeping the notation above, if $E_1[p]$ and $E_2[p]$ are irreducible representations, then $i_{\mathfrak{a}}(\pi_1)$ is a $p$-adic unit.
        \end{lemma}

        \begin{lemma}\label{lemma:jpart}
Keeping the notation above, if $E_1[p]$ and $E_2[p]$ are irreducible representations, then $j_{\mathfrak{b}}(\pi_2)$ is a $p$-adic unit.
        \end{lemma}

    \subsubsection{The $i$-Part}

        \paragraph{Proof of Lemma~\ref{lemma:ipart}}

        \begin{proof}[Proof of Lemma~\ref{lemma:ipart}]
The idea of the proof of this part comes from the article of Ribet--Takahashi \cite{ribet1997parametrizations}.

This part can be proved using the theory of Hilbert modular forms parallel to that of modular forms. First, $\pi_{1,\ast}$ still commutes with the Hecke operators $T_{\mathfrak{m}}$ with $(\mathfrak{m},\mathfrak{N})=1$. Since $\Phi_u(J)$ is Eisenstein,
$T_{\mathfrak{m}}$ acts on $\Phi_u(J)$ by scalar multiplication by $\mathrm{Norm}(\mathfrak{m})+1$. Hence, since the morphism $\pi_{1,\ast}$ commutes with the Hecke action, $\mathrm{Norm}(\mathfrak{m})+1-C(\mathfrak{m},\mathbf{f})$ annihilates $\image{\pi_{1,\ast}}$.

If the $p$-part of $\image{\pi_{1,\ast}}$ is nontrivial, then for every $\mathfrak{m}$ with $(\mathfrak{m},\mathfrak{N})=1$ one gets $\mathrm{Norm}(\mathfrak{m})+1\equiv C(\mathfrak{m},\mathbf{f})\bmod p$.
The residual representations of $G_F$ attached to $\mathbf{f}$ and $1+\chi_F^{cyc}$ have the same traces and eigenvalues at the Frobenius elements of almost all places, hence the same characteristic polynomials, and therefore the two representations are isomorphic. In particular, this representation is reducible. By the modularity conjectures, this shows that the $E[p]$ representation is also reducible, which is a contradiction.
        \end{proof}

    \subsubsection{The $j$-Part}

        \subsubsection{The $p$-adic Uniformization Theorem}
This subsubsection follows the work of Bertolini and Darmon \cite{bertolini1999p}.

To avoid confusion of notation, when stating the $p$-adic uniformization results we use $\ell$ in place of $p$. Consider the $\ell$-adic upper half-plane $\mathscr{H}_\ell=\mathbb{C}_\ell-F_u$, where $\mathbb{C}_\ell$ is a completion of an algebraic closure of $F_u$, and $u\mid\ell$. Analogously to the theory at Archimedean places, $GL_2(F_u)$ acts on $\mathscr{H}_\ell$ by fractional linear transformations.

Consider a finite place $u$ at which $E$ has split multiplicative reduction. Write the conductor as $\mathfrak{N}=\mathfrak{MD}u$ in such a way that $\mathfrak{D}u$ has $d'$ prime ideal factors, with $d'$ of parity different from $[F:\mathbb{Q}]$. Consider again the Shimura curve $X_{\mathfrak{M},u\mathfrak{D}}$ defined earlier. The quaternion algebra defining $X_{\mathfrak{M},u\mathfrak{D}}$ is then ramified at $u$.

Let $B$ be another quaternion algebra over $F$, totally definite and ramified at $\mathfrak{D}$ (by the theory of quaternion algebras, such a quaternion algebra exists, because its ramification set has an even number of elements). Let $R$ be an Eichler order in $B$ of level $\mathfrak{M}$. Consider the subgroup $\Gamma_u$ of $F_u^\times\slash R[1/\varpi_u]^\times$ consisting of elements of even $u$-adic valuation. In fact, $\bar{\Gamma}_u:=(\Gamma_u)/(\Gamma_u)_{tor}$ can be identified with the character group of $J_{\mathfrak{M},u\mathfrak{D}}$ at $u$.

Since $E$ is modular and attached to the Hilbert modular form $\mathbf{f}$, there is a Hecke action on $\bar{\Gamma}_u$. Consider the submodule $\bar{\Gamma}_u^{\mathbf{f}}$ of $\bar{\Gamma}_u$ such that $T_{\mathfrak{m}}\bar{\Gamma}_u^{\mathbf{f}} =C(\mathfrak{m},\mathbf{f})\bar{\Gamma}_u^{\mathbf{f}}$ for $(\mathfrak{m},\mathfrak{N})=1$. Then $\bar{\Gamma}_u^{\mathbf{f}}\cong \mathbb{Z}$. Choose a generator $e^{\mathbf{f}}$ of it.

Using the monodromy pairing, define the map $\alpha_{\mathbf{f}}:\bar{\Gamma}_u\rightarrow\bar{\Gamma}_u$, $e\mapsto m_{J(\mathfrak{M},u\mathfrak{D}),u}(e,e^{\mathbf{f}})e^{\mathbf{f}}$.

Write $\mathcal{N}:=\Hom(\bar{\Gamma}_u,\mathbb{Z})$, and let $\mathcal{N}^{\mathbf{f}}$ be its submodule on which the Hecke operators act by $T_{\mathfrak{m}}\mathcal{N}^{\mathbf{f}} =C(\mathfrak{m},\mathbf{f})\mathcal{N}^{\mathbf{f}}$ for $(\mathfrak{m},\mathfrak{N})=1$. Then $\alpha_{\mathbf{f}}$ induces a map $\alpha_{\mathbf{f}}:\mathcal{N}\rightarrow\mathbb{Z}$.

Returning to the description of the uniformization theorem: through the isomorphism $B\otimes F_u\cong M_2(F_u)$, $\Gamma_u$ acts on $\mathscr{H}_\ell$. This action is discontinuous, so the rigid quotient space $\Gamma_u\slash\mathscr{H}_\ell$ is in fact the set of $\mathbb{C}_\ell$-points of some curve $\mathscr{X}$ over $F_u$.

        \begin{theorem}[Cherednik, Drinfeld, Boutot--Zink]
        \label{theorem:padicuniform}
The curves $\mathscr{X}$ and $X$ are isomorphic over $\mathbb{C}_\ell$.
        \end{theorem}

        \begin{proof}
For $F=\mathbb{Q}$, this is the work of Cherednik and Drinfeld, see    Drinfeld \cite{drinfeld1976coverings}. For general totally real fields, this    is the result of Boutot and Zink, see Boutot--Zink
\cite[Theorem 0.1]{boutot1995p}.
        \end{proof}

Although the Cherednik--Drinfeld--Boutot--Zink theorem is stated in greater   generality, for this article the formulation above is sufficient.

By Theorem~\ref{theorem:padicuniform}, one can give a rigid-analytic   description of the Jacobian variety of the Shimura curve   $X_{\mathfrak{M},u\mathfrak{D}}$. Consider a divisor   $D=P_1+\cdots+P_r-Q_1-\cdots-Q_r\in Div^0(\mathscr{H}_\ell)$. One can define a   theta function   $$\vartheta(z;D)=\prod_{\epsilon\in\Gamma}   \frac{(z-\epsilon P_1)\cdots(z-\epsilon P_r)}   {(z-\epsilon Q_1)\cdots(z-\epsilon Q_r)}.$$   This theta function satisfies the following functional equation: for   $\delta\in\Gamma_u$,
        $$\vartheta(\delta z;D)=\phi_D(\bar{\delta})\vartheta(z;D).$$
Here $\bar{\delta}$ is the image of $\delta$ in $\bar{\Gamma}_u$, and   $\phi_D\in\Hom(\bar{\Gamma}_u,\mathbb{C}_\ell^\times)=\mathcal{N}\otimes   \mathbb{C}_\ell^\times$ is independent of the choice of $z$. Furthermore, for   $\bar{\alpha}\in\bar{\Gamma}_u$, the value of   $\phi_{(\alpha z)-(z)}(\bar{\delta})$ is again independent of the choice of   $z$. Then   $$[\;,\;]:\bar{\Gamma}_u\times\bar{\Gamma}_u\rightarrow F_u^\times;\qquad   [\bar{\alpha},\bar{\delta}]:\phi_{(\bar{\alpha}z)-(z)}(\bar{\delta}).$$

Recall again that $\bar{\Gamma}_u$ can be identified with $\mathcal{X}_u(J_{\mathfrak{M},u\mathfrak{D}})$. The following theorem explains the relation between the pairing just defined and the monodromy pairing:

        \begin{proposition}
Let $\operatorname{ord}_u$ be the normalized valuation on $F_u$. Then $$\operatorname{ord}_u\circ[\;,\;] =m_{J(\mathfrak{M},u\mathfrak{D}),u}.$$
        \end{proposition}

        \begin{proof}
See Theorem 7.6 of Manin \cite{manin1976p}.
        \end{proof}

Hence $\operatorname{ord}_u\circ[\;,\;]$ is a symmetric positive-definite bilinear form, consequently the map induced by $[\;,\;]$,
        $$j:\bar{\Gamma}_u\rightarrow\mathcal{N}\otimes F_u^\times,$$
is injective. Write $\Theta:=j(\bar{\Gamma}_u)$.

Consider the map   $$Div^0(\mathscr{H}(\mathbb{C}_\ell))\rightarrow   (\mathscr{N}\otimes\mathbb{C}_\ell^\times)/\Theta,\qquad   D\mapsto\phi_{\tilde{D}},$$   where $\tilde{D}$ is any lift of the divisor $D$ to $\mathscr{H}_\ell$.

        \begin{theorem}[Gerritzen]\label{theorem:heckeequive}
The map above does not depend on the choice of $\tilde{D}$, and hence is    well defined. Moreover, this map is trivial on principal divisors, and the    map above induces a Hecke-equivariant isomorphism from the set of    $\mathbb{C}_\ell$-points of the Jacobian variety $\mathscr{J}$ of    $\mathscr{X}$ to $(\mathscr{N}\otimes\mathbb{C}_\ell^\times)/\Theta$.
        \end{theorem}

        \begin{proof}
See the book of Gerritzen \cite{gerritzen2006schottky}.
        \end{proof}

Recall $\alpha_{\mathbf{f}}$, it again induces the map   $\alpha_{\mathbf{f}}\otimes id:\mathcal{N}\otimes\mathbb{C}_\ell^\times   \rightarrow\mathbb{C}_\ell^\times$. Moreover, by the assumption that $E$ is   optimal, $\alpha_{\mathbf{f}}$ gives the map   $J_{\mathfrak{M},u\mathfrak{D}}\rightarrow E$.

The following proposition is the version over $F_u$ of Proposition 4.4 of   Bertolini--Darmon \cite{bertolini1999p}.

        \begin{proposition}\label{prop:bdcommutativediagram}
Write $\Theta^{\mathbf{f}}$ for the submodule induced by    $\mathcal{N}^{\mathbf{f}}$.

Then the module $q^{\mathbb{Z}}:=\ker\{\mathbb{C}_\ell^\times\rightarrow    E(\mathbb{C}_\ell)\}$ in the Tate uniformization theory of $E$ is    canonically isomorphic to $\Theta^{\mathbf{f}}$, and the following diagram
commutes up to sign:

            \begin{center}
            \begin{tikzcd}
            0 \arrow[r] & \Theta \arrow[r] \arrow[d, "\alpha_{\mathbf{f}}"]
            & \mathcal{N}\otimes\mathbb{C}_\ell^\times \arrow[r, "\Phi"]
            \arrow[d, "\alpha_{\mathbf{f}}\otimes id"]
            & \mathscr{J}(\mathbb{C}_\ell) \arrow[r] \arrow[d, "\alpha_{\mathbf{f}}"]
            & 0 \\
            0 \arrow[r] & \Theta^{\mathbf{f}} \arrow[r]
            & \mathbb{C}_\ell^\times \arrow[r, "\Phi_{Tate}"]
            & E(\mathbb{C}_\ell) \arrow[r] & 0
            \end{tikzcd}
            \end{center}
        \end{proposition}

        \begin{proof}
The right square is guaranteed by Theorem~\ref{theorem:heckeequive},    Theorem~\ref{theorem:padicuniform}, and the split multiplicative reduction    at $u$. It remains to prove $\Theta^{\mathbf{f}}=\ker\Phi_{Tate}$.

Note that the image of    $\alpha_{\mathbf{f}}\otimes id:\mathcal{N}\otimes\mathbb{C}_\ell^\times    \rightarrow\mathbb{C}_\ell^\times$,    namely $\mathbb{C}_\ell^\times=\mathcal{N}^{\mathbf{f}}\otimes    \mathbb{C}_\ell^\times$, can in fact be regarded as a submodule of    $\mathcal{N}\otimes\mathbb{C}_\ell^\times$. Furthermore, since Jacobian    varieties and elliptic curves are self-dual, $E(\mathbb{C})$ can also be    regarded as a subvariety of $\mathscr{J}(\mathbb{C})$. One thus obtains    $\ker\Phi_{Tate}=\mathcal{N}^{\mathbf{f}}\otimes\mathbb{C}_\ell^\times\cap    \Theta$, and this is exactly $\Theta^{\mathbf{f}}$.
        \end{proof}

        \paragraph{Proof of Lemma~\ref{lemma:jpart}}

        \begin{proof}[Proof of Lemma~\ref{lemma:jpart}]
This proof follows the idea of Section 2 of Takahashi    \cite{takahashi2001degrees}.

Here one uses the theory of $p$-adic uniformization. Following the notation    above, here $u=\mathfrak{b}$, and the $\mathfrak{aD}$ of the lemma plays the    role of the $\mathfrak{D}$ there.

First, following the idea of this paragraph, assume that $\mathfrak{b}$ is    split multiplicative.

Consider the ideal of $\mathbb{Z}$ given by the monodromy pairing
            \[
            I=m_{J_2,u}(e^{\mathbf{f}},\mathcal{X}_u(J_2)),
            \]
where $e^{\mathbf{f}}$ generates $\bar{\Gamma}_u^{\mathbf{f}}$.

\textbf{Claim}: $I$ is generated both by $i_{\mathfrak{b}}(\pi_2)$ and by    $\bar{c}_{\mathfrak{b}}(\pi_2)$.

Once the claim holds, since    $\bar{c}_{\mathfrak{b}}(\pi_2)=i_{\mathfrak{b}}(\pi_2)\times    j_{\mathfrak{b}}(\pi_2)$, one immediately gets    $j_{\mathfrak{b}}(\pi_2)=1$.

\textbf{Proof of the claim}: for    $\pi_{2,\ast}:\mathcal{X}_u(J_2)\rightarrow\mathcal{X}_u(E_2)$, there is    also its dual morphism, since $J_2$ and $E_2$ are self-dual, one obtains    another morphism $\pi_2^\ast:\mathcal{X}_u(E_2)\rightarrow    \mathcal{X}_u(J_2)$. Hence $\pi_2^\ast(\mathcal{X}_u(E_2))\subseteq    \bar{\Gamma}_u^{\mathbf{f}}$. By the same reasoning as in    Ribet--Takahashi \cite{ribet1997parametrizations}, one gets    $[\bar{\Gamma}_u^{\mathbf{f}}:\pi_2^\ast(\mathcal{X}_u(E_2))]    =j_{\mathfrak{b}}(\pi_2)$. One can therefore choose a generator $x_u$ of    $\mathcal{X}_u(E_2)$ such that
$\pi_2^\ast(x_u)=j_{\mathfrak{b}}(\pi_2)e^{\mathbf{f}}$.

Now consider $y\in\mathcal{X}_u(J_2)$, then    $\pi_{2,\ast}(y)=tx_u$ for some $t\in\mathbb{Z}$. By definition one can    compute    $$t\bar{c}_{\mathfrak{b}}(\pi_2)=t\,m_{E_2,u}(x_u,x_u)    =m_{E_2,u}(x_u,tx_u)=m_{E_2,u}(x_u,\pi_{2,\ast}(y)).$$
Then, by the properties of the monodromy pairing, one continues to compute
            \[
            m_{E_2,u}(x_u,\pi_{2,\ast}(y))
            =m_{J_2,u}(\pi_2^\ast(x_u),y)
            =m_{J_2,u}(j_{\mathfrak{b}}(\pi_2)\cdot e^{\mathbf{f}},y)
            =j_{\mathfrak{b}}(\pi_2)\cdot m_{J_2,u}(e^{\mathbf{f}},y).
            \]

That is,    $$t\frac{\bar{c}_{\mathfrak{b}}(\pi_2)}{j_{\mathfrak{b}}(\pi_2)}    =m_{J_2,u}(e^{\mathbf{f}},y).$$

Since    $\bar{c}_{\mathfrak{b}}(\pi_2)=i_{\mathfrak{b}}(\pi_2)\times    j_{\mathfrak{b}}(\pi_2)$, the ideal $I$ is contained in the ideal generated    by $i_{\mathfrak{b}}(\pi_2)$. To get equality, one only needs to note that    $\pi_{2,\ast}$ is surjective. This is because $E_2$ is chosen to be an    optimal quotient. See Proposition 2 of    \cite{ribet1997parametrizations}, i.e. $t$ can be taken to be $1$. Thus the    first part of the claim holds.

On the other hand, by the commutative diagram above and Proposition~\ref{prop:bdcommutativediagram}, one has $m_{J_2,u}(e^{\mathbf{f}},\mathcal{X}_u(J_2)) =\operatorname{ord}_u(\alpha_{\mathbf{f}}(\Theta)) =\operatorname{ord}_u(\Theta^{\mathbf{f}}) =\operatorname{ord}_u(q^{\mathbb{Z}}) =(\operatorname{ord}_u(q))\subseteq\mathbb{Z}$. Since $\operatorname{ord}_u(q)=\bar{c}_{\mathfrak{b}}(\pi_2)$, the second part of the claim follows.

By the claim, $j_{\mathfrak{b}}(\pi_2)=1$. Thus the case at hand is proved.

Moreover, since $E_2[p]$ is irreducible, if one takes a different parametrization map (i.e. not necessarily an optimal quotient), then the degrees of these two parametrizations differ by a $p$-adic unit. Hence the optimal-quotient assumption can be dropped.

Next we treat the case where $\mathfrak{b}$ is not split multiplicative. Consider a quadratic unramified extension $k$ of $F_{\mathfrak{b}}$, then $E_2$ has split multiplicative reduction over $k$. Since all the statements in this section are entirely local, the propositions and lemmas above hold over $k$. Consider a quadratic extension $F'$ of $F$ such that $\mathfrak{b}=\mathfrak{b'}$ is inert, i.e. $F'_{\mathfrak{b}'}=k$. Then, by what was said above, $j_{\mathfrak{b'}}(\pi_2)=1$. But the cokernel of $\pi_{2,\ast}$ over $\kappa_{\mathfrak{b}}$ injects into the cokernel over the unramified quadratic residue field $\kappa_{\mathfrak{b}'}$, so by Definition~\ref{definition:cij}
$j_{\mathfrak{b}}(\pi_2)\mid j_{\mathfrak{b'}}(\pi_2)$, since $j_{\mathfrak{b'}}(\pi_2)=1$, $j_{\mathfrak{b}}(\pi_2)$ is also a $p$-adic unit.
        \end{proof}

\section{An Explicit Liu--Zhang--Zhang Formula}\label{sec:LZZ}

The aim of this section is to prove an explicit  Liu--Zhang--Zhang  formula that connects the special value at the trivial character of an Iwasawa-theoretic $p$-adic $L$-function with the square of the $p$-adic logarithm of a Heegner point, thereby completing the final piece of the proof of the Birch and Swinnerton-Dyer formula.

Formula of this type  originates from the work of  Bertolini-- Darmon--Prasanna \cite[Theorem 5.13]{bertolini2013generalized}, and their formula is also known as the BDP formula. Their setup is on modular curves, later, it was generalized by  Brooks in his article \cite{brooks2015shimura} to Shimura curves over \(\mathbb{Q}\), and his formula is referred to as the Brooks formula in Jetchev--Skinner--Wan \cite{jetchev5birch}. In recent years, the work of Yifeng Liu, Shouwu Zhang, and Wei Zhang \cite{liu2018p} reformulated the above results in the language of automorphic forms and generalized them to Shimura curves over totally real fields. In this section, we will employ the results of Liu--Zhang--Zhang and we call the identity obtained from their results the \textbf{Liu--Zhang--Zhang formula}.

The section is organized as follows. Subsection~\ref{section:lzzpadicLfunction} recalls the Liu--Zhang--Zhang $p$-adic $L$-function and its $p$-adic Waldspurger formula, introducing the automorphic notation where it is needed. Subsection~\ref{section:explicitlzz} specializes to elliptic curves: after the reduction and the local computations of Cai--Shu--Tian, we compare the interpolation with the anticyclotomic $p$-adic $L$-function of Jetchev--Skinner--Wan when $F=\mathbb{Q}$, define the Iwasawa-theoretic $p$-adic $L$-function $\mathcal{L}^{IW}_{E,K,\phi}$, and prove the trivial-character Liu--Zhang--Zhang formula announced above. 

\subsection{The Liu--Zhang--Zhang $p$-adic $L$-function}
\label{section:lzzpadicLfunction}

This subsection fixes the automorphic setup and states the theorems of Liu--Zhang--Zhang in the form used in this article. We introduce the notation as it is needed, and refer to Liu--Zhang--Zhang \cite{liu2018p} for the full theory.

    \subsubsection{Notions in Section~\ref{sec:GZ}}
    \label{section:lzzobjects}

We keep the conventions fixed at the beginning of Subsubsection~\ref{sec:assumptions}: $F$ is a totally real field, $K/F$ is a totally imaginary quadratic extension, and the decomposition $w=v\bar{v}$ of the place $w$ of $F$ in $K$ is an assumption running through this article. 

Let $\mathbb{C}_p$ be the completion of an algebraic closure of $F_w$, fix an isomorphism $\iota: \mathbb{C}_p \cong \mathbb{C}$ whose restriction to $\bar{F}_w$ is $\iota_w^{-1}$, and regard all fields, including $F$ and $K$, as subfields of $\mathbb{C}_p$. This convention is adopted because the work of Liu--Zhang--Zhang is completely $p$-adic, and connects the $p$-adic world with the complex world through the abstract isomorphism $\iota$.

Recall the following objects of Section~\ref{sec:GZ}:
    \begin{itemize}
\item $\mathbb{B}$ is a totally definite incoherent $\mathbb{A}$-quaternion algebra with ramification set $\Sigma(\mathbb{B})=\{u\mid\mathfrak{D}\}\cup \{u\mid\infty\}$.

\item $X(\mathbb{B})$ is the Shimura curve attached to this quaternion algebra.

\item $A$ is an abelian variety over $F$ parametrized by $X$.

\item $\pi_A$ is the automorphic representation of $\mathbb{B}^\times_f$ for the parametrization.

\item $M_A=End_F(A)\otimes \mathbb{Q}$.

\item $\omega_A: F^\times \backslash \mathbb{A}^{\infty\times} \rightarrow M_A^\times$ is the central character associated with $A$.

\item $A^\vee$ is the dual abelian variety of $A$, and $\pi_{A^\vee}$ and $M_{A^\vee}$ are defined similarly.

\item $(\cdot,\cdot)_A:\pi_A \times \pi_{A^\vee} \rightarrow M_A$ is the canonical pairing.
    \end{itemize}

For ease of reference to the original article of Liu--Zhang--Zhang, write $A^+:=A$ and $A^-:=A^\vee$.

\subsubsection{Various $L$-functions}

We shall use the following standard $L$-functions. Let $\mathcal{K}$ be a field extension containing $M_A$, and let $\rho_{A,u}$ be the representation given by the $\ell$-adic Tate module of $A$, where $\ell \nmid u$.
    \begin{itemize}
\item $L(s,\rho_{A,u}) \in M_A \otimes_{\mathbb{Q}} \mathbb{C} $ is the local $L$-function of the Galois representation.

\item $L(s,\rho_{A,u},Ad) \in M_A \otimes_{\mathbb{Q}} \mathbb{C} $ is the local adjoint $L$-function.

\item for a locally constant character $\chi_u:F_u^\times \rightarrow \mathcal{K}^\times $, $L(s,\rho_{A,u}\otimes \chi_u) \in \mathcal{K}\otimes_{\mathbb{Q}} \mathbb{C}$ is the local twisted $L$-function. If $\psi:F_u\rightarrow \mathcal{K}^\times$ is a nontrivial additive character, then there is the $\epsilon$ factor $\epsilon(1/2,\psi,\rho_{A,u}\otimes \chi_u)$.

\item for a locally constant character $\chi_u:K_u^\times \rightarrow \mathcal{K}^\times $ satisfying $\omega_{A,u}\cdot \chi_u|_{F_u^\times}=1$, $L(s,\rho_{A,u}, \chi_u) \in \mathcal{K}\otimes_{\mathbb{Q}} \mathbb{C}$ is the local Rankin-Selberg $L$-function and $\epsilon(1/2,\rho_{A,u},\chi_u)$ is the corresponding $\epsilon$ factor.

\item for an embedding $\tau: \mathcal{K}\rightarrow \mathbb{C}$, which induces a map $\tau: \mathcal{K}\otimes_{\mathbb{Q}} \mathbb{C}\rightarrow \mathbb{C}$,
        $$L(s,\rho_A^{(\tau)}):=\prod_{u<\infty}\tau L(s,\rho_{A,u})$$
is the global $L$-function of $A$.

\item $A$ is called automorphic if its global $L$-function $L(s,\rho_A^{(\tau)})$ is in fact the finite part of the $L$-function of an irreducible cuspidal representation of $GL_2(\mathbb{A})$.
    \end{itemize}

From now on, fix an automorphic abelian variety $A$ of $GL_2$-type over $F$ together with an inclusion $M_A\subseteq \mathbb{C}_p$ (as agreed above, all fields are regarded as subfields of $\mathbb{C}_p$).

    \subsubsection{Characters and distribution algebras}
    \label{section:lzzcharacters}

To state the interpolation property of the Liu--Zhang--Zhang $p$-adic $L$-function, we need a class of algebraic characters of  $K^\times\backslash\mathbb{A}_K^{\infty\times}$.\footnote{Liu--Zhang--Zhang write $E$ for the CM field, and its finite ad\`eles by $E^\times$ and $\mathbb{A}_E^{\infty\times}$. Throughout this section we write $K$, $K^\times$, and $\mathbb{A}_K^{\infty\times}$ instead, since $E$ denotes the elliptic curve elsewhere in this article.}

Let $\mathfrak{S}$ be the set of open compact subgroups of $\mathbb{A}_K^{\infty v \times}$, and let $\mathcal{K}/F_w$ be a complete field extension. A character of weight $k\in \mathbb{Z}$ $$\chi: K^\times \backslash \mathbb{A}_K^{\infty \times} \rightarrow \mathcal{K}^\times$$ is one satisfying the following two conditions:

    \begin{itemize}
\item $\chi$ is invariant under the action of some $V^{w} \in \mathfrak{S}$. $V^w$ is called the tame level of $\chi$.

\item there exists an open compact subset $V_w$ of $K_w^{\times}$ such that $\chi(t)=(t_v/t_{\bar{v}})^k$.
    \end{itemize}

For a character $\chi$ of weight $k\in \mathbb{Z}$, one can define two characters $\check{\chi}_v$ and $\check{\chi}_{\bar{v}}$ of $F_w^\times$, defined respectively by $\check{\chi}_v(t)=t^{-k}\chi_{v}(t)$ and $\check{\chi}_{\bar{v}}(t)=t^k \chi_{\bar{v}}(t)$.

Suppose $\tau: \mathbb{C}_p\cong \mathbb{C}$. Then for a character $\chi$ of weight $k\in \mathbb{Z}$ one can define an associated $\mathbb{C}$-valued character $\chi^{(\tau)}$, whose local factors are defined as follows:
    \begin{itemize}
\item $\chi_u^{(\tau)}=1$ if $u|\infty$ and $u\neq \tau|_F$.

\item $\chi_u^{(\tau)}(z)=(z/\bar{z})^k$ if $u=\tau|_F$.

\item $\chi_u^{(\tau)}=\tau \chi_u$ for $u < \infty$ and $u\neq w$.

\item $\chi_w^{(\tau)}(t)=\tau\bigl( \check{\chi}_v(t_v) \check{\chi}_{\bar{v}}(t_{\bar{v}})\bigr)$, and $t\in K_w^\times$.
    \end{itemize}

Then $\chi^{(\tau)}:=\bigotimes_u \chi_u^{(\tau)}: \mathbb{A}_K^{\times} \rightarrow \mathbb{C}^\times$ is called the $\tau$-avatar of $\chi$.

Assume further that $\mathcal{K}$ contains $M_A$, and denote by $\Xi(A,\mathcal{K})_k$ the set of characters of weight $k$ satisfying the following two conditions:
    \begin{itemize}
\item $\omega_A \cdot \chi|_{\mathbb{A}^{\infty \times}}=1$.

\item $\# \{u < \infty, u\neq w| \epsilon(1/2,\rho_{A,u},\chi_u)=-1\} \equiv d-1 \mod 2 $.

    \end{itemize}

Finally write $\Xi(A,\mathcal{K})=\bigcup_{\mathbb{Z}}\Xi(A,\mathcal{K})_k$. In particular, the trivial character $\mathbb{1}$ is a character of weight $0$.

The values of the $p$-adic $L$-function will be elements of a distribution algebra. Let $\mathcal{K}/F_w$ be a complete field extension containing $M_A$. For a locally constant character $\omega: F^\times \backslash \mathbb{A}^{\infty\times} \rightarrow M^\times$, let $\mathcal{C}(\omega, \mathcal{K})$ be the locally convex $\mathcal{K}$-linear space whose elements are the locally analytic functions $f$ on $K^\times \backslash \mathbb{A}_K^{\infty \times}$ satisfying the following conditions:
    \begin{itemize}
\item $f$ is invariant under the action of some $V^w \in \mathfrak{S}$.

\item for $x\in K^\times \backslash \mathbb{A}_K^{\infty \times}$ and $t\in F^\times \backslash \mathbb{A}^{\infty\times} $ one has $f(xt)=\omega(t)^{-1}f(x)$.
    \end{itemize}
Let $\mathcal{D}(\omega,\mathcal{K})$ be the strong dual of $\mathcal{C}(\omega,\mathcal{K})$, called the distribution algebra.

Furthermore, let $\mathcal {D}(A,\mathcal{K})$ be the quotient algebra of the distribution algebra $\mathcal{D}(\omega_A,\mathcal{K})$ , where the ideal being quotiented out is generated by the elements that take trivial values on $\Xi(A,\mathcal{K})$. In the interpolation theorem below the coefficient field will be $MF_w^{lt}$, where $M=M_A$ and $F_w^{lt}$ is a specific extension of $F_w$ (in the notation of Liu--Zhang--Zhang, $F_w^{lt}$ is the field obtained from $F_w^{nr}$ by adjoining the ``period'' of the Lubin--Tate group. See \cite{liu2018p}, Section 1.4).

    \subsubsection{The pairing lemma}
    \label{section:lzzpairing}

Before stating the interpolation theorem, we recall the pairing lemma on which it is based.

    \begin{lemma}\label{lemma:lzzlemma}
Let $k\ge1$. For $\iota:\mathbb{C}_p \cong \mathbb{C}$, there is a unique $\mathbb{B}^{\infty\times } \times \mathbb{A}_K^{\infty\times}$-invariant bilinear pairing

$$(\cdot,\cdot)_{A,\chi}^{(\iota)}:(\pi^+\otimes \sigma_\chi^+)\times (\pi^-\otimes \sigma_\chi^-) \rightarrow (Lie A^+\otimes Lie A^-)\otimes \mathbb{C}$$ such that for arbitrary $f_\pm \in \pi^\pm$, $\varphi_\pm \in \sigma_\chi^\pm$ and $\omega_\pm \in H^0(A^\pm,\Omega_{A^\pm}^1)$, the following equality holds:

$$\frac{\left\langle  \omega_+\otimes \omega_-,(f_+\otimes \varphi_+,f_-\otimes \varphi_-)_{A,\chi}^{(\iota)} \right\rangle }{(\iota\varphi_+\otimes c_\iota^\ast \iota \varphi_-\otimes \mu^k)}=\int_{X_\iota(\mathbb{C})}\frac{\Theta_\iota^{k-1}f_+^\ast \omega_+\otimes c_\iota^\ast\Theta_\iota^{k-1}f_-^\ast\omega_-}{\mu^k}dx.$$

where
    \begin{itemize}
\item $\left\langle \cdot,\cdot\right\rangle $ is the canonical pairing between $H^0(A^+,\Omega_{A^+}^1)\otimes H^0(A^-,\Omega_{A^-}^1)$ and $LieA^+\otimes LieA^-$.

\item $\sigma_\chi^\pm$ is the one-dimensional space in $H^0(Y^\pm,\Omega_{X,Y^\pm}^{\otimes -k})_F\otimes \mathcal{K}$ consisting of elements satisfying $T_t^\ast(\varphi)=\chi(t)^\pm\varphi$, and $(\cdot,\cdot)_\chi$ is a pairing $\sigma_\chi^+\times \sigma_\chi^- \rightarrow \mathcal{K}$.

\item $\mu$ is an arbitrary Hecke invariant hyperbolic metric on ${X}_\iota(\mathbb{C})$.

\item $c_\iota$ is the complex conjugation.

\item $(\iota\varphi_+\otimes c_\iota^\ast \iota \varphi_-\otimes \mu^k)$ can be regarded as a constant function on $Y^+_\iota(\mathbb{C})$, hence as a complex number.

\item $\Theta_\iota$ is the Shimura--Maass operator (see Paragraph~\ref{section:automorphicperiods} for the definition used in the comparison).

\item $dx$ is the Tamagawa measure on $X_\iota(\mathbb{C})$.
    \end{itemize}

Furthermore, there exists a unique element $\mathbf{P}_\iota(A,\chi)\in (Lie A^+ \otimes Lie A^-)\otimes \mathbb{C}$ such that

$$(\cdot,\cdot)_{A,\chi}^{(\iota)}=\mathbf{P}_\iota(A,\chi) \iota(\cdot,\cdot)_A \otimes \iota(\cdot,\cdot)_\chi$$
    \end{lemma}

    \subsubsection{The $p$-adic $L$-function of Liu--Zhang--Zhang}
    \label{section:lzzinterpolation}

    \begin{theorem}[The Liu--Zhang--Zhang $p$-adic $L$-function]
    \label{theorem:lzzpadicLFUNCTION}
Let $\mathcal{D}(A,MF_w^{lt})$ be the distribution algebra defined above.

There exists a unique element $\mathcal{L}(A)\in (Lie A^+ \otimes_{F\otimes M_A} Lie A^-)\otimes_{F\otimes M_A}\mathcal{D}(A,MF_w^{lt})$ such that for $\chi\in \Xi(A,\mathcal{K})_k$ with $k\ge1$, $M_A F_w^{lt}F_w^{ab} \subseteq \mathcal{K}\subseteq \mathbb{C}_p$, and $\iota: \mathbb{C}_p \cong \mathbb{C}$, the following interpolation formula holds:

    \[
    \iota\mathcal{L}(A)(\chi)=\iota\left(\frac{\epsilon(1/2,\psi,\rho_{A,w}\otimes
    \chi_{\bar{v}})}{L(1/2, \rho_{A,w}\otimes \chi_{\bar{v}} )^2 }\right)\frac{2^{d-1}
    \sqrt{|D_K|}\zeta_F(2)}{L(1,\eta)^2L(1,\rho_A^{(\iota)},Ad)}
    \cdot L(1/2,\rho_A^{(\iota)},\chi^{(\iota)}) \mathbf{P}_\iota(A,\chi).
    \]

Here $\iota\mathcal{L}(A)(\chi)$ and $\mathbf{P}_\iota(A,\chi)$ both lie in $(Lie A^+ \otimes_{F\otimes M_A} Lie A^-)\otimes_{F\otimes M_A}\mathbb{C}$.
    \end{theorem}

    \begin{proof}
See Theorem 3.2.10 of Liu--Zhang--Zhang \cite{liu2018p}.
    \end{proof}

We call the $p$-adic $L$-function given by this theorem the \textbf{Liu--Zhang--Zhang $p$-adic $L$-function}. Note that the interpolation range of the interpolation formula consists only of characters of weight $k\ge1$, the values at characters of weight $0$, i.e. finite-order characters, are given by the $p$-adic Waldspurger formula below, in terms of the $p$-adic logarithm.

\subsubsection{Heegner points, the $p$-adic logarithm, and matrix coefficients}
    \label{section:lzzheegner}

        \paragraph{Heegner points and the $p$-adic logarithm}

Let $Y:=X^{K^\times}$ be the part of $X$ fixed by $K^\times$, and let $Y^\pm$ be the subschemes of $Y$ on which the action of $K^\times$ on the tangent spaces of their points is respectively $t\mapsto(t/\bar{t})^{\pm 1}$. Fix a CM point $P^+\in Y^+(K^{ab})$, and let $P^-$ be the conjugate point of $P^+$.

For arbitrary $f_\pm \in \pi^\pm$, define the following Heegner points on $A^\pm$: $$P_\chi^\pm(f_\pm)=\int_{\overline{K^\times} \backslash \mathbb{A}_K^{\infty \times}}f_\pm(T_tP^\pm)\otimes_M \chi(t)^{\pm 1}dt. $$ where $T_t$ is the Hecke morphism.

    \begin{remark}
The notation here is the one used in Liu--Zhang--Zhang. By class field theory, up to some constants, it is the same as the point $P_\chi(f)$ of Section~\ref{sec:GZ}. The precise difference can be seen in Table~\ref{table:heegnertable}.
    \end{remark}

Suppose $\mathcal{K}$ contains $MF_w^{ab}$, then there is a $p$-adic logarithm map, which is a $\mathcal{K}$-linear map:
$$\log_{A^\pm}:A^\pm(\mathcal{K})\otimes_{M_A}\mathcal{K}\rightarrow Lie A^\pm \otimes_{F\otimes M_A} \mathcal{K}. $$ As a functional on $\pi^+\times \pi^{-}$, $\log_{A^+}P_\chi^+(f_+)\cdot \log_{A^-}P_\chi^-(f_-)$ defines an element in the one-dimensional $\mathcal{K}$-space $$\Hom_{\mathbb{A}_K^{\infty\times}}(\pi^+\otimes \chi,\mathcal{K}) \otimes_{\mathcal{K}} \Hom_{\mathbb{A}_K^{\infty\times}}(\pi^-\otimes \chi^{-1}, \mathcal{K})\otimes_{F\otimes M_A}(Lie A^+ \otimes Lie A^-).$$

        \paragraph{Matrix coefficients}\label{section:matrixcoeffiecents}

The matrix coefficient of the Waldspurger formula is expressed through the following normalized local coefficients. Let $\eta_u$ be the quadratic character determined by $K_u/F_u$. For a character $\chi_u^{(\iota)}: K_u^\times\rightarrow \mathbb{C}$ of $K_u^\times$ and a local $\mathbb{B}_u^\times$-invariant bilinear pairing
    $$(\cdot,\cdot)_u:\pi^+_u \times \pi^-_u \rightarrow \mathbb{C}$$
define the (normalized) local coefficient $$\alpha^\natural(f_{+u},f_{-u};\chi_u^{(\iota)}):= \frac{L(1,\eta_u)L(1,\rho_{A,u}^{(\iota)},Ad)} {\zeta_{F_u}(2)L(1/2,\rho_{A,u}^{(\iota)},\chi^{(\iota)}_u)} \int_{F_u^\times \backslash K_u^\times} (\pi_u^+(t)f_{+u},f_{-u})\chi_u^{(\iota)}dt.$$

    \subsubsection{The $p$-adic Waldspurger formula}
    \label{section:lzzwaldspurger}

    \begin{theorem}[The $p$-adic Waldspurger formula]\label{theorem:padicwaldspurger}
There exists a unique element $$\alpha_\chi(\cdot,\cdot)\in \Hom_{\mathbb{A}_K^{\infty\times}}(\pi^+\otimes \chi, \mathcal{K}) \otimes_{\mathcal{K}} \Hom_{\mathbb{A}_K^{\infty\times}}(\pi^- \otimes \chi^{-1}, \mathcal{K}),$$ such that for arbitrary isomorphisms $\iota: \mathbb{C}_p\cong \mathbb{C}$ and $f_\pm \in \pi^\pm$, one has $$\iota\alpha_\chi(f_+,f_-)=\alpha^\natural(f_+,f_-;\chi^{(\iota)}):=\prod_{u< \infty}\alpha^\natural(f_{+u},f_{-u};\chi_u^{(\iota)}).$$

Furthermore, for $\chi\in \Xi(A,\mathcal{K})_0$, the following $p$-adic Waldspurger formula holds:
$$\log_{A^+}P_\chi^+(f_+)\cdot\log_{A^-}P_\chi^-(f_-)=\mathcal{L}(A)(\chi)\cdot \frac{L(1/2, \rho_{A,w}\otimes \chi_{\bar{v}} )^2}{\epsilon(1/2,\psi,\rho_{A,w} \otimes \chi_{\bar{v}}) }\alpha_\chi(f_+,f_-).$$

This equality is regarded as an equality in $(Lie A^+ \otimes_{F\otimes M_A} Lie A^-)$.
    \end{theorem}

    \begin{proof}
See Theorem 3.3.2 of Liu--Zhang--Zhang \cite{liu2018p}.
    \end{proof}

This formula is not an equality over $\mathbb{C}$ or over $\mathbb{C}_p$. By applying the canonical pairing $\left\langle \omega_+ \otimes \omega_-,\cdot\right\rangle $ gives an equality over $\mathbb{C}_p$. For this purpose, some further notation is needed.

    \subsubsection{The $p$-adic Maass-function version}
    \label{section:lzzmaass}

A function $\phi: X(\mathbb{B})(\mathbb{C}_p)\rightarrow \mathbb{C}_p$ is called a \textbf{$p$-adic Maass function} on $X(\mathbb{B})$ if it is the pullback of a locally analytic function $X(\mathbb{B})_U(\mathbb{C}_p)\rightarrow \mathbb{C}_p$. Let $\mathcal{A}_{\mathbb{C}_p}(\mathbb{B}^\times)$ be the set of all $p$-adic Maass functions. It is a linear space over $\mathbb{C}_p$, equipped with an action of $\mathbb{B}^{\infty\times}$.

Consider $p$-adic Maass functions of the following form $$f^\ast \log_\omega :X(\mathbb{B})(\mathbb{C}_p) \xrightarrow{} A(\mathbb{C}_p) \xrightarrow{\log_\omega} \mathbb{C}_p$$ where $f:X(\mathbb{B})\rightarrow A$ is a nonconstant morphism, $\omega$ is a differential form, and $\log_\omega(-)=\left\langle \omega,\log_A(-)\right\rangle $, where $\left\langle \cdot,\cdot\right\rangle $ is the canonical pairing between $H^0(A,\Omega^1)$ and $Lie(A)$. The space of such $p$-adic Maass functions is automatically a representation of $\mathbb{B}^{\infty\times}$. Consider the subrepresentation on which $M_A$ acts through the inclusion $M_A\subseteq\mathbb{C}_p$, denoted $\pi_{A,Maass}$.

Take a nonzero $\omega \in H^0(A,\Omega_A^1)$, by Lemma 3.4.2 of Liu--Zhang--Zhang \cite{liu2018p}, one can construct a map
    $$\xi: \pi_A \rightarrow \pi_{A,Maass},$$
where the map sending $f$ to $f^\ast\log_\omega|_{X(\mathbb{B})(\mathbb{C}_p)}$ is $\mathbb{B}^{\infty\times}$-equivariant and $M_A$-linear. Moreover, it induces
    $$\pi_A\otimes \mathbb{C}_p\cong \pi_{A,Maass}.$$

Take two nonzero differential forms $\omega_\pm\in H^0(A^\pm,\Omega_{A^\pm}^1)$, and a character $\chi$ of weight $0$. By the above isomorphism there is an isomorphism
    $$\xi_{\omega_\pm}:\pi_{A^\pm} \otimes \mathbb{C}_p \cong \pi_{A^\pm,Maass}.$$
For $\phi_\pm\in \pi_{A^\pm, Maass}$, define $$\alpha_\chi^{\omega_+,\omega_-}(\phi_+,\phi_-)=\alpha_\chi(\xi_{\omega_+}^{-1} \phi_+,\xi_{\omega_-}^{-1}\phi_-)\in \mathbb{C}_p.$$

Next, the Heegner points can be rewritten in terms of Maass functions: for $\phi_\pm \in \pi_{A^\pm, Maass}$ and a character $\chi$ of weight $0$, define the $p$-adic period integral $$\mathcal{P}_{\mathbb{C}_p}(\phi_\pm,\chi^{\pm 1}):=\int_{\overline{K^\times }\backslash \mathbb{A}_K^{\infty\times}}\phi_\pm(T_t P^\pm)\chi(t)^{\pm 1}dt.$$

    \begin{theorem}[The $p$-adic Waldspurger formula for $p$-adic Maass functions]
    \label{theorem:padicwaldspurgerforpadicmaass}
Consider $$\mathcal{L}_{\omega_+,\omega_-}(\pi):=\left\langle  \omega_+\otimes \omega_-,\mathcal{L}(A) \right\rangle $$ as an element of $\mathcal{D}(A,\mathbb{C}_p)$, then for arbitrary $\chi\in \Xi(A,\mathbb{C}_p)_0$ and $\phi_\pm \in \pi_{A^\pm, Maass}$, the following equality over $\mathbb{C}_p$ holds:

    \[
    \mathcal{P}_{\mathbb{C}_p}(\phi_+,\chi)\mathcal{P}_{\mathbb{C}_p}(\phi_-,\chi^{-1})
    = \mathcal{L}_{\omega_+,\omega_-}(\pi)(\chi)\frac{L(1/2, \rho_{A,w}\otimes
    \chi_{\bar{v}} )^2}{\epsilon(1/2,\psi,\rho_{A,w}\otimes \chi_{\bar{v}} ) }
    \alpha_\chi^{\omega_+,\omega_-}(\phi_+,\phi_-).
    \]
    \end{theorem}

    \begin{proof}
Apply $\left\langle \omega_+\otimes \omega_-,\cdot \right\rangle $ to both sides of Theorem~\ref{theorem:padicwaldspurger}, the conclusion follows. (This is also the proof of Theorem 3.4.4 of Liu--Zhang--Zhang \cite{liu2018p}.)
    \end{proof}

    \begin{remark}
Theorem~\ref{theorem:padicwaldspurger} and Theorem~\ref{theorem:padicwaldspurgerforpadicmaass} are exactly the generalizations of the BDP formula and the Brooks formula, we call these two formulas the \textbf{Liu--Zhang--Zhang formulas}.
    \end{remark}

\subsection{An Explicit Liu--Zhang--Zhang Formula}\label{section:explicitlzz}

Based on the formulas of Subsection~\ref{section:lzzpadicLfunction} and the work of Cai--Shu--Tian \cite{cai2014explicit}, we now compute an explicit Liu--Zhang--Zhang formula. The ultimate goal is to obtain an analogue of Proposition 5.1.7 of Jetchev--Skinner--Wan \cite{jetchev5birch}.

    \subsubsection{Reduction to the case $A=E$}
    \label{section:cstreduction}
When $A=E$ is the elliptic curve over $F$ fixed in this article, one has $A^+=A^-=E$, and the central character of $A$ is trivial. In this case $M_A=\mathbb{Q}$.

By the basic assumption $F_w\cong \mathbb{Q}_p$ of this article, one has $F_w^{lt}=\mathbb{Q}_p^{ur}$, the maximal unramified extension of $\mathbb{Q}_p$.

Recall that the conductor of $E$ is $\mathfrak{N}=\mathfrak{MD}$, then the ramification set of the quaternion algebra $\mathbb{B}$ in this section is $\Sigma(\mathbb{B})=\{u|\mathfrak{D}\}\cup \{u|\infty\}$.

$K/F$ is a totally imaginary quadratic extension of the totally real field, with relative discriminant $D\subseteq \mathcal{O}_F$. Assume $(D,\mathfrak{N})=1$ and that $K$ satisfies the Heegner hypothesis (\ref{Heegnerassumption}).

Under the above assumptions, all places $u$ of $F$ can be divided into five mutually disjoint classes:
    \begin{itemize}
\item $u|\infty$ are the infinite places.

\item $u|\mathfrak{M}$ are finite places for which the local field extension is split.

\item $u|\mathfrak{D}$ are finite places for which the local field extension is inert.

\item $u|D$ are finite places for which the local field extension is ramified.

\item in the remaining cases, the local field extension is unramified.

    \end{itemize}

    \subsubsection{Local computations of Cai--Shu--Tian}
    \label{section:cstlocal}
For the convenience of computation, we list the results for the various local quantities at the different places. These results mainly come from Section 3 of Cai--Shu--Tian \cite{cai2014explicit}.

        \subsubsection{Notation and theorems of Cai--Shu--Tian}
We restate part of the notation of Cai--Shu--Tian, combined with the situation of this article.

        \begin{itemize}
\item Write $q_u$ for the absolute norm of the prime ideal corresponding to the finite place $u$.

\item Write $\delta_F$ for the different ideal of $F/\mathbb{Q}$, $\delta_{F,u}$ for the local different ideal of $F_u/\mathbb{Q}_\ell$, where $u|\ell$, and $D_F$ for the absolute discriminant of $F/\mathbb{Q}$.

\item $D$, as an ideal of $\mathcal{O}_F$, is the relative discriminant of $K/F$, and $D_u$ is the local relative discriminant.

\item $\sigma=\otimes_{u\le \infty} \sigma_u$ is the $GL_2$ representation of $\pi_E$ obtained by the Jacquet-Langlands correspondence, this is also the automorphic representation attached to $E$.

\item for $\sigma_u$, write $n_u=n(\sigma_u)$ for the conductor of $\sigma_u$, in the situation of this article, $n_u=1$ if and only if $u|\mathfrak{N}$, and $n=0$ otherwise.

\item $\Tilde{\sigma}$ is the contragredient representation of $\sigma$.

\item for arbitrary $g \in \sigma $, write $W_g$ for the Whittaker function associated with $g$.

\item define the Petersson pairing of $g_1\in \sigma,g_2\in \Tilde{\sigma}$ to be $$\left\langle g_1,g_2\right\rangle _{Pet,GL_2}:=\int_{Z(\mathbb{A})GL_2(F) \backslash GL_2(\mathbb{A})}g_1(t)g_2(t)dt. $$

\item $\chi$ is a Hecke character of $K$, in the sequel only the trivial character is considered, so the conductor of $\chi$ is trivial, and by the notation of Cai--Shu--Tian, every local $c_u=0$.

\item For $f_1=\otimes f_{1,u}$, $f_2=\otimes f_{2,u}\in \pi_E=\otimes \pi_{E,u}$, choose local pairings so that $(\cdot,\cdot)_E=\otimes (\cdot,\cdot)_{u}$, and define
$$\beta_u(f_{1,u},f_{2,u}):=\frac{L(1,\eta_u)L(1,\pi_u,Ad)} {L(1/2,\pi_u,\chi_u)\zeta_{F_u}(2)}\int_{K_u^\times/F_u^\times} \frac{(\pi(t_u)f_{1,u},f_{2,u})_u}{(f_{1,u},f_{2,u})_u}\chi_u(t_u)dt_u.$$ Moreover, if $f=f_1=f_2$, one also writes $\beta_u(f_u)=\beta_u(f_u,f_u)$ for short.
        \end{itemize}

The following propositions use the new notation above. First, we state the result on the Petersson pairing.

        \begin{proposition}\label{prop:adjoint}
For $g_1\in \sigma,g_2\in\Tilde{\sigma}$ with $W_{g_i}=\otimes W_{i,u}, i=1,2$, the following formula holds:

$$\left\langle g_1,g_2\right\rangle _{Pet,GL_2}=\frac{2\Lambda(1,\sigma,Ad)}{\Lambda_F(2)} \prod_u \alpha_u(W_{1,u},W_{2,u}),$$ where

$$\alpha_u(W_{1,u},W_{2,u})=\frac{\zeta_{F_u}(2)}{L(1,\sigma_u,Ad) \zeta_{F_u}(1)}\left\langle W_{1,u},W_{2,u}\right\rangle _{W},$$
and $\left\langle \cdot,\cdot\right\rangle _{W}$ is the pairing between Whittaker functions, and here $\Lambda$ denotes the completed $L$-function.
        \end{proposition}

        \begin{proof}
See Proposition 2.1 of Cai--Shu--Tian \cite{cai2014explicit}.
        \end{proof}

        \begin{proposition}
Write $W_{u,0}$ for the normalized new vector of $\sigma_u$. If $u$ is a finite place, then
            \begin{center}
            $\alpha_u(W_{0,u})|\delta_{F,u}|^{1/2} = \left\{
            \begin{aligned}
            &1 & \sigma_u\text{ is unramified}, \cr
            &\frac{\zeta_{F_u}(2)}{\zeta_{F_u}(1)}L(1,\sigma_u,Ad)^{-a_{\sigma_u}} &
            \text{otherwise}.
            \end{aligned}
            \right.$
            \end{center}
where $a_{\sigma_u}\in\{0,1\}$, and $a_{\sigma_u}=0$ if and only if $\sigma_u$ is a subrepresentation of $Ind(\mu_1,\mu_2)$ with at least one $\mu_i$ unramified.

When $u$ is an infinite place and $\sigma_u$ is a discrete series of weight $k$, one has $\alpha_u(W_{0,u})=2^{-k}$.
            \end{proposition}

            \begin{proof}
See Proposition 3.11 of Cai--Shu--Tian \cite{cai2014explicit}.
            \end{proof}

            \begin{proposition}
For a finite place $u$, consider the parametrization $f:X\rightarrow E$, then
            \begin{center}
            $\beta_u(f_u)|\delta_{F,u} D_u |^{-1/2} = \left\{
            \begin{aligned}
            &1 & \text{if }n(\sigma_u)=0, \cr
            &\frac{\zeta_{F_u}(1)}{\zeta_{F_u}(2)}L(1,\pi_{E,u},Ad)^{a_{\sigma_u}} &
            \text{if }n>0 \text{ and } K_u \text{ is split}, \cr
            &e(1-q_u^{-e})\frac{L(1,\pi_{E,u},Ad)}{L(1/2,\pi_{E,u},\chi_u)} & \text{if
            }n>c_u \text{ and } K_u \text{ is nonsplit},
            \end{aligned}
            \right.$
            \end{center}
            \end{proposition}

            \begin{proof}
See Proposition 3.12 of Cai--Shu--Tian \cite{cai2014explicit}. Here, since $c_u=0$, invoking this proposition directly ignores the cases with $c_u>0$ in the original proposition.
            \end{proof}

            \begin{remark}
Concerning the relation between $\beta$ and the $\alpha^\natural$ of Paragraph~\ref{section:matrixcoeffiecents} or Theorem~\ref{theorem:padicwaldspurger}, let $f$ still be the map parametrizing $E$ by $X$. By definition they satisfy the following relation
$$\sqrt{|D_K|} \prod_{u<\infty}\beta_u(f_u) \times \frac{(f,f)_E}{(f,f)_\infty}= \alpha^\natural (f,f;1). $$ The reason is nothing other than the fact that $\alpha^\natural$ is the product over all finite places, and the choice of Haar measures.
            \end{remark}

        \subsubsection{A Table of classification}

We collect the needed results in the following table.

        \begin{proposition}\label{prop:cstlocal}
Using the notation and assumptions of this section, one has the following table:

        \begin{longtable}{L{3.4cm}L{1.9cm}L{2.0cm}L{2.0cm}L{1.9cm}L{1.8cm}}
        \caption{The quantities computed in Cai--Shu--Tian}
        \label{CSTcomputation}\\
        \toprule
                              & \multicolumn{2}{c}{$\in \Sigma(\mathbb{B})$}
                              & \multicolumn{3}{c}{$\notin \Sigma(\mathbb{B})$}    \\
        \midrule
        \endfirsthead
        \endhead
        $u$ &
          \multicolumn{1}{c}{$u\mid\infty$} &
          \multicolumn{1}{c}{$u\mid\mathfrak{D}$} &
          \multicolumn{1}{c}{$u\mid\mathfrak{M}$} &
          \multicolumn{1}{c}{$u\mid D$} &
          otherwise \\
        \midrule
        $\mathbb{B}_u$        & \multicolumn{2}{c}{Division}
                              & \multicolumn{3}{c}{$M_2(F_u)$}                     \\
        $F_u$                 & \multicolumn{1}{c}{$\mathbb{R}$}
                              & \multicolumn{4}{c}{$F_u$}                                                   \\
        $K_u$ &
          \multicolumn{1}{c}{$\mathbb{C}$} &
          \multicolumn{1}{c}{inert} &
          \multicolumn{1}{c}{split} &
          \multicolumn{1}{c}{ramified} &
          unramified \\

        $n=n(\sigma_u)$       & \multicolumn{1}{c}{}
                              & \multicolumn{1}{c}{1} & \multicolumn{1}{c}{1}
                              & \multicolumn{1}{c}{0} & 0 \\
        $c_u$                   & \multicolumn{1}{c}{}
                              & \multicolumn{1}{c}{0} & \multicolumn{1}{c}{0}
                              & \multicolumn{1}{c}{0} & 0 \\
        $\beta_u(f_u)|D_u\delta_{F,u}|_u^{-1/2}$ &
          \multicolumn{1}{c}{} &
          \multicolumn{1}{c}{$1-q_u^{-1}$} &
          \multicolumn{1}{c}{$1+q_u^{-1}$} &
          \multicolumn{1}{c}{1} &
          1 \\
        $\alpha_u(W_{0,u})|\delta_{F,u}|_u^{1/2}$ &
          \multicolumn{1}{c}{$2^{-k}$} &
          \multicolumn{1}{c}{$\frac{1}{1+q_u^{-1}}$} &
          \multicolumn{1}{c}{$\frac{1}{1+q_u^{-1}}$} &
          \multicolumn{1}{c}{1} &
          1 \\
        \bottomrule
        \end{longtable}
        \end{proposition}

        \begin{proof}
This is filling in the propositions of the previous subsubsection.
        \end{proof}

    \subsubsection{Comparison of the $p$-adic $L$-functions when $F=\mathbb{Q}$}
    \label{section:fcomp}
Although, in essence, Liu--Zhang--Zhang generalized the work of Bertolini--Darmon--Prasanna, as well as the work of Brooks, we still need some work for Iwasawa theory or for proving the Birch and Swinnerton-Dyer formula. Specifically, the Liu--Zhang--Zhang $p$-adic $L$-function and the anticyclotomic $p$-adic $L$-function over $\mathbb{Q}$ may differ by some quantities at the level of interpolation, and these quantities have a crucial influence on proving the exact version of the Birch and Swinnerton-Dyer formula.

This subsubsection focuses on the interpolation of these two types of $p$-adic $L$-functions when $F=\mathbb{Q}$, and gives the $p$-adic $L$-function needed in this article. It is to be expected that this $p$-adic $L$-function is exactly the Liu--Zhang--Zhang $p$-adic $L$-function multiplied by some constants.

        \paragraph{The interpolation formula of Jetchev--Skinner--Wan}

We first recall the interpolation formula of Jetchev--Skinner--Wan \cite[Section 5.1]{jetchev5birch}. In this case $K$ is an imaginary quadratic field, and $\mathbf{f}$, which is associated with $E$, is a modular form of weight $2$. Here $\mathfrak{N}=N$ is a positive integer, and $N=\mathfrak{MD}=N^+\cdot N^-$.

Consider the continuous characters $\Psi:G_K\rightarrow \mathbb{Q}_p^\times$ satisfying the following conditions, and denote the set of them by $\Sigma_{cc}$:
        \begin{itemize}
\item $\Psi^c=\Psi^{-1}$.

\item $\Psi$ factors through $\Gamma$, the Galois group of the anticyclotomic $\mathbb{Z}_p$ extension of $K$.

\item $\Psi$ is crystalline at $v,\bar{v}$.

\item $\Psi$ has Hodge--Tate weight $-k<0$ at $v$.
        \end{itemize}

By global class field theory one can give an algebraic Hecke character $\Psi^{alg}$ such that $\Psi_\infty^{alg}(z)=z^k\bar{z}^{-k}$.

Let $\mathcal{R}$ be a finite extension of the completion of $\mathbb{Q}_p^{ur}$, and $\Lambda:=\mathcal{R}[[\Gamma]]$. The above $\Psi$ can be extended to $\Lambda$. There exists an $L_p(\mathbf{f})\in \Lambda$ such that for $\Psi \in \Sigma_{cc}$ with $k \equiv 0 \mod p-1$, the following interpolation formula holds:
        $$
        L_p(\mathbf{f},\Psi):=\Psi(L_p(\mathbf{f}))=E_{\bar{v}}(\mathbf{f},\Psi)^2
        \cdot t_K \cdot \frac{C(\mathbf{f},\Psi)}{\alpha(\mathbf{f},\mathbf{f}_B)
        W(\mathbf{f},\Psi)}\cdot \Omega_{p}^{4k}\frac{L(\mathbf{f},\Psi^{alg},1)}
        {\Omega_\infty^{4k}}
        $$
where each term is as follows:
        \begin{itemize}
\item $E_{\bar{v}}(\mathbf{f},\Psi)=(1-\Psi^{alg}(\varpi_{\bar{v}})a_pp^{- 1}+\Psi^{alg}(\varpi_{\bar{v}})^2p^{-1})$, where $\varpi_{\bar{v}}$ is a uniformizer of $K_{\bar{v}}$.

\item $t_K$ is a power of $2$, depending only on $K$.

\item $C(\mathbf{f},\Psi)=\frac{1}{4}\pi^{2k-1}\Gamma(k)\Gamma(k+1)w_K^2 \sqrt{|D_K|}\prod_{\ell|N^{-}}\frac{\ell-1}{\ell+1}$, where $w_K$ is the number of roots of unity in $K$.

\item $W(\mathbf{f},\Psi)$ is a $p$-adic unit, since $p$ does not divide the conductor of $E$.

\item $\alpha(\mathbf{f},\mathbf{f}_B)=\frac{\left\langle \mathbf{f},\mathbf{f} \right\rangle _{\Gamma_0(N)} }{\left\langle \mathbf{f}_B,\mathbf{f}_B \right\rangle _{\Gamma_0^B(N^+)}}$ is the ratio of Petersson pairings.

\item $\Omega_p$ and $\Omega_\infty$ are the $p$-adic period and the complex period, respectively.
        \end{itemize}

        \begin{remark}
Here, by computation, we adopt the result of Ming-Lun Hsieh \cite{hsieh2014special}. In $C(\mathbf{f},\Psi)$ on the right-hand side, the exponent of $w_K$ is $2$. In Jetchev--Skinner--Wan \cite[(5.1 a)]{jetchev5birch} the same constant appears with $w_K$ to the first power. Since $w_K$ is a $p$-adic unit (as shown in the proof of Theorem~\ref{theorem:LZZoftrivialcharacter} below), the two normalizations agree up to $p$-adic units, which is all the comparison below needs.
        \end{remark}

        \paragraph{Automorphic preliminaries and periods}
        \label{section:automorphicperiods}

The comparison of the two interpolation formulas uses the complex model of the Shimura curve and a number of period normalizations, we collect them here.

For the quaternion algebra $\mathbb{B}$, an $\iota$-nearby quaternion algebra datum ($\iota$-nearby data) of $\mathbb{B}$ is
        \begin{itemize}
\item a quaternion algebra $B(\iota)/F$ such that $B(\iota)_u$ is definite at every infinite place $u\neq \iota|_F$.

\item some isomorphisms $B(\iota)_u\cong \mathbb{B}_u$ for finite places $u\neq w$.

\item an isomorphism $B(\iota)_u\cong M_2(\mathbb{R})$ for $u=\iota|_F$.
        \end{itemize}
In fact, $B(\iota)/F$ can be understood as the quaternion algebra obtained by adjusting the ramification at the infinite place $\iota|_F$. Thus the complex points of the Shimura curve $X$ can be described through $B(\iota)$. Write $X_\iota$ for the Shimura curve obtained from $X$ by base change to $\mathbb{C}$ via $\iota$, and $\check{X}_\iota$ for the underlying real analytic space of $X_\iota$.

Let $\mathcal{A}^{(2k)}(B(\iota)^\times)$ be the space of real analytic automorphic forms on $B(\iota)^\times(\mathbb{A})$ of weight $2k$ at $\iota|_F$, $k\ge1$, and invariant under the action of $B(\iota)_u$ at the other infinite places. Lemma 2.4.15 of Liu--Zhang--Zhang \cite{liu2018p} gives an isomorphism:
        $$
        \phi_\iota: H^0(\check{X}_\iota,(\check{\Omega}_{X_\iota}^1)^{\otimes k})
        \rightarrow \mathcal{A}^{2k}(B(\iota)^\times).
        $$

A rather crucial period in Liu--Zhang--Zhang \cite{liu2018p} is the Lubin--Tate differential
        $$
        \omega_\nu:=c^\ast\nu^\ast \frac{dT}{T},
        $$
where $dT/T$ is the differential form on $ \hat{\mathbb{G}}_m$, $\nu: \mathcal{LT}\rightarrow \hat{\mathbb{G}}_m$ is a generator of $\Hom(\mathcal{LT},\hat{\mathbb{G}}_m)$, and $c:\mathfrak{X}(\infty) \rightarrow \mathcal{LT}$ is the classifying morphism (see Liu--Zhang--Zhang \cite{liu2018p}, p.~766). It can be regarded as an element of $H^0(\mathfrak{X}(\infty),\Omega_{\mathfrak{X}(\infty)}^1)\hat{\otimes}F_w^{lt}$. Furthermore, write
        $$
        \omega_{\psi_\pm}:=\Upsilon^\ast_\pm (\omega_{\nu})|_{Y^\pm\otimes
        F_w^{lt}F_w^{ab} }.
        $$
where $\Upsilon_\pm:X(\pm \infty)\otimes F_w^{ab}\rightarrow \mathcal{X}(\infty)\otimes F_w^{ab}$ are the transition isomorphisms (see Liu--Zhang--Zhang \cite{liu2018p}, p.~764). Here $X(\pm \infty)$ are the schemes given by the Shimura curves through a specific limiting procedure, and $\mathcal{X}(\infty)$ is a certain model of $X(\pm \infty)$ whose formal completion is the above $\mathfrak{X}(\infty)$.

Consider $\iota(P^\pm)=[\pm i,t_\pm]\in Y_\iota^\pm(\mathbb{C})$, a representative of the complex parametrization, where $t_\pm \in \mathbb{A}_K^{\infty \times}$. Then define $\zeta_{\iota,\pm} \in \mathbb{C}^\times$ to be the constant for which the following equality holds:
        $$
        dz([\pm i,t_\pm])=\zeta_{\iota,\pm} \cdot
        \iota(\omega_{\psi\pm})|_{P^\pm}.
        $$

We also need the periods of CM elliptic curves. Consider a finite extension $F_0$ of the Hilbert class field of an imaginary quadratic field $K_0$, and let $E_0/F_0$ be a CM elliptic curve whose CM field is $K_0$. We may assume that $E_0$ has good reduction at the prime ideals of $F_0$ above $p$, let $\mathcal{E}_0$ be a model of $E_0$ over $F_{0,(p)}$. Write $\omega_{E_0}$ for a generator of the differentials of $H^0(\mathcal{E}_0,\Omega_{\mathcal{E}_0})$. Then, through $\iota$, one has
        $$
        H^0(\mathcal{E}_0,\Omega_{\mathcal{E}_0}) \otimes \mathbb{C}\cong
        H^0(E_0/\mathbb{C}),
        $$
where the right-hand side carries a standard invariant differential $\omega_{\mathbb{C}}=dz$, so that one can define the following infinite period $\Omega_\infty \in \mathbb{C}^\times$ such that:
        $$
        \omega_{E_0}=\Omega_\infty \cdot \omega_{\mathbb{C}}.
        $$

Similarly, following the convention of regarding all fields as subfields of $\mathbb{C}_p$, one can define the following $p$-adic period. Consider the model $\mathcal{E}_0/R$ of $E_0$ over the ring of integers of $F_{0,\mathfrak{p}_0}^{ur}$, and $\hat{\mathcal{E}_0}\cong \hat{\mathbb{G}}_m$ the completion at the origin. Fix an isomorphism (this isomorphism is not canonical, but changing it only changes the period by a $p$-adic unit), and define the pullback of the differential $dt/t$ on $\hat{\mathbb{G}}_m$ as
        $$
        \omega_{can}\in H^0(\mathcal{E}_0/R,\Omega^1_{\mathcal{E}_0/R}).
        $$
Define the following $p$-adic period $\Omega_p \in \mathbb{C}_p^\times$ such that:
        $$
        \omega_{E_0}=\Omega_p \cdot  \omega_{can}.
        $$

For CM false elliptic curves parametrized by Shimura curves, similar periods can be defined, and the relations between them can be explained, see Jetchev--Skinner--Wan \cite[Section~4.5]{jetchev5birch}.

        \paragraph{A consequence of the Kodaira--Spencer map}
        \label{subsubsection:kodairdaspencer}

Let $\lambda :\mathcal{E}\rightarrow X$ be the universal elliptic curve on the Shimura curve. Suppose there is a Kodaira--Spencer map, as a map of sheaves on $X(\pm \infty)$,
        $$
        KS:(\lambda_\ast \Omega_{\mathcal{E}/X(\pm \infty)})^{\otimes 2} \rightarrow
        \Omega_{X(\pm \infty)}.
        $$
Suppose the curve corresponding to a CM point $P$ is $E_0$, then by definition one has
        $$
        KS(\omega_{can}^{\otimes 2})= \omega_{\psi}|_P.
        $$

At the same time, consider the $\mathbb{C}$-version. By the proof of Proposition 3.4 of Brooks \cite{brooks2015shimura}, one has
        $$
        \iota (KS(dz_0^{\otimes 2}))= \frac{1}{2\pi i}dz|_P.
        $$
Now, by definition
        $$
        dz|_P=\zeta_\iota \omega_{\psi}|_P.
        $$
Using $dz|_P=\zeta_\iota \omega_{\psi}|_P$, $\omega_{E_0}=\Omega_p \cdot \omega_{can}$, and $\omega_{E_0}=\Omega_\infty \cdot dz_0$, from the three formulas, and converting both sides into $\omega_{E_0}$ and comparing coefficients, one obtains
        $$
        \zeta_\iota=\zeta_{\iota}(P)=2\pi i \cdot
        \left(\frac{\iota \Omega_p}{  \Omega_\infty}\right)^2.
        $$

        \paragraph{The interpolation comparison}
        \label{subsubsection:interpolation}

Take $\omega\in H^0(E,\Omega_E^1)$ to be a fixed N\'eron differential. By Theorem~\ref{theorem:lzzpadicLFUNCTION}, i.e. Theorem 3.2.10 of Liu--Zhang--Zhang \cite{liu2018p}, for a character $\chi$ of weight $k>0$, the following equality holds:
        \begin{align*}
        \left\langle \omega\otimes\omega,\iota\mathcal{L}(E)(\chi)\right\rangle \iota(f_+,f_-)_E
        =&\frac{\epsilon(1/2,\psi,\pi_{E,p}\otimes \chi_{\bar{v}}^{(\iota)})}
        {L(1/2,\pi_{E,p}\otimes \chi_{\bar{v}}^{(\iota)})^2}
        \frac{\sqrt{|D_K|}\zeta_{\mathbb{Q}}(2)L(1/2,\pi_E,\chi^{(\iota)})}
        {L(1,\eta)^2L(1,\pi_E,Ad)}\\
        &\cdot \left\langle \omega\otimes\omega,\mathbf{P}_\iota(E,\chi)\right\rangle
        \iota(f_+,f_-)_E.
        \end{align*}

Multiplying both the numerator and the denominator of the right-hand side of this equality by $\iota( \varphi_+,\varphi_-)_\chi$, by Lemma~\ref{lemma:lzzlemma} one has the following equality:
        \begin{align*}
        \left\langle \omega\otimes\omega,\iota\mathcal{L}(E)(\chi)\right\rangle \iota(f_+,f_-)_E
        =&\frac{\epsilon(1/2,\psi,\pi_{E,p}\otimes \chi_{\bar{v}}^{(\iota)})}
        {L(1/2,\pi_{E,p}\otimes \chi_{\bar{v}}^{(\iota)})^2} \cdot
        \frac{\sqrt{|D_K|}\zeta_{\mathbb{Q}}(2)L(1/2,\pi_E,\chi^{(\iota)})}
        {L(1,\eta)^2L(1,\pi_E,Ad)}\\
        &\cdot \frac{(\iota\varphi_+\otimes c_\iota^\ast \iota \varphi_-\otimes
        \mu^k)}{\iota( \varphi_+,\varphi_-)_\chi}
        \int_{X_\iota(\mathbb{C})}\mu^{-k}\,\Theta_\iota^{k-1}f_+^\ast\omega\otimes
        c_\iota^\ast\Theta_\iota^{k-1}f_-^\ast\omega\,dx.
        \end{align*}

Then, by the formula on p.~811 of Liu--Zhang--Zhang \cite{liu2018p}, the product of the last two lines has the following formula:

        \begin{align*}
        &\frac{\iota\varphi_+\otimes c_\iota^*\iota\varphi_-\otimes\mu^k}
        {\iota(\varphi_+, \varphi_-)_\chi} \int_{X_\iota(\mathbb{C})}
        \frac{\Theta_{\iota}^{k-1}f_+^*\omega\otimes c_{\iota}^*\Theta_{\iota}^{k-1}
        f_-^*\omega}{\mu^k}\,dx \\
        &= (\zeta_{\iota}^+\zeta_{\iota}^-)^k \cdot \chi^{({\iota})}(t_+^{-1}t_-)
        \cdot \left\langle  \Delta_{+,{\iota}}^{k-1}\bigl(
        \phi_{\iota}(f_+^*\omega)\bigr), \Delta_{-,{\iota}}^{k-1}\bigl(\mathrm{R}(j)
        \phi_{\iota}(f_-^*\omega)\bigr) \right\rangle _{Pet,B}.
        \end{align*}
where $R(j)$ is the translation action by $j=diag(1,-1)$. Then, by Lemma 3.4.6 of Liu--Zhang--Zhang \cite{liu2018p}, the Shimura--Maass operators in the Petersson inner product can be removed, i.e. one has
        \[
        \left\langle  \Delta_{+,{\iota}}^{k-1}\bigl(
        \phi_{\iota}(f_+^*\omega)\bigr), \Delta_{-,{\iota}}^{k-1}\bigl(\mathrm{R}(j)
        \phi_{\iota}(f_-^*\omega)\bigr) \right\rangle _{Pet,B}
        =\frac{k!(k-1)!}{4^{k-1}}\left\langle 
        \phi_{\iota}(f_+^*\omega), \mathrm{R}(j)\phi_{\iota}(f_-^*\omega)\right\rangle _{Pet,B}.
        \]

Collecting terms, one has the following formula:
        \begin{equation*}
        \begin{split}
        \left\langle \omega\otimes\omega,\iota\mathcal{L}(E)(\chi)\right\rangle \iota(f_+,f_-)_E
        &=\frac{\epsilon(1/2,\psi,\pi_{E,p}\otimes \chi_{\bar{v}}^{(\iota)})}
        {L(1/2,\pi_{E,p}\otimes \chi_{\bar{v}}^{(\iota)})^2} \cdot(\zeta_{\iota}^+
        \zeta_{\iota}^-)^k \cdot \chi^{({\iota})}(t_+^{-1}t_-)\\
        &\quad \frac{\sqrt{|D_K|}k!(k-1)!L(1/2,\pi_E,\chi^{(\iota)})}{4^kL(1,\eta)^2}
        \\
        &\quad \left\langle  \phi_{\iota}(f_+^*\omega), \mathrm{R}(j)\phi_{\iota}(f_-^*\omega)
        \right\rangle _{Pet,B} \cdot \frac{\zeta_{\mathbb{Q}}(2)}{L(1,\pi_E,Ad)}.
        \end{split}
        \end{equation*}

Now we handle the last factor. By Proposition~\ref{prop:adjoint} one has the following formula

$$\left\langle g_1,g_2\right\rangle _{Pet}=\frac{2\Lambda(1,\sigma,Ad)}{\Lambda_F(2)} \prod_u \alpha_u(W_{1,u},W_{2,u}),$$

Suppose $\phi_\iota(f_+^\ast \omega)$ and $\phi_\iota(f_-^\ast \omega)$ correspond to the automorphic forms $g_+,g_-$ on $GL_2$, since the Haar measures in Cai--Shu--Tian differ by a factor of $\pi$, one has

        $$
        \frac{\zeta_{\mathbb{Q}}(2)}{L(1,\pi_E,Ad)}=\frac{2\pi}{\left\langle g_+,g_-\right\rangle 
        _{Pet,GL_2}}\cdot\frac{L_\infty(1,\sigma,Ad)}{\Gamma_{\mathbb{R}}(2)}\cdot
        \prod_{u\le \infty} \alpha_u.
        $$
where $\Gamma_{\mathbb{R}}(s)=\pi^{-s/2}\Gamma(s)$ and, by Proposition 1.11 of Cai--Shu--Tian \cite{cai2014explicit}, one has (through ratios)
        
        $$
        L_{\infty}(1,\sigma,Ad)=2^{1+k_\infty}\cdot
        \frac{\Gamma(k_\infty)}{4^{k_\infty}\pi^{k_\infty+1}}.
        $$
Here $k_\infty$ is the weight of $\sigma$ at the infinite place,  in the situation discussed in this article $k_\infty=2$, i.e.:
        $$
        L_{\infty}(1,\sigma,Ad)=2^{1+k_\infty}\cdot
        \frac{\Gamma(k_\infty)}{4^{k_\infty}\pi^{k_\infty+1}}=\frac{1}{2\pi^3}.
        $$

By Proposition~\ref{prop:cstlocal} one knows
        $$\prod_u \alpha_u=2^{-2}\prod_{\ell |N}\frac{1}{1+\ell^{-1}},$$
therefore, one has
        $$
        \frac{\zeta_{\mathbb{Q}}(2)}{L(1,\pi_E,Ad)}=\frac{1}{4\pi \left\langle g_+,g_-\right\rangle 
        _{Pet,GL_2}}\cdot \prod_{\ell |N}\frac{1}{1+\ell^{-1}}.
        $$

Furthermore, substituting this equality into the interpolation formula, one obtains

        \begin{align*}
        \left\langle \omega\otimes\omega,\iota\mathcal{L}(E)(\chi)\right\rangle \iota(f_+,f_-)_E
        =&\frac{\epsilon(1/2,\psi,\pi_{E,p}\otimes \chi_{\bar{v}}^{(\iota)})}
        {L(1/2,\pi_{E,p}\otimes \chi_{\bar{v}}^{(\iota)})^2} \cdot
        (\zeta_{\iota}^+\zeta_{\iota}^-)^k \cdot \chi^{({\iota})}(t_+^{-1}t_-)\\
        &\cdot \frac{\sqrt{|D_K|}k!(k-1)!\prod_{\ell |N}\frac{1}{1+\ell^{-1}}
        }{4^{k+1}L(1,\eta)^2\pi}\\
        &\cdot \frac{\left\langle  \phi_{\iota}(f_+^*\omega), \mathrm{R}(j)\phi_{\iota}(f_-^*\omega)
        \right\rangle _{Pet,B}}{ \left\langle g_+,g_-\right\rangle _{Pet,GL_2}}\cdot
        L(1/2,\pi_E,\chi^{(\iota)}).
        \end{align*}

Now suppose $g_1=g_2=\mathbf{f}'$ is the new vector of the adelic automorphic form on $GL_2$ corresponding to the modular form $\mathbf{f}$. Then
        $$
        \frac{\left\langle  \phi_{\iota}(f_+^*\omega), \mathrm{R}(j)\phi_{\iota}(f_-^*\omega)
        \right\rangle _{Pet,B}}{ \left\langle g_+,g_-\right\rangle _{Pet,GL_2}}=\alpha(\mathbf{f},\mathbf{f}_B)^{-
        1}.
        $$

Recall that the definitions of $\zeta_\iota^+$ and $\zeta_\iota^-$ and $t_\pm$ come from the choice of the CM points $P^\pm$. Now take $P^+=P^-$, then $t_+=t_-$ and $\zeta_\iota^+=\zeta_\iota^-$. Moreover, by the result of Paragraph~\ref{subsubsection:kodairdaspencer}, one obtains
       $$
       (\zeta_{\iota}^+\zeta_{\iota}^-)^k \cdot \chi^{({\iota})}(t_+^{-1}t_-)=
       (2\pi i)^{2k} \frac{\iota \Omega_p^{4k}}{\Omega_\infty^{4k}}.
       $$

Now take $f_+=f_-=\phi$ to be the parametrization of $E$. In this way, the above equality becomes:
        \begin{equation*}
        \begin{split}
        \left\langle \omega\otimes\omega,\iota\mathcal{L}(E)(\chi)\right\rangle \iota(\phi,\phi)_E
        &=(-1)^k\frac{\epsilon(1/2,\psi,\pi_{E,p}\otimes \chi_{\bar{v}}^{(\iota)})}
        {L(1/2,\pi_{E,p}\otimes \chi_{\bar{v}}^{(\iota)})^2} \cdot
        \frac{\iota \Omega_p^{4k}}{\Omega_\infty^{4k}}\\
        &\quad \frac{\sqrt{|D_K|}k!(k-1)!\prod_{\ell |N}\frac{1}{1+\ell^{-1}}
        \pi^{2k-1} }{4L(1,\eta)^2 \alpha(\mathbf{f},\mathbf{f}_B)}   \\
        &\quad    L(1/2,\pi_E,\chi^{(\iota)}).
        \end{split}
        \end{equation*}

Next, multiply both sides of the equality by $\sqrt{|D_K|}\prod_{u< \infty}\beta_u(\phi_u)$. First, by Proposition~\ref{prop:cstlocal}, one has:
        $$
        \sqrt{|D_K|}\prod_{u< \infty}\beta_u(\phi_u)=\prod_{\ell|N^+}(1+\ell^{-1})
        \prod_{\ell|N^-}(1-\ell^{-1}).
        $$

Substituting this in, one has:

        \begin{equation*}
        \begin{split}
        \left\langle \omega\otimes\omega,\iota\mathcal{L}(E)(\chi)\right\rangle \alpha^\natural
        (\phi,\phi.\chi^{(\iota)})
        &=(-1)^k\frac{\epsilon(1/2,\psi,\pi_{E,p}\otimes \chi_{\bar{v}}^{(\iota)})}
        {L(1/2,\pi_{E,p}\otimes \chi_{\bar{v}}^{(\iota)})^2} \cdot
        \frac{\iota \Omega_p^{4k}}{\Omega_\infty^{4k}}\\
        &\quad \frac{\sqrt{|D_K|}k!(k-1)!\prod_{\ell |N^-}\frac{\ell -1}{\ell+1}
        \pi^{2k-1} }{4L(1,\eta)^2 \alpha(\mathbf{f},\mathbf{f}_B)}   \\
        &\quad    L(1/2,\pi_E,\chi^{(\iota)}).
        \end{split}
        \end{equation*}

    \subsubsection{The $p$-adic $L$-function in the sense of Iwasawa theory}
    \label{section:iwasawa}
By the previous subsubsection, we will first define a $p$-adic $L$-function in the sense of Iwasawa theory from the interpolation formula. After discussing some issues about Heegner points, we give the key theorem of this section, namely the trivial-character Liu--Zhang--Zhang formula.

        \paragraph{Definition of the $p$-adic $L$-function in the sense of Iwasawa theory}
        \label{section:iwasawadef}
When $F=\mathbb{Q}$, the Liu--Zhang--Zhang $p$-adic $L$-function differs from the $p$-adic $L$-function used by Jetchev--Skinner--Wan to prove the Birch and Swinnerton-Dyer formula by constants. Using the quantities by which they differ, define the following $p$-adic $L$-function in the sense of Iwasawa theory

        \begin{definition}[The $p$-adic $L$-function in the sense of Iwasawa theory]
For an elliptic curve $E$ over a totally real field $F$, $K/F$ a totally imaginary quadratic extension, and $\phi:X\rightarrow E$ a parametrization of $E$ by the Shimura curve $X$, call $$\mathcal{L}_{E,K,\phi}^{IW}(-):=(w_K^2 \cdot L(1,\eta)^2) \times \left\langle \omega\otimes \omega, \iota\mathcal{L}(E)(-)\right\rangle  \alpha^\natural(\phi,\phi,-) $$the $p$-adic $L$-function in the sense of Iwasawa theory (associated with $E$, $K$, and $\phi$).
        \end{definition}

        \begin{theorem}
When $F=\mathbb{Q}$, the $p$-adic $L$-function in the sense of Iwasawa theory defined above, $\mathcal{L}_{E,K,\phi}^{IW}(-)$, and the anticyclotomic $p$-adic $L$-function $L_p(\mathbf{f})$ of Jetchev--Skinner--Wan \cite[Section~5]{jetchev5birch},  where $\mathbf{f}$ is the modular form corresponding to $E$, satisfy the following relation, up to a $p$-adic unit:
            $$
            \mathcal{L}_{E,K,\phi}^{IW}(-)=  L_p(\mathbf{f},-)
            $$
        \end{theorem}

        \begin{proof}
This follows from the interpolation comparison of Paragraph~\ref{subsubsection:interpolation}.
        \end{proof}

        \subsubsection{The  Liu--Zhang--Zhang formula at trivial character}
        \label{section:trivialcharacter}

The following theorem is the main result of this subsubsection. We call it the trivial-character Liu--Zhang--Zhang formula:

        \begin{theorem}
        \label{theorem:LZZoftrivialcharacter}
Notations and assumptions as above, up to a $p$-adic unit, the following equality holds:
            $$
            \mathcal{L}_{E,K,\phi}^{IW}(\mathbb{1})=\left(\frac{1+p-a_w}{p}\right)^2
            \log_{\omega}(P^0_{\mathbb{1}}(\phi))^2.
            $$
        \end{theorem}

        \begin{proof}
First, substituting the trivial character $\mathbb{1}$ into the definition of $\mathcal{L}_{E,K,\phi}^{IW}(-)$ and using Theorem~\ref{theorem:padicwaldspurgerforpadicmaass}, one obtains
            $$
            \mathcal{L}_{E,K,\phi}^{IW}(\mathbb{1})=(w_K^2 \cdot L(1,\eta)^2)\cdot
            \left( \frac{1}{L(1/2,\pi_{E,w}\otimes \mathbb{1}_w)^2} \right)
            \mathcal{P}_{\mathbb{C}_p}(\phi,\mathbb{1})^2.
            $$
Here we have used that the local $\epsilon$-factor $\epsilon(1/2,\psi,\pi_{E,w}\otimes\mathbb{1}_w)$ is a $p$-adic unit (indeed it is trivial, since $\pi_{E,w}$ is unramified and $\psi$ has level $0$). It can therefore be absorbed into the $p$-adic unit.

Substituting the definition of the local $L$-function, one obtains
            $$
            \mathcal{L}_{E,K,\phi}^{IW}(\mathbb{1})=(w_K^2 \cdot L(1,\eta)^2)\cdot
            \left( \frac{1+p-a_w}{p} \right)^2 \mathcal{P}_{\mathbb{C}_p}(\phi,\mathbb{1})^2.
            $$

Note that $F_w\cong \mathbb{Q}_p$ is used here.

By the conventions of Cai--Shu--Tian and Liu--Zhang--Zhang, their  measures differ by $2L(1,\eta)$ and  by Lemma 2.3 in \cite{cai2014explicit}, and since here $p>2$,up to a $p$-adic unit we have
            $$
            \mathcal{L}_{E,K,\phi}^{IW}(\mathbb{1})=(w_K^2 )\cdot \left(
            \frac{1+p-a_w}{p} \right)^2 \log_{\omega}(P^0_{\mathbb{1}}(\phi))^2.
            $$

By the assumptions of this article, the exponent of $v$ in the factorization of $p$ in the extension $K/\mathbb{Q}$ can only be $1$. If $(w_K,p)\neq1$, then $\mathbb{Q}(\zeta_p)\subseteq K$. And $p$ factorizes in $\mathbb{Q}(\zeta_p)$ as $(1-\zeta_p)^{p-1}$. This shows that if $p$ factorizes in $K$, the exponent at every prime ideal is necessarily a multiple of $p-1$. Since $p>2$ implies $p-1>1$, this contradicts the existence of $v$, and hence $(w_K,p)=1$. Thus $w_K$ can be deleted here.
        \end{proof}

        \paragraph{Remarks on Heegner points}
Different references have different conventions for Heegner points. This article has mentioned some Heegner points in Section~\ref{sec:GZ}, and also mentioned Heegner points in Subsubsection~\ref{section:lzzheegner} of this section. The exposition of these Heegner points follows the same line and almost follows the conventions of Yuan--Zhang--Zhang \cite{yuan2013gross}, but they differ by some constants coming from Haar measures. We write these differences in the Appendix.  We mainly adopt the conventions of Cai--Shu--Tian \cite{cai2014explicit} since the main theorem of Section~\ref{sec:GZ} comes from the formula of Cai--Shu--Tian.  

The comparison of our convention with the Heegner point described in Jetchev--Skinner--Wan is explained in the footnote on p.~426 of their article \cite{jetchev5birch}. In short, the Heegner point of Cai--Shu--Tian and the Heegner point $z_K$ described in Jetchev--Skinner--Wan differ only by a $p$-adic unit when $F=\mathbb{Q}$.

Therefore, in view of the exposition of Jetchev--Skinner--Wan on Heegner points, replacing it directly by $z_K$, the proof of the theorem is obtained.Theorem~\ref{theorem:LZZoftrivialcharacter} is exactly Proposition 5.1.7 of Jetchev--Skinner--Wan \cite{jetchev5birch} when $F=\mathbb{Q}$.

\section{A Birch and Swinnerton-Dyer Formula over A Totally Imaginary Quadratic Extension}\label{sec:BSDoverK}
This section uses the results of the previous sections to give a Birch and Swinnerton-Dyer formula over a totally imaginary quadratic extension.

The section is divided into two subsections. In the first Subsection~\ref{section:mainconjecture}, we state an Iwasawa main conjecture without proof, of course. And in Subsection~\ref{section:maintheorem} we prove the main theorem of this article.

\subsection{The Iwasawa Main Conjecture}\label{section:mainconjecture}

\subsubsection{Statement of the main conjecture}
Recall from Section~\ref{sec:selmer} that $X_{ac}^{\Sigma}(M)$ is a $\Lambda_L$-module. Write $X_{ac}(M)$ for the case $\Sigma=\emptyset$. Recall also from Section~\ref{sec:LZZ} that $\mathcal{L}_{E,K,\phi}^{IW}$ is the $p$-adic $L$-function in the Iwasawa-theoretic sense. Let $\mathcal{R}$ be the completion of the ring of integers of a finite extension of $F_w^{ur}$, and write $\Lambda_{\mathcal{R}}:=\Lambda\hat{\otimes}\mathcal{R}$.

\begin{conjecture}\label{conj:iwasawamain}
The $p$-adic $L$-function $\mathcal{L}_{E,K,\phi}^{IW}$ in the Iwasawa-theoretic sense can be viewed as an element of $\Lambda_{\mathcal{R}}$, and as ideals of $\Lambda_{\mathcal{R}}$ one has the equality
  $$
  \operatorname{char}_{\Lambda}(X_{ac}(M))\Lambda_{\mathcal{R}}=
  (\mathcal{L}_{E,K,\phi}^{IW}).
  $$
\end{conjecture}

\begin{remark}
When $F=\mathbb{Q}$ and under some additional conditions, Theorem 6.1.6 of Jetchev--Skinner--Wan \cite{jetchev5birch} proves the direction that the characteristic ideal is contained in the ideal generated by the $p$-adic $L$-function.
\end{remark}

\begin{remark}
In fact, in Jetchev--Skinner--Wan \cite{jetchev5birch}, proving one half of the main conjecture is sufficient for proving the Birch and Swinnerton-Dyer formula over $\mathbb{Q}$, because the method there proves the formula by comparing upper and lower bounds for the Shafarevich--Tate group, obtaining the result by a squeeze.
\end{remark}

\subsubsection{Application of the main conjecture}
In this subsubsection, based on the main conjecture~\ref{conj:iwasawamain}, we give a formula relating the special value of the $p$-adic $L$-function and the anticyclotomic Selmer group, as follows:

Here and below $C(W):=C^{\emptyset}(W)$ is the local correction factor of the control theorem~\ref{theorem:controltheorems}, i.e.\ the factor $C^{\Sigma}(W)$ in the case $\Sigma=\emptyset$.

\begin{proposition}\label{prop:applicationofIMC}
Assuming the main conjecture~\ref{conj:iwasawamain}, one has
  $$
  \#\mathcal{O}_{F_w}/(\mathcal{L}_{E,K,\phi}^{IW}(\mathbb{1}))
  =\#H^1_{ac}(K,W)\cdot C(W).
  $$
\end{proposition}

\begin{proof}
Let $f_{ac}(T)$ be a generator of $\operatorname{char}(X_{ac}(M))$. Then, by the main conjecture~\ref{conj:iwasawamain}, one has
  $$
  (f_{ac}(T))=(\mathcal{L}_{E,K,\phi}^{IW}).
  $$
Taking the special value on both sides, i.e.\ $T=0$, the $p$-adic $L$-function corresponds to taking the trivial character $\mathbb{1}$. This is equivalent to reducing both sides modulo the principal ideal $(T)$, so that one has the following equality of ideals in $\mathcal{O}_{F_w}$:
  $$
  (\mathcal{L}_{E,K,\phi}^{IW}(\mathbb{1}))=(f_{ac}(0)).
  $$
Next, considering their quotients in $\mathcal{O}_{F_w}$ and using the control theorem~\ref{theorem:controltheorems}, we obtain the proposition.
\end{proof}

\subsection{The Proof of the Main Theorem of This Article}\label{section:maintheorem}

\begin{proof}[Proof of The Main Theorem~\ref{theorem:maintheorem}]
First, using (Analytic rank one), (Reduction type), (Heegner hypothesis), and (Technical assumption--modularity-1, 2, 3), the Gross--Zagier formula of Theorem~\ref{Theorem:grosszaigerimaintheorem} gives, up to a $p$-adic unit,
  $$
  w_K^2\sqrt{|D_K|}\frac{L'(1,E/K)}{\Omega^{\mathrm{cong}}_{\mathbf{f}}}
  =\prod_{u|\mathfrak{D}} c_u(E/K)
  \cdot\left\langle P^0_{\mathbb{1}}(\phi),P^0_{\mathbb{1}}(\phi)\right\rangle _{NT}.
  $$

By (Prime and prime place), the exponent of $v$ in the decomposition of $p$ in the extension $K/\mathbb{Q}$ can only be $1$. If $(w_K,p)\neq 1$, then $\mathbb{Q}(\zeta_p)\subseteq K$. The ideal decomposition of $p$ in $\mathbb{Q}(\zeta_p)$ is $(1-\zeta_p)^{p-1}$, which shows that if $p$ decomposes into ideals in $K$, the exponent of every prime ideal must be a multiple of $p-1$. Since $p>2$ gives $p-1>1$, this contradicts the existence of $v$, and hence $(w_K,p)=1$. Also $\sqrt{|D_K|}$ is a $p$-adic unit. Therefore, deleting $w_K$, and $\sqrt{|D_K|}$, we obtain
  $$
  \frac{L'(1,E/K)}{\Omega^{\mathrm{cong}}_{\mathbf{f}}}
  =\prod_{u|\mathfrak{D}} c_u(E/K)
  \cdot\left\langle P^0_{\mathbb{1}}(\phi),P^0_{\mathbb{1}}(\phi)\right\rangle _{NT}.
  $$

For brevity, write $P=P^0_{\mathbb{1}}(\phi)$ for the Heegner point in the Gross--Zagier formula, and write
  $$
  m_P:=[E(K):\mathbb{Z}\cdot P].
  $$
Then, by definition, one has
  $$
  \left\langle P^0_{\mathbb{1}}(\phi),P^0_{\mathbb{1}}(\phi)\right\rangle _{NT}
  =m_P^2\cdot\mathrm{Reg}(E/K).
  $$
Substituting this into the Gross--Zagier formula, up to a $p$-adic unit one has
  $$
  \frac{L'(1,E/K)}{\Omega^{\mathrm{cong}}_{\mathbf{f}}
  \mathrm{Reg}(E/K)}
  =\prod_{u|\mathfrak{D}} c_u(E/K)\cdot m_P^2.
  $$
Since we only consider $p$-adic theory, up to a $p$-adic unit we have $m_P=[E(K)\otimes\mathbb{Z}_p:\mathbb{Z}_p\cdot P]$. At this point, applying the key formula of Proposition~\ref{prop:keyformula}, we obtain
  $$
  m_P^2=\frac{\#\Sha(E/K)[p^\infty]}{\#H^1_{ac}(K,W)}
  \cdot\left(
  \frac{\#\mathbb{Z}_p/(\frac{1-a_w+p}{p}\log_\omega P)}
  {\#H^0(K_v,W)}
  \right)^2.
  $$
Then, by the corollary of the Iwasawa main conjecture, substituting Proposition~\ref{prop:applicationofIMC} and replacing the anticyclotomic Selmer group by the $p$-adic $L$-function, we have
  $$
  m_P^2=\frac{\#\Sha(E/K)[p^\infty]}
  {\#\mathcal{O}_{F_w}/(\mathcal{L}_{E,K,\phi}^{IW}(\mathbb{1}))}
  \cdot\left(
  \frac{\#\mathbb{Z}_p/(\frac{1-a_w+p}{p}\log_\omega P)}
  {\#H^0(K_v,W)}
  \right)^2\cdot C(W).
  $$
By the Liu--Zhang--Zhang formula for the trivial character, Theorem~\ref{theorem:LZZoftrivialcharacter}, and the definition of $C(W)$ in the control theorem~\ref{theorem:controltheorems}, we obtain
  \[
  m_P^2=\frac{\#\Sha(E/K)[p^\infty]}
  {\#\mathcal{O}_{F_w}/(\frac{1-a_w+p}{p}\log_\omega P)^2}
  \cdot\left(
  \#\mathbb{Z}_p/\left(\frac{1-a_w+p}{p}\log_\omega P\right)
  \right)^2
  \cdot\prod_{u|\mathfrak{M}}c_u(E/K).
  \]
By reduction, we obtain
  $$
  m_P^2=\#\Sha(E/K)[p^\infty]\cdot\prod_{u|\mathfrak{M}}c_u(E/K).
  $$

Finally, substituting $m_P^2$ into the Gross--Zagier formula and noting that the Tamagawa numbers are nontrivial only at $\mathfrak{N}=\mathfrak{MD}$, we immediately obtain, up to a $p$-adic unit,
  $$
  \frac{L'(E/K,1)}{\Omega^{\mathrm{cong}}_\mathbf{f}
  \mathrm{Reg}(E/K)}
  =\#\Sha(E/K)[p^\infty]\prod_{u} c_u(E/K).
  $$
Thus a (variant) Birch and Swinnerton-Dyer formula over $K$ is proved.
\end{proof}

\newpage
\appendix

\section{Comparison of Conventions in Yuan--Zhang--Zhang, Cai--Shu--Tian, and Liu--Zhang--Zhang}\label{app:conventions}
\begin{center}
\setlength{\tabcolsep}{4pt}
\footnotesize
\begin{longtable}[c]{L{2.8cm}L{4.1cm}L{4.2cm}L{2.6cm}}
\caption{Comparison of Haar measures in different references}
\label{haartable}\\
\toprule
\multicolumn{4}{c}{\textbf{Comparison of Haar measures in different
references}} \\
\midrule
\endfirsthead
\endhead
\multicolumn{1}{c}{\textbf{Literature}} & \multicolumn{1}{c}{$K_u^\times/F_u^\times$} & \multicolumn{1}{c}{$\mathbb{A}_K^\times/K^\times$} & \textbf{Source} \\
\midrule
Yuan--Zhang--Zhang & not specified, with a condition on the product & product measure. Total volume $2L(1,\eta)$ & Book p.~8 and Erratum item 4 \\
Cai--Shu--Tian & not specified, but with a condition on the product & product measure. Total volume $2L(1,\eta)$ & the article, p.~2543 \\
Liu--Zhang--Zhang & $u\mid\infty$: $\mathrm{Vol}(K_u^\times/F_u^\times)=1$\newline
$u$ splits: $\mathrm{Vol}(F_u^\times)=1$\newline
$u$ inert: $\mathrm{Vol}(K_u^\times/F_u^\times)=1$\newline
$u$ ramified: $\mathrm{Vol}(K_u^\times/F_u^\times)=2$ & Tamagawa measure $\times$ $\frac{L(1,\eta)}{2^d\sqrt{|D_K|}}$ & the article, p.~752,\newline Definition 1.8.2 \\
\bottomrule
\end{longtable}
\end{center}

\begin{remark}
In the Yuan--Zhang--Zhang row, the Erratum reference comes from the website \cite{erratum-GZSC}.
\end{remark}

\begin{center}
\setlength{\tabcolsep}{3pt}
\footnotesize
\renewcommand{\arraystretch}{1.3}
\begin{longtable}{L{2.1cm}L{1.7cm}L{7.3cm}L{2.4cm}}
\caption{Comparison of Heegner points in different references}
\label{table:heegnertable}\\
\toprule
\multicolumn{4}{c}{\textbf{Comparison of Heegner points}} \\
\midrule
\endfirsthead
\toprule
\multicolumn{4}{c}{\textbf{Heegner points (continued)}} \\
\midrule
\endhead
\bottomrule
\endfoot

\multicolumn{1}{c}{\textbf{Literature}} & \multicolumn{1}{c}{\textbf{Category}} & \multicolumn{1}{c}{$\boldsymbol{P_{\chi}(f)}$} & \multicolumn{1}{c}{\textbf{Source}} \\
\midrule

\textbf{Yuan--Zhang--Zhang} & Sym\-bol & $\int_{\mathrm{Gal}(\bar{K}/K)} f(P)^t\chi(t)\,dt$ & \\
\textbf{Yuan--Zhang--Zhang} & Explan\-ation & $\mathrm{Gal}(\bar{K}/K)$ has volume $1$ & Book p.~7 \\

\textbf{Cai--Shu--Tian} & Sym\-bol & $\int_{{\mathbb{A}_K^\times/K^\times\mathbb{A}^\times}} f(P)^t\chi(t)\,dt$ & \\
\textbf{Cai--Shu--Tian} & Explan\-ation & $\mathbb{A}_K^\times/K^\times \mathbb{A}^\times$ has volume $2L(1,\eta)$ & the article, p.~2550 \\

\textbf{Liu--Zhang--Zhang} & Sym\-bol & $\int_{{\mathbb{A}_K^{\infty\times}/\overline{K^\times}}} f(T_tP)\chi(t)\,dt$ & \\
\textbf{Liu--Zhang--Zhang} & Explan\-ation & $\mathbb{A}_K^{\infty\times}/\overline{K^\times}$ has volume $2$ & the article, p.~792, Definition 3.3.1 \\

\end{longtable}
\end{center}

\bibliographystyle{alpha}
\bibliography{ref}

\end{document}